\documentclass[mnsc,nonblindrev]{informs3update}
\usepackage[export]{adjustbox}
\usepackage{amsmath, amsfonts, amssymb}
\usepackage{mathrsfs}
\usepackage{authblk, color, bm, graphicx}
\usepackage{array,multirow,graphicx}
\usepackage{subfig}
\usepackage{placeins}
\usepackage{float}
\usepackage{dsfont}
\usepackage{epstopdf}
\usepackage{epsfig}
\usepackage{cases}
\usepackage{hyperref}
\usepackage{cancel}
\usepackage{soul}
\usepackage{algorithm}
\usepackage{makecell}
\usepackage{mathtools}
\usepackage{enumitem}
\usepackage{blindtext}
\usepackage{subfig}
\usepackage[toc,page]{appendix}
\usepackage[noend]{algpseudocode}
\usepackage{comment}
\usepackage{hyperref}

\usepackage{natbib}
\bibpunct[, ]{(}{)}{,}{a}{}{,}%
\def\bibfont{\small}%

\OneAndAHalfSpacedXI
\TheoremsNumberedThrough
\ECRepeatTheorems
\EquationsNumberedThrough
\MANUSCRIPTNO{}

\renewcommand{\qed}{\hfill $\square$}

\newcommand{\bmh}[1]{\hat{\bm{#1}}}
\newcommand{\bmt}[1]{\tilde{\bm{#1}}}

\newcommand{\ep}[1]{\mathbb{E}_{\mathbb{P}}\left[#1\right]}
\newcommand{\eq}[1]{\mathbb{E}_{\mathbb{Q}}\left[#1\right]}

\newcommand{\epstar}[1]{\mathbb{E}_{\mathbb{P}^{\star}}\left[#1\right]}
\newcommand{\epdag}[1]{\mathbb{E}_{\mathbb{P}^\dag}\left[#1\right]}

\newcommand{\pp}[1]{\mathbb{P}\left[#1\right]}

\newcommand{\pphat}[1]{\hat{\mathbb{P}}\left[#1\right]}

\newcommand{\cx}{\mathcal{X}}
\newcommand{\cz}{\mathcal{Z}}

\newcommand{\cf}{\mathcal{F}}

\newcommand{\cq}{\mathcal{Q}}
\newcommand{\cp}{\mathcal{P}}

\newcommand{\ck}{\mathcal{K}}
\newcommand{\bbq}{\mathbb{Q}}
\newcommand{\bbp}{\mathbb{P}}
\newcommand{\bbr}{\mathbb{R}}
\newcommand{\bbphat}{\hat{\mathbb{P}}}

\newcommand{\ephat}[1]{\mathbb{E}_{\hat{\mathbb{P}}}\left[#1\right]}

\newcounter{nappendix}% Counter
\begin{document}
\RUNTITLE{Robust Conic Satisficing}
\TITLE{Robust Conic Satisficing}

\ARTICLEAUTHORS{%
	\AUTHOR{Arjun Ramachandra}
	%\AFF{Singapore University of Technology and Design\\ \EMAIL{\href{mailto:arjun_ramachandra@mymail.sutd.edu.sg}{arjun\_ramachandra@mymail.sutd.edu.sg}}}
	\AFF{Engineering Systems \& Design, Singapore University of Technology and Design\\
	%Department of Industrial Systems Engineering \& Management, National University of Singapore \\
	\EMAIL{\href{mailto:arjun.ramachandra@gmail.com}{arjun.ramachandra@gmail.com}}}
	\AUTHOR{Napat Rujeerapaiboon}
	\AFF{Department of Industrial Systems Engineering \& Management, National University of Singapore \\ \EMAIL{\href{mailto:napat.rujeerapaiboon@nus.edu.sg}{napat.rujeerapaiboon@nus.edu.sg}}}
	\AUTHOR{Melvyn Sim}
	\AFF{Department of Analytics \& Operations, National University of Singapore\\ \EMAIL{\href{mailto:melvynsim@gmail.com}{melvynsim@gmail.com}}}
}

\ABSTRACT{
%Inspired by the principle of satisficing \citep{Simon_1955}, 
\indent In practical optimization problems, we typically model uncertainty as a random variable though its true probability distribution is unobservable to the decision maker. Historical data provides some information of this distribution that we can use to approximately quantify the risk of an evaluation function that depends on both our decision and the uncertainty. This empirical optimization approach is vulnerable to the issues of overfitting, which could be overcome by several data-driven robust optimization techniques. To tackle overfitting, \cite{long2022robust} propose a {\em robust satisficing} model, which is specified by a performance target and a penalty function that measures the deviation of the uncertainty from its nominal value, and it yields solutions with superior out-of-sample performance. We generalize the robust satisficing framework to conic optimization problems with recourse, which has broader applications in predictive and prescriptive analytics including risk management, statistical supervised learning, among others. We derive an exact semidefinite optimization formulation for biconvex quadratic evaluation function, with quadratic penalty and ellipsoidal support set. More importantly, under complete and bounded recourse, and reasonably chosen polyhedral support set and penalty function, we propose safe approximations that do not lead to infeasible problems for any reasonably chosen target. We demonstrate that the assumption of complete and bounded recourse is however not unimpeachable, and then introduce a novel {\em perspective casting} technique to derive an equivalent conic optimization problem that satisfies the stated assumptions.  
% model convex evaluation function as  equivalent conic optimization problem with complete and bounded recourse over a {\em featured  domain}.
Finally, we showcase a computational study on data-driven portfolio optimization that maximizes the expected exponential utility of the portfolio returns. We demonstrate that the solutions obtained via robust satisficing can significantly improve over the solutions obtained by stochastic optimization models, including the celebrated Markowitz model, which is prone to overfitting. 
}
\KEYWORDS{robust optimization, robust satisficing, conic optimization, affine recourse adaptation}

\maketitle

%%%%%%%%%%%%%%%%%%%%%%%%%%%%%%%%%%%%%%%%%%%%%%%%%%%%%%%%%%%%%%%%%%%%%%%%%%%%%%%%%%%%%%%%%%%%%%%%%%%%%%%%%%%%%%%%%%%%%%%%%%%%

%%%%%%%%%%%%%%%%%%%%%%%%%%%%%%%%%%%%%%%%%%%%%%%%%%%%%%%%%%%%%%%%%%%%%%%%%%%%%%%%%%%%%%%%%%%%%%%%%%%%%%%%%%%%%%%%%%%%%%%%%%%%
\section{Introduction}\label{sec:introduction}
\vspace{2pt}
\begin{center}
%``Uncertainty is the only certainty there is'' - John Allen Paulos\\ 
 %\begin{quotation}\sanskrit
%\centering नासतो विद्यते भावो नाभावो विद्यते सतः ।
 %\end{quotation}
\centering ``Of the impermanent there is no certainty'' - Gita (2.16)
 \end{center}
\vspace{5pt}

Data is an integral part of predictive and prescriptive analytics that we would use for evaluating approximately the potential risks of decisions that might arise in future.  Although historical data provides limited information of the model's underlying  uncertainty, the true probability distribution of the random variable remains unobservable to the decision maker. In stochastic optimization, we evaluate the model's risk-based objective function using an explicit probability distribution with parameters estimated from historical data \citep{shapiro2014lectures,birge2011introduction}. In particular, an empirical optimization problem is a stochastic optimization problem when we use the empirical distribution for risk evaluation. However, as articulated in  \cite{smith2006optimizercurse} {\em Optimizer's curse}, stochastic optimization models are prone to overfitting and would likely yield inferior solutions when the risks are evaluated on the true distributions. This is exemplified in the study of \cite{DeMiguelGarlappiUppal2009} that demonstrates the inferiority of optimized portfolios when evaluated on the actual data.    

To tackle overfitting, \cite{esfahani2018data} propose the seminal framework for data-driven robust optimization, by incorporating an ambiguity set of probability distributions that are proximal to the empirical distribution with respect to the Wasserstein metric. The approach alleviates the optimizer's curse by avoiding overfitting to the empirical distribution. When appropriately sized, the ambiguity set captures the true distribution with high level of confidence, and the robust solutions could yield superior out-of-sample performance compared to solutions obtained from empirical optimization models.  \cite{shafieezadeh2019regularization} extend the models to solving supervised learning problems and establish the relation with regularization for Lipschitz continuous loss functions and unbounded support for the underlying random variable. As in the case of regularization, the size of the ambiguity set is a hyper-parameter that can be determined by cross-validation techniques. 

\cite{long2022robust} recently proposed an alternative approach known as robust satisficing for improving the solutions of data-driven optimization problems . For one data sample, this model corresponds to a case of the GRC-sum model of  \cite{Ben-talEtAl2017}. Instead of sizing the ambiguity set, the satisficing model delivers solutions that aim to achieve the specified risk-based target to a feasible extent, notwithstanding the potential mismatch between the true distribution and the empirical distribution. \cite{sim2021tractable} extend the robust satisficing approach to solving supervised learning problems and show in studies on well-known data sets that the robust satisficing approach would generally outperform the robust optimization approach, when the respective target and size parameters are determined via cross-validation. In this paper, we generalize the robust satisficing framework to conic optimization problems with recourse, which has broader applications in predictive and prescriptive analytics including, {\em inter alia}, risk management and statistical supervised learning.

%Unlike robust optimization methods, which only hedge against pre-defined uncertainty, robust satisficing offers full protection by giving nature a free-hand to vary over the entire uncertainty support. The model compensates for this increased protection by allowing the constraints to be violated while simultaneously controlling the degree of infeasibility.  
%The decision maker has the flexibility to choose the degree of sub-optimality relative to the emperical optimization nominal objective value, by specifying a target, unlike the robust optimization model,  where the size of the uncertainty set needs to be known apriori. 
%\cite{Simon_1955}, who proposes the term {\em satisficing}, argues that target plays an important role in decision making, especially in complex situations involving uncertainty. 

%Robust satisficing  has also been used in \cite{Schwartz_BenHaim_Dacso_2011makes} to describe a decision model that maximizes the robustness to uncertainty of achieving a satisfactory target. The decision criterion in our robust satisficing framework belongs to the family of satisficing decision criteria axiomatized by \cite{brownsim2009satisficing}, which has an embedded preference for diversification that, serendipitously, also leads to computational tractability when used in convex optimization problems.

As in the case of robust optimization models, it can be computationally challenging to solve robust satisficing models exactly. However, since the inception of robust optimization, we now have an arsenal of tools to address and solve either exactly, or by providing tractable safe approximations for various kinds of robust optimization models \citep[see, for instance, ][among others]{Bental1998,Bertsimas_Sim_2004,bental2004adjustable,kuhn2011primal_ldr,bertsimas2010optimality_ldr,delage2010distributionally,wiesemann2014distributionally,bental2015nonlinear_concave,bertsimas2019adaptive,chen2020rsome}.  We wish to highlight 
\cite{Bertsimas2016} for proposing the dualizing technique and applying affine dual recourse adaptation to address an adaptive robust linear optimization problem. When the uncertainty set is polyhedral, this approach can also be used to provide safe approximations for robust optimization models with  biconvex constraint functions including those with recourse adaptation \citep{trevor2018dual_twostage,Roos2020}. We use this approach to obtain solutions to our proposed robust conic satisficing models. We also highlight the key challenge that in providing safe approximations to the robust satisficing models; it is {\em sine qua non} to ensure that the approximations do not lead to infeasible problems if the chosen target is above the optimum objective when the empirical optimization problem is minimized. This necessitates careful analysis and derivation of conditions for which 
safe approximations can be provided for a broad class of robust satisficing models. We summarize the contributions of this paper below:

\begin{enumerate}
    \item  We provide a unifying robust conic satisficing framework to tackle overfitting in data-driven optimization that can be applied to predictive and prescriptive analytics. We solve the robust conic satisficing model either exactly when possible, or provide a {\em tractable safe approximation} that ensures its feasibility for any chosen target  that is achievable by the empirical optimization problem.
    \item For a biconvex quadratic evaluation function with a quadratic penalty and an ellipsoidal support set, we derive an exact semidefinite optimization formulation of the robust satisficing problem.  For a more general conic representable evaluation function, we derive a tractable safe approximation using {\em affine dual recourse adaptation} under assumptions of complete and bounded recourse along with a polyhedral support set and polyhedral penalty function.  
    \item  We show that the perspectification approach to model nonlinear convex function as a conic optimization problem does not necessarily have the property of complete and bounded recourse. Hence, we introduce a novel {\em perspective casting} approach to model such an evaluation function as an equivalent conic optimization problem with complete and bounded recourse over a {\em featured domain.} 
    \item For linear optimization with recourse, we show that the tractable safe approximation via affine dual recourse adaptation dominates the primal approach in terms of the approximation quality and computational performance. 
    \item
 We showcase the performance benefits of the robust satisficing framework with a portfolio selection problem that maximizes expected exponential utility compared to the solutions generated from empirical optimization and stochastic optimization models, which are prone to overfitting. 
\end{enumerate}
\textbf{Notation:} We use $\mathbb{R}$  to denote the space of reals while $\mathbb{R}_{+}$ and $\mathbb{R}_{++}$ denote the sets of non-negative and strictly positive real numbers, respectively. We use boldface small-case letters (\textit{e.g.} $\bm{x}$) to denote vectors, capitals (\textit{e.g.} $\bm{A}$) to denote matrices and capital caligraphic letters to denote sets (\textit{e.g.} $\mathcal{Z}$) including cones (\textit{e.g.} $\mathcal{K}$). $\mathcal{R}^{m,n}$ and $\mathcal{L}^{m,n}$ are used to denote the set of all functions and its sub-class of affine functions, respectively, from $\mathbb{R}^m$ to $\mathbb{R}^n$. The transpose of a vector (matrix) is denoted by $\bm{x}^\top\;(\bm{A}^\top$). 
We use $[n]$  to denote the running index set $\{1,2,...,n\}$, and the superscript indexing (\textit{e.g.} $\bm{w}^i$) is used to  denote the $i^{th}$ vector (matrix) among a countable set of vectors $\{\bm{w}^i\}$ (matrices) and subscript indexing (\textit{e.g.} $\bm{A}_i$) to denote the $i^{th}$ row of a matrix $\bm{A}$. A random vector is denoted with a tilde sign (\textit{e.g.} $\tilde{\bm{z}}$), and its probability distribution is denoted by $\bmt z \sim \bbp$, $\bbp \in \cp_0(\cz)$,  where $\cp_0(\cz)$ represents the set of all possible distributions with support $\cz \subseteq \bbr^n$.  For a composite random variable, $f(\bmt z)$, $f:\bbr^n\rightarrow \bbr$, we use $\ep{f(\bmt z)}$ to denote expectation of the composite random variable over the distribution $\bmt z\sim \bbp$.
%The only exception to the subscript rule is that the $n$ dimensional vector of all ones (zeros) is denoted by $\bm{1}\;(\bm{0}_{n})$, while the identity matrix of order $n$ is denoted by $\bm{I}_{n}$.
Finally, $\bm{0}~(\bm{1})$ denotes the vector of all zeros (ones) and its dimension should be clear from the context, while the identity matrix of order $n$ is denoted by $\bm{I}_{n}$.

 %and $n \times n$ matrix of zeros is denoted by $(\bm{0}_{n \times n})$.
%%%%%%%%%%%%%%%%%%%%%%%%%%%%%%%%%%%%%%%%%%%%%%%%%%%%%%%%%%%%%%%%%%%%%%%%%%%%%%%%%%%%%%%%%%%%%%%%%%%%%%%%%%%%%%%%%%%%%%%%%%%%

\section{Mitigating ambiguity in data-driven optimization}\label{sec:datadriven}
We consider a data-driven risk-based minimization problem explored in \cite{esfahani2018data}, where we denote by $\bm x \in \mathcal{X} \subseteq \bbr ^{n_x}$ a vector of decision variables and $g(\bm x, \bm z)$ is  the {\em evaluation function} associated with a vector of input parameters $\bm z \in \mathcal{Z} \subseteq \bbr^{n_z}$ that is subject to uncertainty. We denote $\bmt z$ as the corresponding random variable with an underlying probability distribution $\bbp^{\star} \in \mathcal P_0(\cz)$, $\bmt z \sim \bbp^{\star}$. Hence, for a given decision $\bm x \in \cx$, the evaluation function $g(\bm x,\bmt z)$, $\bmt z \sim \bbp^{\star}$, is a composite random variable. Ideally, we would like to determine the optimal solution to the following optimization problem, 
\begin{equation}
    \label{eq:idealopt} 
Z^{\star}= \min_{\bm x \in \mathcal X} \epstar{g(\bm x,\bmt z)},
\end{equation}
where the objective function is the expectation of the evaluation function. However, the underlying distribution $\bbp^{\star}$ is not observable. Instead, we have access to $\Omega$ samples and for each $\omega \in [\Omega]$, $\hat{\bm z}^\omega$  denotes an independent realization of the random variable $\bmt z$, $\bmt z \sim \bbp^{\star}$. In situations of limited data availability, it is well possible to have $\Omega =1$. In which case, using the terminology in robust optimization,  $\bmh z^1$ would represent the nominal values of the input parameters. We also denote the empirical distribution by $\hat{\mathbb P} \in \mathcal P_0(\mathcal Z)$, $\bmt z \sim \hat{\mathbb P}$, such that $\pphat{\bmt z =\hat{\bm z}^\omega } = 1/\Omega$, $\omega \in [\Omega]$. 

To set our work apart from \cite{esfahani2018data}, we focus on an evaluation function that is non-piecewise-linear and bi-convex in $(\bm x,\bm z)$.  As we will show, our framework is a generalization and could address a broader class of problems such as those arising in risk management, for which the evaluation function is associated with a risk-adverse decision maker that can be expressed as 
$$
g(\bm x,\bm z) = \ell(f(\bm x,\bm z)), 
$$
where $f(\bm x,\bm z)$ is the underlying cost function and $\ell:\bbr\rightarrow \bbr$ is a convex increasing disutility function.  For instance, the exponential disutility, which is far more commonly adopted than other types of disutility functions combined \citep{corner1995characteristics}, has the following representation
$$
\ell_a(w) = \left\{\begin{array}{ll}   (\exp(aw)-1)/a & \mbox{if $a>0$,}\\
w& \mbox{if $a=0$}
\end{array}\right.
$$ 
where $a$ is known as the Arrow-Pratt measure of risk-aversion.  Note that the disutility function is normalized  such that $\ell_a(0)=0$ and $1 \in \partial \ell_a(0)$.

The same framework can also be applied to solving linear regression and classification problem \citep[see, for instance,][]{shafieezadeh2019regularization}, where we predict the scalar response variable $\tilde{z}_{n+1}$ using an affine function of $n$ explanatory variables  $\tilde{z}_1, \dots, \tilde z_{n}$. The evaluation function relates to the prediction loss, {\em i.e.}, 
$$
g(\bm x,\bm z) = \ell \left( \left(x_{n+1}+\sum_{i\in [n]} x_i z_i\right) - z_{n+1} \right), 
$$
where $\ell:\bbr \rightarrow  \bbr_+$ is a non-piecewise-linear and convex loss function such as
\begin{enumerate}
\item Quadratic loss: $\ell(w)=w^2$.
\item Huber loss with parameter $\delta>0$: $\ell (w)=\frac{w^2}{2}$ if $|w|\leq \delta$; otherwise, $\ell(w)=\delta(|w|-\frac{\delta}{2})$.
\item Squared hinge loss: $\ell(w)=(\max\{w,0\})^2$.
\item Logexp loss: $\ell(w)=\log(1+\exp(w))$.
\end{enumerate}

\subsection{Empirical and stochastic optimization}

Since $\mathbb P^{\star}$ is unobservable, a reasonable approach is to solve the {\em empirical optimization} model  as follows
\begin{equation}
    \label{eq:empopt} 
Z_0= \min_{\bm x \in \mathcal X} \ephat{g(\bm x,\bmt z)} = \min_{\bm x \in \mathcal X} \frac{1}{\Omega} \sum_{ \omega \in [\Omega]} g(\bm x,\bmh z^\omega).
\end{equation}

\begin{assumption} \label{assm:existence}
We assume that the empirical optimization problem \eqref{eq:idealopt} is solvable, {\em i.e.}, there exists $\bmh x \in \cx$ such that $$
Z_0=  \frac{1}{\Omega}\sum_{ \omega \in [\Omega]} g(\bmh x,\bmh z^\omega).$$
\end{assumption}

More generally, in stochastic optimization, we would assume a distribution $\bbp^\dag \in \cp_0(\bbr^n_z)$, $\bmt z \sim \bbp^\dag$ in which the parameters to characterize the distribution $\bbp^\dag$ is determined from the empirical distribution, $\bbphat$. For instance, it is common to assume multivariate normal distribution, {\em i.e.},  $\bbp^\dag = \mathcal N( \bmh \mu, \bmh \Sigma )$ where  $\bmh \mu = \ephat{\bmt z}$ and $\bmh \Sigma = \ephat{(\bmt z- \bmh \mu)(\bmt z- \bmh \mu)^\top}$. Subsequently, stochastic optimization solves the following optimization problem, 
\begin{equation}
    \label{eq:stocopt} 
Z^\dag= \min_{\bm x \in \mathcal X} \epdag{g(\bm x,\bmt z)},
\end{equation}
which has an objective function that is much harder to evaluate compare to the empirical optimization problem. Nevertheless, {\em sample average approximation} (SAA) can be  used to approximate the objective function of Problem~\eqref{eq:stocopt} by solving the following optimization problem 
\begin{equation}
    \label{eq:SAAopt} 
Z^{\dag,S}= \min_{\bm x \in \mathcal X} \frac{1}{S} \sum_{ s \in [S]} {g(\bm x,{\bm z}^{\dag,s})},
\end{equation}
where ${\bm z}^{\dag,s}$, $s \in [S]$ are $S$ samples independently generated from the distribution $\bbp^\dag$. We stress that SAA differs from empirical optimization because in the latter, we only have access to the available empirical samples, but not the underlying probability distribution that would enable us to generate an arbitrary number of independent samples. Moreover, it should be noted that with sufficient computational resources, SAA could at best obtain the optimum solution to Problem~\eqref{eq:stocopt}, which is not the same as the actual optimum solution of Problem~\eqref{eq:idealopt}. For instance, if $g(\bm x,\bm z)$ is affine in the uncertain parameters, then
$$
\epdag{g(\bm x,\bmt z)} = g(\bm x,\ephat{\bmt z}|)   =  \ephat{g(\bm x,\bmt z)},
$$
and, in this case, solving the empirical optimization problem would yield the same solution as the exact stochastic optimization problem.

\cite{smith2006optimizercurse} observe that if we solve a stochastic optimization problem using parameters estimated from the empirical distribution, due to the {\em Optimizer's curse}, we should also expect inferior out-of-sample results. In the context of regression and classification, solving the the empirical optimization problem is similar to solving the problem without regularization, which can lead to poor out-of-sample performance due to overfitting. 

\subsection{Data-driven robust optimization}
To mitigate the {\em Optimizer's curse}, \cite{esfahani2018data} propose the seminal data driven robust optimization model, which solves the following optimization problem
\begin{equation}
    \label{eq:datarobustopt} 
Z_r= \min_{\bm x \in \mathcal X} \sup_{\bbp \in \cf(r)}\ep{g(\bm x,\bmt z)},
\end{equation}
for a given ambiguity set that considers a family of probability distributions within the vicinity of $\bbphat$ given by
$$
\mathcal F(r) = \left\{ \bbp \in \cp_0(\cz) \left| \begin{array}{ll}
\bmt z \sim \bbp \\
\Delta_q(\bbp,\bbphat) \leq r
\end{array} \right. \right\},
$$
where 
\begin{equation*}
\Delta_q(\bbp,\bbphat):= \inf_{\bbq \in \cp_0(\cz^2)} \left\{ \left(\eq {\|\bmt{z}-\bmt{v}\|^q}\right)^{1/q}~\left|~ (\bmt z,\bmt v)\sim \bbq, \bmt{z} \sim \bbp, \bmt{v} \sim \bbphat\
\right. \right\},
\end{equation*}
is the type-$q$ Wasserstein distance metric, for $q \geq 1$. The consideration of the  Wasserstein distance metric is inspired by the following statistical property of how likely the unobservable true distribution $\bbp^{\star}$ is within the proximity of $\bbphat$ as follows:
\begin{theorem}{\bf \citep[Theorem 2 of][]{fournier2015rate}.}
\label{thm:fournier2015rate}
Let $\bbp^\Omega$ denote  the distribution that governs the distribution of the independent samples $\bmt{z}^1,\ldots,\bmt{z}^\Omega$ drawn from $\bbp^{\star}$ for which the empirical distribution $\bbphat$ is constructed. When the true data-generating distribution $\bbp^{\star}$, $\bmt z \sim \bbp^{\star}$ is a light-tailed distribution such that 
$\mathbb{E}_{\bbp^{\star}}[\exp(\|\bmt{z}\|^\alpha )] < \infty$ for some $\alpha>1$, then for all $r \geq 0$,
    \begin{equation*}
    \bbp^{\Omega}\left[\Delta_q(\bbp^{\star},\bbphat) > r\right] \leq \varepsilon_q(r)
\end{equation*}
for some function $\varepsilon_q(r):\bbr_+\rightarrow \bbr_+$ that decreases to zero at an exponential rate in $r>0$. 
\end{theorem}

\cite{esfahani2018data} derive an equivalent minimax reformulation of Problem~\eqref{eq:datarobustopt} when $q = 1$. Such a reformulation can be readily extended to other $q \geq 1$, \emph {i.e.,}
\begin{equation}
\label{eq:datarobustopt2}
    \begin{array}{rcll}
         Z_r = &  \min & \displaystyle \frac{1}{\Omega}\sum_{\omega\in[\Omega]}  \sup_{\bm{z}\in\cz} \{ g(\bm{x},\bm{z}) - k (\|\bm{z} - \bmh{z}^\omega\|^q-r^q)\}
         \\
         & {\rm s.t.} &  \bm{x}\in \cx,~ k\geq 0.
    \end{array}
\end{equation}

\cite{shafieezadeh2019regularization} have applied the  data-driven robust optimization framework to improve the solutions of linear regression and classification problems, where they have established the connections with regularization if the loss function is Lipschitz continuous and having an unbounded support $\cz = \bbr^{n_z}$. More recently, \cite{sim2021tractable} extend the results to loss functions derived from Lipschitz continuous functions raised to the power of $q, q\geq 1$. They propose tractable type-$q$ robust optimization and satisficing models and using  type-$q$ Wasserstein distance, and show their relations with $q$-root regularization problems when the support is unbounded. However, apart from these special cases, when the support is bounded and the underlying evaluation function is nonlinear, it remains computational challenging to solve the data-driven robust optimization problem exactly. 

\subsection{On classical robust optimization}
It is interesting to note that for the extreme case of $\Omega=1$, the data-driven robust optimization~problem becomes
\begin{equation}
\label{eq:datarobustopt3}
\begin{array}{rcll}
         Z_r = &  \min & \displaystyle  \sup_{\bm{z}\in\cz} \{ g(\bm{x},\bm{z}) - k (\|\bm{z} - \bmh{z}^1\|^q-r^q)\}
         \\
         & {\rm s.t.} &  \bm{x}\in \cx,~ k\geq 0,
    \end{array}
\end{equation}
which has semblance to the classical stochastic-free robust optimization problem, 
\begin{equation}
\label{eq:robustopt1}
\begin{array}{rcll}
         \underline {Z}_r & = & \displaystyle  \min_{\bm{x}\in \cx}  \sup_{\bm{z}\in \mathcal U_r} \{ g(\bm{x},\bm{z})\} = \min_{\bm{x}\in \cx}  \sup_{\bm{z}\in \cz}  \min_{k \geq 0} \{ g(\bm{x},\bm{z}) )  -k (\|\bm{z} - \bmh{z}^1\|^q-r^q)\}
    \end{array}
\end{equation}
where $\bmh{z}^1$ is the nominal value and 
$\mathcal U_r = \left\{ \bm z \in \cz | \|\bm{z} - \bmh{z}^1\|\leq r)\right\}$ is the norm-based uncertainty set. Hence, by weak duality $\underline {Z}_r \leq Z_r$, and by strong duality the two coincide when $g(\bm{x},\bm{z})$ is also concave in $\bm z$. 

To illustrate the distinction between the data-driven robust optimization with $\Omega = 1$ sample and classical robust optimization, we consider an uncertainty quantification problem with a univariate $z$. Supposing that $q = 2$, $g(z) = (z-1)^2$ and $\mathcal{Z} = \mathbb{R}_+$, we have
\begin{equation*}
    \underline{Z}_1 = \max_{z \in [0,1]} \left\{ (z-1)^2 \right\} = 1,
\end{equation*}
whereas
\begin{equation*}
\begin{aligned}
    Z_1 &=\ \min_{k \geq 0} \max_{z \geq 0} \left\{ (z-1)^2 - k(z^2 - 1) \right\} \\
    &=\ \min_{k \geq 0,t} \left\{ t \ \vert \ t \geq (z-1)^2 - k(z^2-1) ~ \forall z \geq 0 \right\} \\
    &=\ \min_{k \geq 0,t} \left\{ t \ \vert \ (k-1)z^2 + 2z + (t -1-k) \geq 0 ~ \forall z \geq 0 \right\}.
\end{aligned}
\end{equation*}
For this optimization problem characterizing $Z_1$, it necessarily holds that $k - 1$ and $t-1-k$ are both non-negative, and thus $Z_1 \geq 2 > \underline{Z}_1$.

 While classical robust optimization models are well justified in engineering problems such as trust topology design \cite{Bental1998}, it may not necessarily be the best approach in mitigating risks in management inspired optimization problems. In the context of risk management, where $g(\bm x,\bm z) = \ell(f(\bm x,\bm z))$, we note that since $\ell$ is an increasing function, we have 
$$
 \underline {Z}_r  =  \displaystyle  \ell\left( \min_{\bm{x}\in \cx}  \sup_{\bm{z}\in \mathcal U_r} \{ f(\bm{x},\bm{z})) \}\right),
$$
which, unlike Problem~\eqref{eq:datarobustopt3}, the optimal robust solution of Problem~\eqref{eq:robustopt1} would not be dependent on the disutility function at all. 
 {\em We emphasize that, as far as we know, the statistical property of the robust solution has only been established for when the evaluation  function is affine  \citep{Bertsimas_Sim_2004,ben2000robust} or concave \citep{bertsimas2018data} in its uncertain parameters. }
For instance, under the assumption that the random variable $\bmt z$, $\bmt z \sim \bbp$, $\bbp \in \mathcal P_0(\cz)$, with $\cz = [-1,1]^{n_z}$, $\ep{\bmt z }=\bm 0$ its components are independently distributed, \cite{Bertsimas_Sim_2004} show that if $g(\bm x, \bm z)$ is affine in $\bm z$, by choosing $r=c\sqrt{n_z}$, then
$$
\pp{g(\bm x, \bmt z) \leq \max_{\bm{z}\in \mathcal U_r} g(\bm x, \bm z) } > 
1-\exp(-c^2/2).
$$ However, it is important to note that the bound can be rather weak and may not apply if the evaluation function $g(\bm x, \bm z)$ is not affine  in the uncertain parameters. In the extreme case, if the function is $g(\bm x, \bm z) = \|\bm z\|_1$, then we would have 
$$\pp{g(\bm x, \bmt z) \leq \max_{\bm{z} \in \mathcal U_r} g(\bm x, \bm z)} =
\pp{\| \bmt z\|_1 \leq \max_{\bm{z}\in \cz:\|\bm z\|_1 \leq r} \;\|\bm z\|_1}= 
\pp{\bmt z \in \mathcal U_r},$$ which is the minimal probability guarantee assured by the robust optimization solution. However, such assurance may not be adequate for $r = c\sqrt{n_z}$ as demonstrated in the following result. 
\begin{proposition}
	\label{prop:RObound}
	For each $i\in[n_z]$, let $\tilde{z}_i$ be an independently distributed random variable with support $[-1,1]$, zero mean and variance $\theta >0$. For a given budgeted uncertainty set $\mathcal U_r$, $r=c\sqrt{n_z}$ and $c\in (0, \theta \sqrt{n_z})$, then 
	\begin{equation}
	\label{eqn:RObound_sqrtN}
		\displaystyle \bbp\left [\bmt{z}\in\cal{U}_r \right]\leq \exp\left(-\frac{\left(\theta \sqrt{ n_z} - c\right)^2}{2\theta}\right).
	\end{equation}	
\end{proposition}
\begin{proof}{Proof.}
The proof is relegated to Appendix \ref{appendix:proof_of_results}.
\end{proof}
Hence, the probability that the realization of uncertain parameter $\bmt{z}$ lies in the uncertainty set $\cal{U}_r$, $r=c\sqrt{n_z}$ decreases exponentially to zero as the dimension $n_z$ increases. 

Therefore, one should exercise  caution in using the budgeted uncertainty set beyond evaluation  functions that are affine in the uncertain parameters because its probability bound can become ineffective as $n_z$ increases. In contrast, the finite sample probabilistic guarantee that motivates the data-driven robust optimization model does not depend on the nature of the evaluation function.   

\

\subsection{Data-driven robust satisficing}
A target based data driven robust satsificing model has recently been proposed by \cite{long2022robust}, which we also generalize to type-$q$ Wasserstein distance metric. 
For a given target $\tau \geq Z_0$, we solve the following optimization problem
\begin{equation}
\label{eq:data-drivenRS}
\begin{aligned}
\; &\text{minimize} && k\\
    &\text{subject to} &&
    \ep{g(\bm x,\bmt z)}  \leq \tau + k(\Delta_q(\bbp,\bbphat))^q &\forall \bbp \in \mathcal P_0(\mathcal Z) \\
    & && \bm{x} \in \mathcal{X}, k\geq 0, 
\end{aligned}  
\end{equation}

As a consequence of Theorem~\ref{thm:fournier2015rate}, the feasible solution of the robust satisficing problem would satisfy

\begin{equation*}
    \bbp^\Omega\left[\mathbb{E}_{\bbp^{\star}}\left[
     g(\bm x,\bmt z) \right] > \tau  + k r^q\right] 
    \leq \bbp^\Omega\left[
    \Delta_q(\bbp^\star,\bbphat) >  r\right] \leq 
    \varepsilon_q(r),
\end{equation*}
for all $r \geq 0$, justifying the lower value of $k$, the higher the level of confidence that the actual objective function evaluated on the true distribution is within $kr^q + \tau$, such that the probability of exceeding the confidence range decreases to zero at an exponential rate in $r$, accordingly to the measure concentration result of \cite{fournier2015rate}; see~\emph{e.g.} Theorem~\ref{thm:fournier2015rate}.    

Observe that 
$$
\begin{array}{rcll}
&& \ep{g(\bm x,\bmt z)}  \leq \tau + k(\Delta_q(\bbp,\bbphat))^q &\forall \bbp \in \mathcal P_0(\cz)\\
&\Leftrightarrow & 
\eq{g(\bm x,\bmt z)}  \leq \tau + k \eq {\|\bmt{z}-\bmt{v}\|^q} &\forall \bbq \in \mathcal P_0(\cz^2): (\bmt z,\bmt v) \sim \bbq, \bmt v
\sim \bbphat\\
&\Leftrightarrow & 
\eq{g(\bm x,\bmt z) -k\|\bmt{z}-\bmt{v}\|^q}  \leq \tau  &\forall \bbq \in \mathcal P_0(\cz^2): (\bmt z,\bmt v) \sim \bbq, \bmt v
\sim \bbphat\\
&\Leftrightarrow & \displaystyle 
\frac{1}{\Omega}\sum_{\omega \in [\Omega]} \sup_{\bbp' \in \mathcal P_0(\cz)}\mathbb E_{\bbp'}\left[g(\bm x,\bmt z) -k\|\bmt{z}-\bmh{z}^\omega\|^q\right]  \leq \tau  \\
&\Leftrightarrow & \displaystyle 
\frac{1}{\Omega} \sum_{\omega \in [\Omega]} \sup_{\bm z \in \cz}\left\{g(\bm x,\bm z) -k\|\bm{z}-\bmh{z}^\omega\|^q\right\}  \leq \tau.
\end{array}
$$
Hence, the data-driven robust satisficing problem can be expressed as the following robust optimization problem,
\begin{equation}
\label{eq:data-drivenRS2}
\begin{aligned}
 &\text{minimize} && k\\
    &\text{subject to} &&
     \frac{1}{\Omega}\sum_{\omega \in [\Omega]} \sup_{\bm z \in \cz} \left\{ g(\bm x,\bm z) - k \|\bm z - \bmh z^\omega\|^q\right\}  \leq \tau \\
    & && \bm{x} \in \mathcal{X}, k\geq 0,
\end{aligned}  
\end{equation}
which is computationally equivalent to the data-driven robust optimization model. In the numerical studies of \cite{long2022robust}, the family of solutions generated by the robust satisficing model would dominate of those obtained by solving the data-driven the robust optimization models.
On solving supervised learning problems, 
\cite{sim2021tractable} demonstrate in studies on well-known data sets that the robust satisficing approach would generally outperform the robust optimization approach, when the respective target and size parameters are determined via cross-validation.
Hence, we will focus on solving robust satisficing model. Observe that we can express Problem~\eqref{eq:data-drivenRS2} as
\begin{equation}
\label{eq:data-drivenRS3}
\begin{aligned}
 &\text{minimize} && k\\
    &\text{subject to} &&
     \frac{1}{\Omega}\sum_{\omega \in [\Omega]} \upsilon_\omega  \leq \tau \\
     &&& (\bm x,\upsilon_\omega,k) \in \mathcal Y^\omega & \forall \omega \in [\Omega]\\
     &&& \bm x \in \cx, \bm \upsilon \in \bbr^\Omega, k \in \bbr_+,
\end{aligned}  
\end{equation}
where
$$
\mathcal Y^\omega \triangleq \left\{(\bm x,\upsilon,k) \in \cx \times \bbr \times \bbr_+ ~\left| ~ \sup_{\bm z \in \cz^\omega} \left\{ g^\omega(\bm x,\bm z) - k p(\bm z) \right\}\leq \upsilon\right.\right\}, 
$$
$
\cz^\omega \triangleq \{ \bm z \in \bbr^{n_z} | \bm z + \bmh z^\omega \in \cz \},$
$g^\omega(\bm x,\bm z) \triangleq   g(\bm x,\bm z + \bmh z^\omega)$
and 
$
p(\bm z) \triangleq  \|\bm z\|^q.$
Note that under this representation, we have $\bm 0 \in \cz_\omega$ for all $\omega \in [\Omega]$ and and $p(\bm 0)=0$. The set $\mathcal Y^\omega$ is typically computational intractable to characterize exactly. \cite{long2022robust} focus on piecewise-linear and bi-convex evaluation functions and their proposed approach to solve the robust satisficing model does not extend beyond linearity. 

To extend the scope of robust satisficing, we will replace $\mathcal Y^\omega$, with a computationally tractable convex set, $\bar{\mathcal Y}^\omega \subseteq \mathcal Y^\omega$, for all $\omega \in [\Omega]$, in the following optimization problem, 
\begin{equation}
\label{eq:data-drivenRS4}
\begin{aligned}
 &\text{minimize} && k\\
    &\text{subject to} &&
     \frac{1}{\Omega}\sum_{\omega \in [\Omega]} \upsilon_\omega  \leq \tau \\
     &&& (\bm x,\upsilon_\omega,k) \in \bar{\mathcal Y}^\omega & \forall \omega \in [\Omega]\\
         &&& \bm x \in \cx, \bm \upsilon \in \bbr^\Omega, k \in \bbr_+. 
\end{aligned}  
\end{equation}
As Problem~\eqref{eq:data-drivenRS4} is a conservative approximation of Problem~\eqref{eq:data-drivenRS}, a key challenge in approximating the robust satisficing model is to choose $\bar{\mathcal{Y}}^\omega$, $w \in [\Omega]$ such that Problem~\eqref{eq:data-drivenRS4} remains feasible for any appropriately chosen target $\tau > Z_0$, where we recall that $Z_0$ denotes the optimal objective value of the empirical optimization model~\eqref{eq:idealopt}.

\begin{definition}
\label{def:safeapprox}
We say Problem~\eqref{eq:data-drivenRS4} is a {\em tractable safe approximation} of the data-driven robust satisficing problem~\eqref{eq:data-drivenRS3} if for all $\omega \in [\Omega]$ the following properties of  $\bar{\mathcal Y}^\omega$ hold: 
\begin{enumerate}
    \item $\bar{\mathcal Y}^\omega$ is a computationally tractable convex set.
    \item $\bar{\mathcal Y}^\omega \subseteq \mathcal Y^\omega$.
\item The set
$\left\{ \bm x \in \cx \left| (\bm x,\upsilon, k) \in \bar{\mathcal Y}^\omega \right.\right\}$
is nondecreasing in $\upsilon$ and $k$. 
    \item For any $\bm x \in \cx$ and  $\upsilon > g^\omega(\bm x,\bm 0)$, there exists a finite $\bar{k} \geq 0$ such that 
$(\bm x,\upsilon,\bar k) \in \bar{\mathcal Y}^\omega.$
\end{enumerate}
\end{definition}

\begin{proposition}
\label{prop:RSfeasibility}
Under Assumption~\ref{assm:existence},if Problem~\eqref{eq:data-drivenRS4} is a  tractable safe approximation of Problem~\eqref{eq:data-drivenRS3}, then  Problem~\eqref{eq:data-drivenRS4} is feasible for all $\tau > Z_0$.
\end{proposition}
\begin{proof}{Proof.}
The proof is relegated to Appendix \ref{appendix:proof_of_results}.
\end{proof}

\section{Tractable formulations and safe approximations}{\label{sec:tractable_conic_formulations}}

In this section, we focus on solving the data-driven robust satisficing problem~\eqref{eq:data-drivenRS2}, either exactly or approximately, while ensuring that the problem would remain feasible for all appropriate chosen targets $\tau > Z_0$. As the consequence of Proposition~\ref{prop:RSfeasibility}, we will
derive either exact or tractable safe approximation of the the data-driven robust satisficing problem~\eqref{eq:data-drivenRS2} by focusing on the set 
\begin{equation}
\label{eq:corerobustcounterpart}
\mathcal Y \triangleq \left\{(\bm x,\upsilon,k) \in \cx \times \bbr \times \bbr_+ ~\left| ~ \sup_{\bm z \in \cz} \left\{ g(\bm x,\bm z) - k p(\bm z) \right\}\leq \upsilon \right.\right\},
\end{equation}
for a given convex support set $\cz$ that contains the origin, $\bm 0 \in \cz$, a non-piecewise-linear and biconvex evaluation function $g(\bm x,\bm z)$ and a convex non-negative penalty function that satisfies $p(\bm z) = 0$ iff $\bm z = \bm 0$. Note that $p(\bm{z})$ is not necessarily a power norm and that we have removed the superscript $\omega$ to simply our exposition.  Indeed, how we can construct the tractable convex set $\bar{\mathcal Y} \subseteq \mathcal Y$ that satisfies the properties of Definition~\ref{def:safeapprox}, would depend on the functions $g$ and $p$ as well the support set $\cz.$ While robust counterparts and their tractable approximations have been well studied in the literature, the key challenge is to ensure that for any $\bm x \in \cx$ and  $\upsilon > g(\bm x,\bm 0)$,  there exists a finite $k \in \bbr_+$ such that 
$(\bm x,\upsilon, k) \in \bar{\mathcal Y}$.

For a biconvex quadratic evaluation function, a quadratic penalty function and an ellipsoidal uncertainty set, we derive the exact reformulation in the form of a tractable semidefinite optimization problem. Then, for a more general conic representable function with polyhedral support and penalty, we focus on obtaining a tractable safe approximation using {\em affine dual recourse adaptation} technique. 
 
\subsection{Quadratic evaluation and penalty functions}\label{subsec:quad_constr}
This problem is motivated from the  classical robust optimization proposed in \cite{Bental1998}, which involves the following quadratic biconvex evaluation function,
$$
g(\bm x, \bm z) = \| \bm A(\bm{x})\bm z + \bm a(\bm x)\|_2^2  + \bm b(\bm x)^\top\bm z + c(\bm x),
$$
where $\bm{a}: \mathbb{R}^{n_x} \rightarrow \mathbb{R}^{n_a},\;\bm{A}: \mathbb{R}^{n_x} \rightarrow \mathbb{R}^{n_a \times n_z},\;\bm{b}: \mathbb{R}^{n_x} \rightarrow \mathbb{R}^{n_z},\;c:\mathbb{R}^{n_x} \rightarrow \mathbb{R}$ are affine mappings of the decision variables $\bm{x}$.  This is perhaps the only meaningful nonlinear example in which the robust satisficing problem is tractable and can be formulated exactly as a convex optimization problem without the need for safe approximation. To do so, the penalty function has to be quadratic, and we focus on a squared Euclidean norm, 
$p(\bm z) =\| \bm z \|_2^2 = \bm z^\top \bm z$. We can also  consider an ellipsoidal support set 
$\mathcal{E}_r = \left\{ \bm{z} \in \mathbb R^{n_z} ~|~ \| \bm z \|_2\leq r \right\}$, for $r > 0$.

%Similar to robust optimization,  the tractability of the robust satisficing depends on function $g$, the characterization of the support set $\mathcal{Z}$ (instead of the uncertainty set $\mathcal{E}(r)$) and additionally the penalty function $p$. 

\begin{proposition}\label{prop:quadratic_satisficing}
The set
$$
\mathcal Y = \left\{ (\bm x,\upsilon,k) \in \cx \times \bbr \times \bbr_+ ~\left| 
\|\bm A(\bm{x})\bm z + \bm a(\bm x)\|_2^2  + \bm b(\bm x)^\top \bm z + c(\bm x) \leq k \bm z ^\top \bm  z+\upsilon, \,\, \forall \bm{z}\in \mathcal{E}_r \right.\right\}
$$ is equivalent to the following\\
\begin{equation*}
%\mathcal Y = 
\left\{ (\bm x,\upsilon,k) \in \cx \times \bbr \times \bbr_+ ~\left| 
\begin{array}{l}
\exists \lambda \geq 0:\\
{\left[ \begin{array}{ccc}
{\bm{I}_{n_a}} & {\bm{a}(\bm{x})} & {\bm{A}(\bm{x})} \\ 
{\bm{a}(\bm{x})^{\top}} & {-c(\bm{x})+\upsilon-\lambda r^{2}} & {-\frac{1}{2} \bm{b}(\bm{x})^{\top}} \\ 
{\bm{A}(\bm{x})^{\top}} & {-\frac{1}{2} \bm{b}(\bm{x})} & (k+\lambda) \bm{I}_{n_{z}}
\end{array}
\right] \succeq \bm{0}}
\end{array}
\right.\right\}.
\end{equation*}
Moreover the set is feasible for any $\bm x \in \cx, \upsilon> g(\bm x,\bm 0) =  \|\bm a(\bm x)\|_2^2  + c(\bm x)$. 
\end{proposition}
\begin{proof}{Proof.}
The proof follows from \cite{Bental1998} and is relegated to Appendix \ref{appendix:proof_of_results}.
\end{proof}

For a more general robust satisficing problem, exact tractable reformulations may not exist, and thus approximations would be needed. We next introduce a class of conic representable evaluation functions for which we can derive a tractable safe approximation.

% We consider next the case of polyhedral support sets, and we use polyhedral penalty function to penalize the deviation of the uncertain parameter $\bm{z}$ from its nominal value $\hat{\bm{z}}$.

\subsection{Conic representable evaluation function}
To represent the non-piecewise-linear and biconvex objective function, we consider a rather general conic representable evaluation function $g$ of the form
\begin{equation}
\label{eq:convexfunction}
\begin{aligned}
      g(\bm x, \bm z) =\ &\text{minimize} && {\bm{d}}^\top \bm{y}\\
    &\text{subject to} && \bm{B}\bm{y} \succeq_{\mathcal{K}} \bm{f}(\bm{x}) + \bm{F}(\bm{x})\bm{z} \\
     &&&\bm{y} \in \mathbb{R}^{n_{y}},
\end{aligned} 
\end{equation}
where  $\bm{f}: \mathbb{R}^{n_x} \rightarrow \mathbb{R}^{n_f},\;\bm{F}: \mathbb{R}^{n_x} \rightarrow \mathbb{R}^{n_f \times n_z}$ are affine mappings in $\bm{x}$, and $\mathcal{K}$ is a proper cone,  {\em i.e.}, a  full dimensional, closed and pointed convex cone. Since the decision variable $\bm{y}$ in Problem~\eqref{eq:convexfunction} is made after observing $\bm x$ and $\bm z$, we will call $\bm y$ the \textit{recourse} variable. Indeed, if the cone $\mathcal{K}$ is the non-negative orthant, then the $g$ function would represent the second stage optimization problem of a standard two-stage stochastic optimization model. The function $g$ is sufficiently broad enough to represent many types of functions considered in the robust optimization literature, including biconvex linear or quadratic constraint functions \citep{Bental1998},  non-linear biconvex constraint functions \citep{Roos2020,trevor2017robust_sdpsocp}, adaptive linear optimization model \citep{bental2004adjustable,Bertsimas2016} and adaptive convex optimization model \citep{trevor2018dual_twostage}.

\iffalse
\begin{assumption}[Convexity and practicable solvability]\label{ass:convextractability}
We assume that $\cx$ and $\cz$ are compact and convex sets. Moreover, we assume that any convex optimization problem over $\bm x \in \cx$, involving a modest number of additional decision variables, linear and $\mathcal K$-conic inequalities are  practicably solvable, {i.e.}, it can be solved to optimality within reasonable time using current available solvers such as CPLEX, Gurobi, Mosek, SDPT3, among others.
\end{assumption}

Most optimization problems become much harder to solve when they are subject to uncertainty. Note that as $\bm{x}$ and $\bm{z}$ appear on the right-hand side of the conic constraint, $g$ is a biconvex function, and we would expect the problem to be much harder to solve exactly if the second argument is subject to uncertainty. In the simplest case, where $g$ is a biaffine function given by  
$$
g(\bm x, \bm z) = \bm x^\top  \bm A \bm z + \bm b^\top \bm x + \bm c^\top \bm z + d, 
$$
we can express 
$$
g(\bm x, \bm z)=\min_{y \in \bbr}\left\{y \mid y \geq \left(\bm b^\top \bm x +d\right) + \left(\bm x^\top  \bm A + \bm c^\top\right)\bm z\right\}.
$$

\fi

\begin{assumption}[Complete and bounded recourse] \label{assm:evalutionfunction} We assume that for any $\bm{v} \in \mathbb{R}^{n_f}$, there exists a $\bm{y} \in \mathbb{R}^{n_y}$ such that~$\bm{B}\bm{y} \succeq_{\mathcal K} \bm{v}$. Moreover, there does not exist $\bm{y} \in \mathbb{R}^{n_y}$ such that $\bm{B}\bm{y} \succeq_{\mathcal{K}} \bm 0$ and $\bm d^\top \bm y<0$. These conditions ensure that Problem~\eqref{eq:convexfunction} is always finite and strictly feasible.
\end{assumption}

We remark that the term {\em complete recourse} is an extension of the same term used in stochastic linear optimization \citep[see, {\em e.g.},][]{birge2011introduction} to our model where the second stage problem is a conic optimization problem. It  ensures that Problem~\eqref{eq:convexfunction} is strictly feasible. To see the strict feasibility, let $\bm v \in \bbr^{n_f}$ such that $\bm{v} \succ_{\mathcal K} \bm{0}$, and under compete recourse, there exists a $\bm{y} \in \mathbb{R}^{n_y}$ such that $\bm{B}\bm{y} \succeq_{\mathcal{K}} \bm{f}(\bm{x}) + \bm{F}(\bm{x})\bm{z} +\bm v \succ_{\mathcal{K}} \bm{f}(\bm{x}) + \bm{F}(\bm{x})\bm{z}$. The bounded recourse condition ensures that the $g$ function of Problem~\eqref{eq:convexfunction} is always finite. 

As we will later show, we require  Assumption~\ref{assm:evalutionfunction} to obtain a tractable safe approximation of the robust satisficing problem. However, it is not  necessarily clear how we can transform a nonlinear evaluation function to satisfy the assumption. 

\begin{proposition} \label{prop:propercone}
Consider a differentiable convex loss function, 
$\ell:\bbr^{n_{w}}\rightarrow \bbr$ whose epigraph
${\rm epi}(\ell)\triangleq\left\{ (\bm w,v) \in \bbr^{n_{w}}\times \bbr\mid \ell(\bm w) \leq v \right\}$ does not contain a line. 
The perspective cone of $\ell$, defined as 
$$
\mathcal{K}_{\ell} \triangleq \left\{ (\bm w,v_1,v_2) \in \mathbb R^{n_w} \times \mathbb R\times \mathbb R_{+} \mid v_2 \ell(\bm w/v_2) \leq v_1\right\},
$$
where
$$
0 \ell(\bm w/0)  \triangleq  \sup_{\bm u \in \bbr^{n_w}}\left\{ \nabla \ell(\bm u)^\top\bm w  \right\}.
$$
is a proper cone, {\em i.e.}, closed, convex, solid and pointed.
\end{proposition}
\begin{proof}{Proof.}
%The proof is relegated to Appendix %\ref{appendix:proof_of_results}.
%\end{proof}
%\begin{proof}{Proof of Proposition~\ref{prop:propercone}.}
It is not difficult to see that $\mathcal{K}_\ell$ is a cone that is closed and solid. 
We will next show that the cone $\mathcal{K}_\ell$ is also convex, and pointed. Since $\ell$ is convex and differentiable, then for any $\bm u, \bm w \in \bbr^{n_w}$, we have 
$$
 \ell(\bm w) \geq   \displaystyle \nabla \ell(\bm u)^\top(\bm w - \bm u) + \ell(\bm u)
 $$
so that 
$$
\ell(\bm w) = \max_{\bm u \in \bbr^{n_w}}\left\{ \nabla \ell(\bm u)^\top(\bm w - \bm u) + \ell(\bm u) \right\},
$$
which allows us to define
\begin{equation}
\label{eq:persepctiverep}
y_2 \ell(\bm w/y_2) \triangleq \sup_{\bm u \in \bbr^{n_w}}\left\{ \nabla \ell(\bm u)^\top(\bm w - y_2\bm u) + y_2\ell(\bm u) \right\},
\end{equation}
for $y_2\geq 0$.
As a result, the cone  can also be expressed as 
$$
\mathcal{K}_\ell = \left\{ (\bm w,v_1,v_2) \in \bbr^{n_w} \times \bbr\times \bbr_{+} \mid \nabla \ell(\bm u)^\top(\bm w - v_2\bm u) + v_2\ell(\bm u) \leq v_1 ~\forall \bm u \in \bbr^{n_w} \right\},
$$
which constitutes an intersection of infinitely many half spaces. The cone $\mathcal{K}_\ell$ is thus manifestly convex. We will show by contradiction that the cone is also pointed. Suppose otherwise, there would exist $(\bm w, v_1, v_2)\neq \bm 0$ such that $(\bm w, v_1, v_2) \in \mathcal K_\ell$ and $-(\bm w, v_1, v_2) \in \mathcal K_\ell$. Hence, we must have $v_2 = 0$,
and $(\bm w, v_1)\neq \bm 0$ such that
$$
\begin{array}{rcl}
\sup\limits_{\bm u \in \bbr^{n_w}}\left\{ \nabla \ell(\bm u)^\top\bm w  \right\} &\leq& v_1\\
\sup\limits_{\bm u \in \bbr^{n_w}}\left\{ -\nabla \ell(\bm u)^\top\bm w  \right\} &\leq& -v_1,
\end{array}
$$
or equivalently 
$$
 \nabla \ell(\bm u)^\top\bm w  = v_1 \qquad \forall \bm u \in \bbr^{n_w}.
$$
Hence, for all $k \in \bbr$,
$$
\ell(k\bm w) = \max_{\bm u \in \bbr^{n_w}}\left\{ \nabla \ell(\bm u)^\top(k\bm w - \bm u) + \ell(\bm u) \right\} = kv_1 + \max_{\bm u \in \bbr^{n_w}}\left\{ \nabla \ell(\bm u)^\top(- \bm u) + \ell(\bm u) \right\} = kv_1 + \ell(\bm 0), 
$$
which contradicts that the epigraph of $\ell$ does not contain a line because
$(k\bm w,kv_1 + \ell(\bm 0)) \in {\rm epi}(\ell)$ for all $k \in \bbr$. 
\qed
\end{proof}

Using the perspective of convex function technique, we can express the function as a conic minimization problem as follows
\begin{equation}
\label{eq:perspectivefunctionrep}
\begin{array}{rcll}
\ell(\bm w) = &{\rm minimize}& y \\
&{\rm subject~to} & \left[ \begin{array}{c} \bm 0\\1 \\ 0\end{array}\right] y  \succeq_{\mathcal{K}_\ell}  \left[ \begin{array}{c} - \bm w\\  0 \\-1  \end{array}\right],\\
&& y \in \bbr.
\end{array}  
\end{equation}

We observe that given affine mappings  $\bm{w}: \mathbb{R}^{n_x} \rightarrow \mathbb{R}^{n_w},\;\bm{W}: \mathbb{R}^{n_x} \rightarrow \mathbb{R}^{n_w \times n_z}$,  the evaluation function $g(\bm x,\bm z) = \ell(\bm{w}(\bm{x}) + \bm{W}(\bm{x})\bm{z})$ would be conic representable. However, the condition of complete recourse required in Assumption~\ref{assm:evalutionfunction} does not hold because there does not exist $y \in \bbr$ such that 
$$
\left[ \begin{array}{c} \bm 0\\1 \\ 0\end{array}\right] y  \succeq_{\mathcal{K}_\ell}  \left[ \begin{array}{c}  \bm 0\\  0 \\1  \end{array}\right].
$$
In our computation study, we have verified that the direct application of the perspective of convex function approach does not lead to a feasible tractable safe approximation problem. 

\subsection{Perspective casting}
To resolve the feasibility issue, we introduce a novel {\em perspective casting} approach to provide a conic optimization model that would satisfy Assumption~\ref{assm:evalutionfunction}, and that its objective would match $\ell(\bm w)$ when evaluated over an {\em featured domain}, $\bm w \in \mathcal W \subseteq \bbr^{n_w}$.

\begin{definition}
The {\em perspective depth} of a differentiable convex function $\ell:\bbr^{n_w}\rightarrow \bbr$ on a {\em featured domain}, $\mathcal{W} \in \mathbb{R}^{n_w}$  is defined as  
$$
\bar{P}_{\ell,\mathcal W} \triangleq  \sup_{\bm u \in \mathcal W} \left\{
\nabla \ell(\bm u)^\top\bm u - \ell(\bm u)
\right\}.
$$
\end{definition}

\begin{figure}[hbt!]
    \centering 
    \includegraphics[scale=0.40]{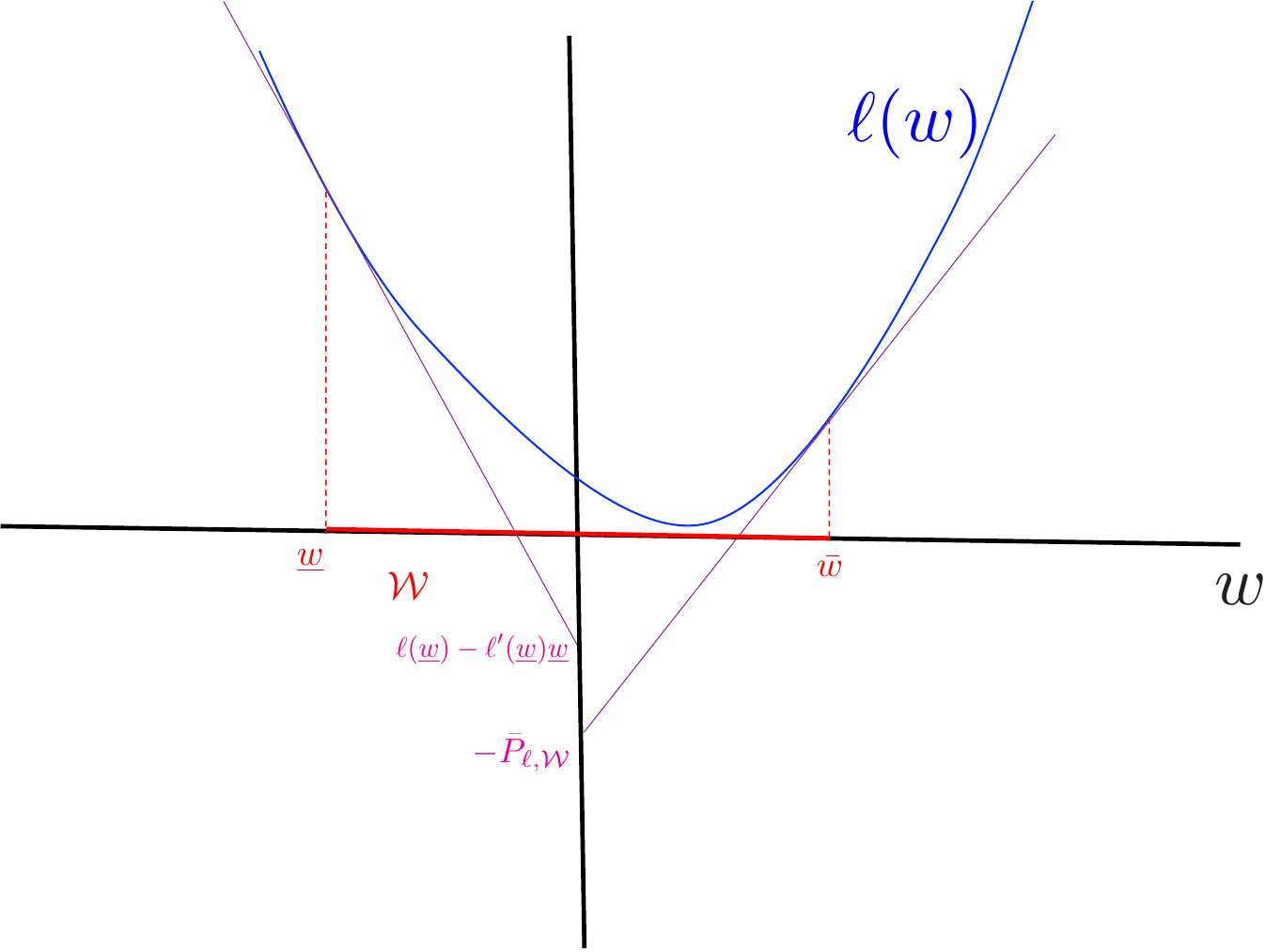}
    \caption{An illustration of perspective depth on $\mathcal W = [\underline w,\bar w]$.}
    \label{fig:perspectivedepth}
\end{figure}

\begin{proposition}
\label{prop:BigM}
For a univariate differentiable convex loss function $\ell: \bbr \rightarrow \bbr$ and a featured domain $\mathcal W = [\underline{w},\bar{w}]$, 
$$
 \bar P_{\ell, \mathcal W} = \max_{ v \in \{\underline w,\bar w\}} \left\{ \ell '( v)v - \ell( v) \right\}.
$$
\end{proposition}
\begin{proof}{Proof.}
The proof is relegated to Appendix \ref{appendix:proof_of_results}.
\end{proof}

Specking intuitively, the perspective depth refers to the distance of the closest illuminating point directly below the origin,  so that the illuminating point would be visible from everywhere on the  surface of the opaque epigraph of $\ell$ that can be mapped from the featured domain,  $\mathcal{W}$ (see Figure~\ref{fig:perspectivedepth}). 
%Next we propose the perspective casting approach for constructing a function $\hat \ell:\bbr^{n_w}\rightarrow \bbr$ that has an epigraph ${\rm epi}(\hat \ell) \supseteq {\rm epi}(\ell)$ so that $\ell(\bm w) = \hat \ell(\bm w)$ for all $\bm w \in \mathcal W$. Moreover,  to provide a conic optimization model with bounded and complete recourse, ${\rm epi}(\hat \ell)$ includes regions that have been blocked by ${\rm epi}(\ell)$ from the illuminating point.

\begin{theorem}[Perspective casting] \label{thm:convextoconic}
Consider a differentiable convex loss function $\ell:\bbr^{n_w}\rightarrow \bbr$ whose epigraph does not contain a line, and a featured domain $\mathcal{W} \subseteq \mathbb{R}^{n_w}$. 
For a given  $P\geq \bar{P}_{\ell,\mathcal W}$, let   
\begin{equation}
\label{eq:convextoconic}
\begin{array}{rcll}
\hat{\ell}(\bm w) = &{\rm minimize}& y_1+ Py_2 \\
&{\rm subject~to} & \underbrace{\left[ \begin{array}{cc} 
\bm 0& \bm 0 \\
1&0 \\ 
0&1\\
0&1
\end{array}\right]}_{\bm{B}} \left[\begin{array}{c} y_1 \\y_2 \end{array}\right]  \succeq_{\mathcal{K}}  \left[ \begin{array}{c}  - \bm w\\  -P \\0\\
1\end{array}\right]\\
&& \bm y \in \bbr^2,
\end{array}  
\end{equation}
where 
$\mathcal{K} = \mathcal K_\ell \times \bbr_+$. Then
$$
\begin{array}{ll}
\hat{\ell}(\bm w) \leq \ell(\bm w) &~~~~\forall  \bm w \in \mathcal \bbr^{n_w}, \\
\hat{\ell}(\bm w) = \ell(\bm w) &~~~~ \forall  \bm w \in \mathcal W.
\end{array}
$$
Moreover, Problem~\eqref{eq:convextoconic} has a complete and bounded recourse. 
\end{theorem}
\begin{proof}{Proof.}
%\begin{proof}{Proof of Theorem~\ref{thm:convextoconic}.}
Observe that  Problem~\eqref{eq:convextoconic} is equivalent to the following optimization problem
$$
\hat{\ell}(\bm w) = 
\min_{y_2 \geq 1}\left\{  Py_2 +y_2\ell(\bm w/y_2) -P\right\},
$$
hence 
$$
\hat{\ell}(\bm w) \leq \ell(\bm w) \qquad \forall  \bm w \in \bbr^{n_w}. 
$$
For any $\bm{w} \in \mathcal{W}$, we have
$$
\begin{array}{rcl}
\hat{\ell}(\bm w) &=& \displaystyle \min_{y_2\geq 1}\left\{  Py_2 +y_2\ell(\bm w/y_2) -P\right\}\\
&=&  \displaystyle  
\adjustlimits \min_{y_2\geq 1}\sup_{\bm u \in \bbr^{n_w}} \left\{  Py_2 +
\nabla \ell(\bm u)^\top(\bm w - y_2\bm u) + y_2\ell(\bm u) -P\right\}\\
&\geq&  \displaystyle  
\adjustlimits \min_{y_2\geq 1}\sup_{\bm u \in  \mathcal{W}} \left\{  Py_2 +
\nabla \ell(\bm u)^\top(\bm w - y_2\bm u) + y_2\ell(\bm u) -P\right\}\\
&=&  \displaystyle  
\adjustlimits \min_{y_2\geq 1}\sup_{\bm u \in \mathcal{W}} \left\{  y_2(P - 
\nabla \ell(\bm u)^\top\bm u + \ell(\bm u))+ 
\nabla \ell(\bm u)^\top\bm w   -P\right\}\\
&=&  \displaystyle  
\sup_{\bm u  \in \mathcal{W}} \left\{  
\nabla \ell(\bm u)^\top(\bm w -\bm u) + \ell(\bm u)) \right\}\\
&=&  \ell(\bm w),
\end{array}
$$
where the first inequality is due to~\eqref{eq:persepctiverep} and the second-to-last equality is due to having $y_2=1$ at optimality because  
$$
P \geq  
\nabla \ell(\bm u)^\top\bm u - \ell(\bm u)
$$
for all $\bm u \in \mathcal{W}$. Finally, to show that Problem~\eqref{eq:convextoconic} has a complete recourse, consider any $\bm w \in \bbr^{n_w}, \bm v \in \bbr^3$ and observe that 
$$
\bm{B}\left[ \begin{array}{c}   y_1 \\y_2\end{array}\right] = 
\left[ \begin{array}{ccc} 
\bm 0& \bm 0 \\
1&0 \\ 
0&1\\
0&1\\
\end{array}\right] \left[ \begin{array}{c}   y_1 \\y_2\end{array}\right]
=
\left[ \begin{array}{c}   \bm 0 \\y_1 \\y_2 \\y_2 \end{array}\right]
\succeq_{\mathcal{K}}  \left[ \begin{array}{c}  \bm w\\  v_1 \\v_2
\\v_3 \end{array}\right]
$$
is feasible for
\begin{equation*}
\begin{aligned}
    &y_1 = v_1 + \left[ (v_3 - v_2)^+ +1 \right] \ell \left( -\frac{\bm{w}}{(v_3 - v_2)^++1}\right) \\
    &y_2 = \max \{ v_2, v_3 \}+1. 
    \end{aligned}
\end{equation*}
To show that the problem has a bounded recourse, observe that for any $\bm y \in \bbr^2, \bm B \bm y \in \mathcal K$,  we have  $y_2 \geq 0$ and
$$
\begin{array}{rcl}
 Py_2 + y_1 &\geq&  Py_2 +y_2\ell(\bm 0/y_2) \\
&=&Py_2 + \sup\limits_{\bm u \in \bbr^{n_w}}\left\{ \nabla \ell(\bm u)^\top(\bm 0 - y_2\bm u) + y_2\ell(\bm u) \right\}\\
&=&Py_2 + y_2\sup\limits_{\bm u \in \bbr^{n_w}}\left\{ -\nabla \ell(\bm u)^\top\bm u + \ell(\bm u) \right\}\\
&\geq&Py_2 + y_2\inf\limits_{\bm u \in \mathcal W}\left\{ -\nabla \ell(\bm u)^\top\bm u + \ell(\bm u) \right\}\\
&=&\left(P-\sup\limits_{\bm u \in \mathcal W}\left\{ \nabla \ell(\bm u)^\top\bm u - \ell(\bm u) \right\} \right)y_2\geq 0,
\end{array}
$$
where the first equality again follows from~\eqref{eq:persepctiverep}. Hence, the proof is completed.
\qed
\end{proof}

We note that the Huber loss function with parameter $\delta>0$ has a finite perspective depth of $\delta^2/2$ on its entire domain, {\em i.e.}, $\mathcal W=\bbr$. However, although the Huber function is Lipschitz continuous, it does not imply that any differentiable convex Lipschitz continuous function would have finite perspective depth on $\mathcal W=\bbr$, as exemplified in the following function:   
$$
\ell(w) = |w| -\log(|w|+1),
$$
under which $\bar{P}_{\ell, \bbr} \geq \lim_{w \uparrow \infty} \ell'(w)w - \ell(w) = \infty$.

\begin{example}
The log-sum-exponential function commonly used in logistic regression,
$$
\ell(\bm w) = \log\left(\sum_{i \in [n_w]}\exp(w_i) \right), 
$$
has a finite perspective depth of zero on its entire domain $\mathcal W = \bbr^{n_w}$. To see this, let 
 $$
 r(t) = \ell(\bm u t)
 $$
for some  $\bm u \in \bbr^{n_w}$. Observe that 
$$
r'(1) -r(1) = \nabla \ell(\bm u)^\top\bm u - \ell(\bm u).
$$
Since $r$ is convex and differentiable in $t$, from Proposition~\ref{prop:BigM}, we have 
$$
r'(1) -r(1) \leq \max\left\{r'(0)\times 0 -r(0), \lim_{t\uparrow \infty}\{ r'(t)t -r(t) \} \right\} =
\max\left\{-\log(n_w), \lim_{t\uparrow \infty}\{ r'(t)t -r(t) \} \right\}.
$$
However,
$$
\begin{array}{rcl}
\lim\limits_{t \uparrow \infty}\{r'(t)t -r(t)\} &=& \lim\limits_{t \uparrow \infty} \log\left( \frac{\exp\left(\frac{\sum\limits_{i \in [n_w]} \exp(u_i t) u_i t }{\sum\limits_{i \in [n_w]}\exp(u_i t)}\right)  }{\sum\limits_{i \in [n_w]} \exp(u_i t) } \right)\\
&=& \lim\limits_{t \uparrow \infty} \log\left( \frac{\exp\left(\frac{\sum\limits_{i \in [n_w]} \exp((u_i-\bar u) t) u_i t }{\sum\limits_{i \in [n_w]}\exp((u_i-\bar u) t)}\right)  }{\exp(\bar u t)\sum\limits_{i \in [n_w]} \exp((u_i-\bar u) t) } \right)\\
&=& \lim\limits_{t \uparrow \infty} \log\left( \frac{\exp\left(\frac{\bar n \bar u t}{\bar n}\right)  }{\bar n \exp(\bar u t) } \right) = -\log(\bar n),
\end{array}
$$
where $\bar u = \max_{i\in [n_w]}\{u_i\}$ and $\bar{n} = |\{ i \in [n_w] | u_i = \bar u \} |$.

Hence, for any $\bm u \in \bbr^{n_w}$
$$
\nabla \ell(\bm u)^\top\bm u - \ell(\bm u) \leq \max\{-\log(n_w),-\log(\bar n) \} \leq 0,
$$
and that the bound is tight when $\bar n=1$, {\em i.e.}, $\bm u$ has a unique maximum.
\end{example}

\begin{example}[Disutility function]
\label{example:convexrecourse}
We can consider the disutility function of the cost associated with a two-stage linear optimization problem with complete recourse as follows
$$
\begin{aligned}
 g(\bm x, \bm z) =\ &\text{minimize} && \ell\left(\bm c^\top \bm x +  {\bm{d}}^\top \bm{y}\right)\\
    &\text{subject to} && \bm {D}\bm{y} \geq \bm{a}(\bm{x}) + \bm{A}(\bm{x})\bm{z} \\
     &&&\bm{y} \in \mathbb{R}^{n_{y}},
\end{aligned}
$$
for some differentiable convex and increasing function $\ell:\mathbb R \rightarrow \mathbb R \cup \{\infty\}$, representing disutility, some matrix $\bm D \in \mathbb{R}^{n_a \times n_y}$ and affine mappings $\bm{a}: \mathbb{R}^{n_x} \rightarrow \mathbb{R}^{n_a},\;\bm{A}: \mathbb{R}^{n_x} \rightarrow \mathbb{R}^{n_a \times n_z}$. Hence, for any $\bm v \in \bbr^{n_{a}}$ there exists $\bm y \in \bbr^{n_y}$ such that $\bm D \bm y \geq \bm v$. Suppose we can find $\underline w$ and $\bar w$ such that 
$$
\underline w  \leq \ell^{-1}(g(\bm x,\bm z)) \leq \bar w \qquad \forall \bm x \in \cx, \bm z \in \cz. 
$$
From Proposition~\ref{prop:BigM}, we choose $P\in \bbr$ such that
$$
P \geq \max \{ \ell'( \underline{w})\underline{w} - \ell( \underline{w}), \ell'( \bar{w})\bar{w} - \ell( \bar{w}) \}.
$$
Then, according to Theorem~\ref{thm:convextoconic}, we would have the following conic representation.
$$
\begin{array}{rcll}
 g(\bm x, \bm z) = &{\rm minimize}& v_1 + P v_2  \\
&{\rm subject~to} & \left[\begin{array}{cccc} 
\bm d^\top & 0&  0\\
\bm 0^\top& 1&0 \\ 
\bm 0^\top& 0&1 \\ 
\bm 0^\top& 0&1\\
\bm D & 0 &0
\end{array}\right]
\left[ \begin{array}{c} \bm y \\
v_1 \\ v_2 \end{array}\right]
 \succeq_{\mathcal{K}}  \left[ \begin{array}{c}  - \bm c^\top \bm x \\  -P \\0\\1
\\
 \bm{a}(\bm{x}) 
 \end{array}\right]+ 
 \left[ \begin{array}{c}  \bm 0^\top \\  \bm 0^\top \\\bm 0^\top\\
 \bm 0^\top\\
  \bm{A}(\bm{x})
 \end{array}\right]\bm{z}
 \\
&&\bm y \in \bbr^{n_y}, \bm v \in \bbr^2
\end{array}  
$$
where 
$\mathcal{K} = \mathcal K_{\ell}  \times \mathbb R_+  \times \bbr^{n_{a}}_+.$
\end{example}

\subsection{Polyhedral penalty function and support}{\label{subsec:Polyhedral penalty}}

 Under the  conic representable evaluation function,  we can express the feasible set
\begin{equation}
\label{opt:cone_satisficing}
\mathcal Y = \left\{ (\bm x, \upsilon  , k) \in \cx \times \bbr \times \bbr_+ \left|
\begin{array}{ll}
\exists  \bm{y}\in \mathcal{R}^{n_{z},n_{y}}:\\
  \bm{d}^\top\bm{y}(\bm{z}) -  k p(\bm z) \leq \upsilon & \forall \bm{z} \in \mathcal{Z} \\
\bm{B}\bm{y}(\bm{z}) \succeq_{\mathcal{K}} \bm{f}(\bm{x}) + \bm{F}(\bm{x})\bm{z} & \forall \bm{z} \in \mathcal{Z} 
\end{array} \right. \right\}
\end{equation}
where $\mathcal{R}^{m,n}$ denotes the family of all functions from $\mathbb{R}^m$ to $\mathbb{R}^n$, {\em i.e.,}
$$
\mathcal{R}^{m,n}=\{ \bm y~|~ \bm y: \bbr^m \rightarrow \bbr^n\}.   
$$
The feasible set remains intractable even if we restrict the recourse $\bm y$ to a static function that does not depend on $\bm z$. As we will reveal, we overcome this challenge by using a technique of dualizing twice similar to that in \cite{Roos2020}, first over the recourse variables $\bm{y}$ and then over the uncertain parameters $\bm{z}$ to absorb the conic nature of the original problem into a new uncertainty set while the polyhedral support and penalty show up as linear constraints of the resultant formulation. We can thus use familiar approximation methods such as affine recourse adaptation to obtain tractable feasible set $\bar{\mathcal Y} \subseteq \mathcal Y$. Although feasibility is not guaranteed with such approximations even under assumptions of complete recourse \citep[see, {\em e.g.},][]{bertsimas2019adaptive},  our results however show that whenever $\upsilon > g(\bm x, \bm 0)$, the {\em affine dual recourse adaptation} would always yield a feasible solution, $(\bm x, \upsilon,k) \in \bar{\mathcal Y}$ for some $k\in \bbr_+$. To do so, we require the following assumptions on the  penalty function and support, %\notenr{Change to strict inequality ($\nu > g(\bm{x},\bm{0})$) to be consistent with the proof of Thm 4.}
\begin{assumption}[Polyhedral support and penalty] \label{assm:penaltysupport}
We assume the following:
\begin{itemize}
\item[(i)] %{\bf Polyhedral support:}  
The support set $\mathcal{Z}$ is a nonempty polyhedron 
$$\mathcal Z = \{ \bm{z} \in \mathbb{R}^{n_z} \mid  \bm{H}\bm{z} \leq \bm{h}\},$$ for some $\bm{H} \in \mathbb{R}^{n_h \times n_z}$ and $\bm{h} \in \mathbb{R}_+^{n_h}$. 
\item[(ii)] %{\bf Polyhedral penalty:} 
The polyhedral penalty function 
$p(\bm \zeta): \bbr^{n_z}\rightarrow \bbr_+$ can be expressed as 
$$
p(\bm \zeta) = \max_{(\bm \lambda, \eta) \in \mathcal V} \{ \bm \lambda^\top \bm \zeta - \eta\},
$$
where $\mathcal V$ is a polyhedral  
$$
\mathcal V = \{ (\bm \lambda , \eta)  \in \bbr^{n_z} \times \bbr_+ \mid \exists \bm \mu \in \bbr^{n_\mu} ~:~ \bm M \bm \lambda + \bm N \bm \mu  + \bm s \eta \leq \bm t   \},
$$ for some $\bm{M} \in \mathbb{R}^{n_m \times n_z}$, $\bm{N} \in \mathbb{R}^{n_m \times n_\mu}$, and $\bm{s}, \bm t \in \mathbb{R}^{n_m}$. In addition, (i) $\mathcal V$ contains the origin so that $p(\bm \zeta)\geq 0$ and $p(\bm 0)=0$, (ii) there exists $\hat{\bm \mu} \in \bbr^{n_\mu}$ such that $\bm{N}\hat{\bm{\mu}} < \bm{t}$, which ensures $p(\bm \zeta)>0$ if $\bm \zeta \neq \bm 0$, and (iii) $\mathcal V$ is bounded, which ensures $p(\bm \zeta)< \infty$. 
\end{itemize}
\end{assumption}
It is common to choose a polyhedral norm as the penalty function, which has a similar representation.
\begin{proposition}[Polyhedral norm]\label{prop:polypenaltynorm}
Under Assumption~\ref{assm:penaltysupport}, if the penalty function $p$ is a norm, then it has the representation
\begin{equation*}
    p(\bm{\zeta}) = \max_{\bm\lambda \in \bbr^{n_z}, \bm\mu \in \bbr^{n_\mu}} \left\{ \bm\lambda^\top\bm{\zeta} \mid \bm{M}\bm\lambda + \bm{N}\bm\mu \leq \bm{t}  \right\},
\end{equation*}
for which
$$
 \left\{ \bm \lambda \mid \bm{M}\bm\lambda + \bm{N}\bm\mu \leq \bm{t}  \right\}
=
\left\{ \bm \lambda  \mid -\bm{M}\bm\lambda + \bm{N}\bm\mu \leq \bm{t}  \right\}.
$$
Its dual norm is given by
\begin{equation*}
\begin{aligned}
    p^\star(\bm{\zeta}) &=   \min_{\bm\mu \in \bbr^{n_\mu}, \delta \in \bbr_+} \left\{ \delta \mid \bm{M\zeta} + \bm{N\mu} \leq \delta \bm{t} \right\}.
\end{aligned} 
\end{equation*}
\end{proposition}

\begin{proof}{Proof.}
The proof is relegated to Appendix \ref{appendix:proof_of_results}.
\end{proof}
\begin{example}[Budgeted norm]
\label{example:budgeted_norm}
We illustrate the modeling potential of the norm-based penalty defined in Proposition \ref{prop:polypenaltynorm} with a budgeted norm which computes the sum of the $\Gamma \in \{1,\dots,n_z\}$ largest absolute components of an $n_z$-dimensional vector, \emph{i.e.}, 
$$ p_\Gamma(\bm \zeta) = \max_{\mathcal S \subseteq [n_z], |\mathcal S|=\Gamma} \sum_{i \in \mathcal S} |\zeta_i|
$$
so that $p_1(\bm \zeta)= \|\bm \zeta\|_\infty$ and $p_{n_z}(\bm \zeta) = \|\bm \zeta\|_1$.
%{|x_{i_1}|+|x_{i_2}|+\hdots +|x_{i_k}|\},\;1\leq i_1 \leq i_2 \hdots \leq i_k \leq n$$,
This can be represented as the following linear optimization problem 
\begin{equation*}
\begin{array}{rcl}
  p_\Gamma(\bm \zeta)  &=&  \displaystyle \max_{\bm \lambda \in \bbr^{n_z}}\left\{ \bm{\lambda}^\top \bm{\zeta} \left|  \sum_{i \in [n_z]} |\lambda_i| \leq \Gamma, |\lambda_i| \leq 1, ~~\forall i \in [n_z]\right.\right\}\\
  &=& \displaystyle \max_{\bm \lambda, \bm \mu \in \bbr^{n_z}}\left\{ \bm{\lambda}^\top \bm{\zeta} \left|  
  \begin{array}{ll}
  \sum\limits_{i \in [n_z]} \mu_i \leq \Gamma\\ 
  \lambda_i  -\mu_i \leq 0&\quad\forall i \in [n_z]\\
  -\lambda_i  -\mu_i \leq 0&\quad\forall i \in [n_z]\\
  \mu_i \leq 1 &\quad\forall i \in [n_z]
  \end{array}\right.\right\},
  \end{array}
\end{equation*}
which satisfies the properties of the polyhedral penalty in Assumption~\ref{assm:penaltysupport}.
\end{example}

We also remark that since second-order conic constraints can be approximated accurately via a modest sized polyhedron \citep{ben2001polyhedral}, the representation of polyhedral penalty is quite broad and can be used to approximate many different types of convex nonlinear penalty functions such as, {\em inter alia}, convex polynomials and $\ell_p$-norms, for $p\geq 1$.

\subsection{Affine dual recourse adaptation}{\label{subsec:affine dual adaptation}}

Under Assumptions~\ref{assm:evalutionfunction} and \ref{assm:penaltysupport}, we will show that the set $\mathcal{Y}$ which is defined in~\eqref{opt:cone_satisficing} through a conic inequality and a polyhedral uncertainty set admits an equivalent reformulation as a set of a similar nature but with linear constraints and a conic uncertainty set. The following result is a precursor for obtaining the alternate formulation of $\mathcal{Y}$, which would eventually enable us to obtain a safe approximation of the robust satisficing problem via affine dual recourse adaptation. 
\begin{proposition}\label{prop:innermax}
Under Assumption~\ref{assm:penaltysupport}, for any $\bm{a} \in \mathbb{R}^{n_z}$ and $k \geq 0$, we have
$$
\max_{\bm{z} \in \mathcal{Z}} \left\{\bm{a}^\top\bm{z} - k p(\bm{z}) \right\} =\min_{\eta \in \bbr_+, {\bm\beta} \in \mathbb{R}^{n_h}_+}  \left\{ \bm{\beta}^\top\bm{h} + \eta  \mid  (\bm a -\bm H^\top\bm \beta,\eta, k) \in \bar{\mathcal V}  \right\} 
$$
where $\bar{\mathcal V}$ is the perspective cone of $\mathcal V$ given by
$$
\bar{\mathcal V}=  \{ (\bm \lambda , \eta,k)  \in \bbr^{n_z} \times \bbr_+^2 \mid \exists \bm \mu \in \bbr^{n_\mu}:~ \bm M \bm \lambda + \bm N \bm \mu + \bm s \eta \leq  \bm t k    \}.
$$
If the penalty function $p$ is a norm, then  we have
\begin{equation*}
\max_{\bm{z} \in \mathcal{Z}} \left\{\bm{a}^\top\bm{z} - k p(\bm{z}) \right\} =\min_{{\bm\beta} \in \mathbb{R}^{n_h}_+}  \left\{ \bm{\beta}^\top\bm{h} \mid  p^\star(\bm a -\bm H^\top\bm\beta) \leq k \right\}.
\end{equation*}
\end{proposition}
\begin{proof}{Proof.}
The proof is relegated to Appendix \ref{appendix:proof_of_results}.
\qed 
\end{proof}

\begin{theorem}
\label{thm:dual_satisficing}
Under Assumptions~\ref{assm:evalutionfunction} \& \ref{assm:penaltysupport}, for any $\bm{x} \in \mathbb{R}^{n_x}$ and $k \geq 0$, we have
\iffalse 
$$
g(\bm{x},\bm{z}) - k p(\bm{z}) \leq t, \; \forall \bm{z} \in \mathcal{Z} %\text{$\quad$\notenr{Proof modified because of the new format.}}
$$ 
is equivalent to
\fi 
\begin{equation}
\label{opt:cone_satisficing_dual}
    \mathcal{Y} = \left\{ (\bm{x},\upsilon,k) \in \mathcal{X} \times \mathbb{R} \times \mathbb{R}_+
    \left| 
    \begin{array}{ll}
    \exists \bm\beta \in \mathcal{R}^{n_f, n_h}, \bm \mu \in \mathcal{R}^{n_f, n_\mu}, \eta \in  \mathcal{R}^{n_f, 1}: \\ 
    \bm{\rho}^\top \bm{f}(\bm{x}) + \bm\beta(\bm\rho)^\top \bm{h} + \eta(\bm\rho) \leq \upsilon &\forall \bm\rho \in \mathcal{P} \\
\bm M( \bm{F}(\bm{x})^\top \bm\rho-\bm{H}^\top\bm\beta(\bm\rho)) + \bm N \bm\mu(\bm\rho) + \bm s \eta(\bm\rho) \leq \bm t k~~ &\forall \bm\rho \in \mathcal{P} \\
    \bm\beta(\bm\rho) \geq \bm{0}, \ \eta(\bm\rho) \geq \bm{0} & \forall \bm\rho \in \mathcal{P}
    \end{array}
    \right. \right\} 
\end{equation}
where $\mathcal{P}=\left\{ \bm\rho \in \mathcal{K}^\star \mid \bm{B}^\top \bm\rho = \bm{d} \right\}$. %representing the dual uncertainty set. $\bm{\beta}^i$ replacing $\bm{y}^i$ as the \textit{wait-and-see} decisions, $\bm\rho^i$ replacing $\bm{z}$ as an uncertain vectors.
\end{theorem}

\begin{proof}{Proof.}
First, it follows from~\eqref{opt:cone_satisficing} that 
\begin{equation*}
\begin{aligned}
    (\bm{x},\upsilon,k) \in \mathcal{Y} &\;\;\Longleftrightarrow\;\; \max_{\bm{z} \in \mathcal{Z}} \; \min_{\bm{y}} \left\{ \bm{d}^\top \bm{y} - k p( \bm{z}) \mid \bm{B}\bm{y} \succeq_{\mathcal{K}} \bm{f}(\bm{x}) + \bm{F}(\bm{x})\bm{z} \right\} \leq \upsilon. 
    %&\; \Longleftrightarrow \; \max_{\bm{z} \in \mathcal{Z}} \; \max_{\bm{\rho} \in \mathcal{K}^\star} \left\{ (\bm{\rho})^\top \left( \bm{f}(\bm{x}) + \bm{F}(\bm{x})\bm{z} \right) - k \Vert \bm{z} \Vert :\; \bm{B} \bm\rho = \bm{d} \right\} \leq 0. 
\end{aligned}    
\end{equation*}
Under complete and bounded recourse, observe that the inner minimization (over $\bm{y}$) is strictly feasible and its objective is finite. Thus, we can transform it into a maximization problem via conic duality, that is,
\begin{equation*}
\begin{aligned}
    (\bm{x},\upsilon,k) \in \mathcal{Y} &\;\; \Longleftrightarrow \;\; \max_{\bm{z} \in \mathcal{Z}} \; \max_{\bm{\rho} \in \mathcal{P}} \left\{ \bm{\rho}^\top \left( \bm{f}(\bm{x}) + \bm{F}(\bm{x})\bm{z} \right) - k p(\bm{z}) \right\} \leq \upsilon \\
    &\;\; \Longleftrightarrow \;\; \max_{\bm{\rho} \in \mathcal{P}} \left\{ \bm\rho^\top \bm{f}(\bm{x}) + \max_{\bm{z} \in \mathcal{Z}}  \left\{ \bm{\rho}^\top \bm{F}(\bm{x})\bm{z} - k p(\bm{z}) \right\} \right\} \leq \upsilon. 
\end{aligned}    
\end{equation*}
Invoking Proposition~\ref{prop:innermax} to transform the inner maximization (over $\bm{z}$) to a minimization problem (over $\bm\beta$, $\bm \mu$  and $\eta$) completes the proof. \qed 
\end{proof}

In line with the recourse variable $\bm{y}$ and the uncertainty $\bm{z}$ in the original definition \eqref{opt:cone_satisficing} of $\mathcal{Y}$, we may interpret $(\bm\beta, \bm \mu,\eta)$ and $\bm\rho$ as a recourse variable and as an uncertainty, respectively. Inspired by how~\eqref{opt:cone_satisficing_dual} is obtained, we shall refer to $(\bm\beta, \bm \mu,\eta)$ as the dual recourse and to $\mathcal{P}$ as the dual uncertainty set. Under Assumption~\ref{assm:evalutionfunction}, the feasible set $\mathcal{Y}$ as defined in~\eqref{opt:cone_satisficing} is guaranteed to be non-empty, but it is not necessary easy to identify a primal recourse $\bm{y}$ that is compatible with a feasible $(\bm{x},\upsilon,k)$. On the other hand, we will show below that there always exists a feasible dual recourse $(\bm\beta, \bm \mu,\eta)$ that is easy to construct. This observation will then be used to construct the following tractable subset $\bar{\mathcal{Y}}$ of $\mathcal{Y}$:
\begin{equation}
\label{opt:cone_satisficing_dual_ldr}
    \bar{\mathcal{Y}} = \left\{ (\bm{x},\upsilon,k) \in \mathcal{X} \times \mathbb{R} \times \mathbb{R}_+
    \left| 
    \begin{array}{ll}
    \exists \bm\beta \in \mathcal{L}^{n_f, n_h}, \bm \mu \in \mathcal{L}^{n_f, n_\mu}, \eta \in  \mathcal{L}^{n_f, 1}: \\ 
    \bm{\rho}^\top \bm{f}(\bm{x}) + \bm\beta(\bm\rho)^\top \bm{h} + \eta(\bm\rho) \leq \upsilon &\forall \bm\rho \in \mathcal{P} \\
\bm M( \bm{F}(\bm{x})^\top \bm\rho-\bm{H}^\top\bm\beta(\bm\rho)) + \bm N \bm\mu(\bm\rho) + \bm s \eta(\bm\rho) \leq \bm t k~~ &\forall \bm\rho \in \mathcal{P} \\
    \bm\beta(\bm\rho) \geq \bm{0}, \ \eta(\bm\rho) \geq \bm{0} & \forall \bm\rho \in \mathcal{P}
    \end{array}
    \right. \right\} 
\end{equation}
where  $\mathcal{L}^{m,n}$ denotes the sub-class of functions in $\mathcal{R}^{m,n}$ that are affinely dependent on the inputs as follows:
$$
\mathcal{L}^{m,n} =\left\{ \bm y \in \mathcal{R}^{m,n} \left|  
\begin{array}{l}
\exists \bm \pi \in \bbr^n, \bm \Pi \in \bbr^{n\times m}:~\bm y(\bm z) = \bm \pi +  \bm \Pi \bm z\\
\end{array}
\right.\right\}. 
$$
\begin{theorem}
\label{thm:satisficing_dual_ldr_feasiblity}
    Under Assumptions~\ref{assm:evalutionfunction} \& \ref{assm:penaltysupport}, $\bar{\mathcal{Y}}$ is a tractable safe approximation of $\mathcal{Y}$.
\end{theorem}

\begin{proof}{Proof.}
    For any $\bm{x} \in \mathcal{X}$ and $\upsilon > g(\bm{x},\bm{0})$, we first show that the solution $\bm \beta(\bm \rho) = \bm 0$, and $\eta(\bm \rho) = 0$ robustly satisfies the requirements in~\eqref{opt:cone_satisficing_dual_ldr}. With vanishing $\bm\beta$ and $\eta$, we first observe that 
    \begin{equation*}
    \begin{aligned}
         \max_{\bm\rho \in \mathcal{P}} \left\{ {\bm\rho}^\top  \bm{f}(\bm{x}) + \bm\beta(\bm\rho)^\top\bm{h} + \eta(\bm\rho) \right\} &=  \min_{\bm{y}} \left\{  \bm{d}^\top \bm{y} \mid \bm{B} \bm{y} \succeq_{\mathcal{K}} \bm{f}(\bm{x})\right\} 
        =g(\bm{x},\bm{0}) < \upsilon,
    \end{aligned}
    \end{equation*}
    where the first equality is due to the strong duality (as the minimization problem is strictly feasible). %established in the proof of Lemma~\ref{lemma:k0} 
    %and the~second equality holds because of the optimality of the recourse variable $\hat{\bm{y}}$ given the decision $\hat{\bm{x}}$ in~{\color{red}\eqref{opt:nominaldualaffine3.2}}. 
    
    We next show that the remaining constraints can be satisfied because there exists $\hat{k} > 0$ such that 
    \begin{equation}
    \label{ineq:penalty}
    \bm M\bm{F}(\bm{x})^\top \bm\rho + \bm N \hat{\bm \mu}\hat{k} \leq \bm t\hat{k}  \qquad   \forall \bm\rho \in \mathcal{P}. 
    \end{equation}
Indeed, since $\bm N \hat{\bm \mu} < \bm t$, it implies that the set 
$\{ \bm \lambda \mid \bm M \bm \lambda + \bm N \hat{\bm \mu} \leq \bm t  \}$ must contain the origin in its interior. Hence, there exists a norm $\|\cdot \|$ such that 
$$
\{ \bm \lambda \mid \| \bm \lambda\| \leq 1  \} \subseteq \{ \bm \lambda \mid \bm M \bm \lambda + \bm N \hat{\bm \mu} \leq \bm t  \}.
$$
Therefore, it suffices to show that there exists a finite $\hat{k} > 0 $ such that 
$$
\max_{\bm\rho \in \mathcal{P}}  \Vert \bm{F}({\bm{x}})^\top \bm\rho \Vert \leq \hat{k}.
$$
Suppose that the dual uncertainty set $\mathcal{P}$ is unbounded for the sake of a contradiction. Then, there exists a vector $\bm{v} \in \mathbb{R}^{n_f}$ such that $\max_{\bm\rho \in \mathcal{P}} \;  {\bm\rho}^\top \bm{v}$ is unbounded, and therefore its corresponding dual $\min_{\bm y} \left\{ \bm{d}^\top\bm{y} \mid \bm{B} \bm{y} \succeq_{\mathcal{K}} \bm{v} \right\}$ must be infeasible, which contradicts with Assumption~\ref{assm:evalutionfunction}. Hence, $\mathcal{P}$ is a bounded set, and a finite and positive $\hat{k}$ exists, which in turns implies that $(\bm{x},\upsilon,\hat{k}) \in \bar{\mathcal{Y}}$.  
     
Next, to show that $\bar{\mathcal{Y}}$ is a tractable set, observe that since the dual recourse variables $(\bm\beta,\bm\mu,\eta)$ are restricted to affine functions, we can express the requirements in~\eqref{opt:cone_satisficing_dual_ldr} more compactly as
\begin{equation*}
\begin{aligned}
    \exists \bm{L} \in \bbr^{(n_h+n_\mu+1) \times (n_f+1)} \ \forall \bm\rho \in \mathcal{P}: \ \  \bm\gamma(\bm{x},\upsilon,k,\bm{L}) + \bm\Gamma(\bm{x},\upsilon,k,\bm{L})\bm\rho \leq \bm{0} 
\end{aligned}    
\end{equation*}
where $\bm{L} $ gathers the affine dual recourse adaption coefficients of $(\bm\beta, \bm\mu, \eta)$ and $\bm\gamma,\bm\Gamma$ are appropriate affine mappings. Under Assumption~\ref{assm:evalutionfunction}, for any $\bm\gamma \in \bbr^{n_\gamma}$ and $\bm\Gamma \in \bbr^{n_\gamma \times n_f}$, the robust counterpart of $\bm\gamma + \bm\Gamma \bm\rho \leq \bm{0}, \; \forall \bm\rho \in \mathcal{P}$ is given by the following linear conic constraint 
\begin{equation*}
    \exists\bm{V} \in \bbr^{n_\gamma \times n_y}: \left\{ \begin{array}{l}
        \bm\gamma + \bm{V}\bm{d} \leq \bm{0} \\
        \bm{B} (\bm{V}_i)^\top \succeq_{\ck} \bm\Gamma_i^\top \quad \forall i \in [n_\gamma].
    \end{array} \right. 
\end{equation*}
Indeed, observe that the given robust constraint can be written down as
\begin{equation*}
    \max_{\bm\rho \in \mathcal{P}} \; \gamma_i + \bm\Gamma_i \bm\rho \leq 0 \quad \forall i \in [n_\gamma]
    \;\; \Longleftrightarrow \;\;
    \min_{\bm{v}^i \in \bbr^{n_y}} \left\{ \gamma_i + \bm{d}^\top\bm{v}^i \mid \bm{B}\bm{v}^i \succeq_{\mathcal{K}} \bm\Gamma_i^\top \right\}  \leq 0 \quad \forall i \in [n_\gamma],
\end{equation*} 
where the equivalence holds because the minimization problem is convex and strictly feasible. We then denote $\left[ \bm{v}^1, \hdots, \bm{v}^{n_\gamma} \right]^\top$ by $\bm{V}$. Since $\bar{\mathcal{Y}}$ constitutes a tractable set and is increasingly larger with $k$ and $\upsilon$, we conclude that it is indeed a safe tractable approximation of $\mathcal{Y}$ as desired. 
\qed 
\end{proof}

We now consider a special, but an important case where the tractable safe approximation is exact. 
\begin{theorem}
\label{thm:datadriven_dual_satisficing_special}
Suppose that the evaluation function $g$ can be expressed as 
\begin{equation}
\label{eq:gomega_special}
   g(\bm x,\bm z) = \min_{y \in \bbr}\left\{ y~|~ \bm{1} y \geq \bm{f}(\bm{x}) + \bm{F}(\bm{x})\bm{z} \right\},
\end{equation}
then $\bar{\mathcal Y} = \mathcal Y$. 
\end{theorem}
\begin{proof}{Proof.}
Observe that the dual uncertainty set is now a simplex, $\mathcal{P}=\{ \bm \rho \in \bbr^{n_f}_+ \mid \bm 1^\top \bm \rho = 1\}$. It has well been shown \citep[see, {\em e.g.},][]{zhen2018adjustable} that if the uncertainty set is a simplex, then approximation via affine recourse adaption is exact.  
\qed 
\end{proof}

We can now present the tractable safe approximation of the data-driven robust satisficing problem~\eqref{eq:data-drivenRS}, noting that for each $\omega \in [\Omega]$
$$
 \mathcal{Z}^\omega = \{ \bm{\zeta} \in \mathbb{R}^{n_z} \mid  \bm{H}\bm{\zeta} \leq \bm{h}^\omega\}
$$
where $\bm h^\omega = \bm{h}-\bm{H}\hat{\bm z}^\omega$, and the evaluation function $g^\omega:\cx\times \cz_\omega \rightarrow \bbr$,
\begin{equation}
\label{eq:gomega}
\begin{aligned}
   g^\omega(\bm x,\bm z) = \; &\text{minimize} && \bm{d}^\top \bm{y} \\
    &\text{subject to} && \bm{B}\bm{y} \succeq_{\mathcal{K}} \bm{f}^\omega(\bm{x}) + \bm{F}(\bm{x})\bm{z} \\
    & && \bm{y}\in \mathbb{R}^{n_y},
\end{aligned}  
\end{equation}
where $\bm{f}^\omega(\bm{x}) = \bm{f}(\bm{x})  + \bm{F}(\bm{x})\hat{\bm{z}}^\omega$. 

Using affine dual recourse adaptation that we have derived earlier, the tractable safe approximation of the data-driven robust satisficing in the form of  Problem~\eqref{eq:data-drivenRS4} is given by
\begin{equation}
\label{opt:satisficing_dd_dual}
\begin{aligned}
    &{\min} && k  \\
    &{\rm s.t.} &&  \frac{1}{\Omega} \sum_{\omega \in [\Omega]}\upsilon_\omega \leq \tau 
      \\
     &&&   {\bm\rho}^\top \bm{f}^\omega(\bm{x})  + \left(\bm{h}^\omega \right)^\top \bm\beta^\omega (\bm\rho) + \eta^\omega(\bm\rho)\leq \upsilon_\omega  && \hspace{-30mm} \forall \bm{\rho}  \in \mathcal{P}, \omega \in [\Omega]  \\
     & &&  \bm M\left( \bm{F}(\bm{x})^\top \bm\rho-\bm{H}^\top\bm\beta^\omega (\bm\rho) \right)+ \bm N \bm \mu^\omega(\bm\rho) + \bm{s} \eta^\omega(\bm\rho) \leq \bm t k && \hspace{-30mm}\forall \bm{\rho}  \in \mathcal{P}, \omega \in [\Omega] \\
     & && \bm\beta^\omega (\bm\rho) \geq \bm{0},\;\eta^\omega(\bm\rho)\geq 0 && \hspace{-30mm} \forall \bm\rho \in \mathcal{P}, \omega \in [\Omega] \\
    & && \bm{x} \in \mathcal{X},\;  k \in \mathbb{R}_+,\; \bm \upsilon \in \bbr^{\Omega}, \;\bm{\beta}^1, \hdots, \bm\beta^\Omega \in \mathcal{L}^{n_{f}, n_{h}},\; \bm\mu^1, \hdots, \bm\mu^\Omega \in \mathcal{L}^{n_{f}, n_\mu},\; \eta^1, \hdots,\eta^\Omega \in \mathcal{L}^{n_{f}, 1}
\end{aligned}
\end{equation}
where $\mathcal{P}=\left\{ \bm\rho \in \mathcal{K}^\star \mid \bm{B}^\top \bm\rho = \bm{d} \right\}$. 

This result is computationally significant since, despite the difficulty to solve it exactly, if $\tau > Z_0$, thanks to Proposition~\ref{prop:RSfeasibility} we can still obtain a feasible solution of Problem~\eqref{eq:data-drivenRS} by solving a modest sized conic optimization problem.

%As the consequence of Theorem~\ref{thm:datadriven_dual_satisficing_special}, when the affine dual recourse adaptation is used to solve tractable robust satisficing problems in which the constraint function is a saddle function, including biaffine, we will also recover the optimum solution.     

%%%%%%%%%%%%%%%%%%%%%%%%%%%%%%%%%%%%%%%%%%%%%%%%%%%%%%%%%%%%%%%%%%%%%%%%%%%%%%%%%%%%%%%%%%%%%%%%%%%%%%%%%%%%%%%%%%%%%%%%%%%%
\subsection{On linear optimization with recourse}
{\label{appendix:adaptive_linear}}
We revisit the data-driven linear optimization problem of \cite{long2022robust} and explore how the affine dual recourse adaptation can be used to provide a better tractable safe approximation, and with superior computational performance. Specifically, we consider the $\ell_1$-norm penalty function, $p(\bm \zeta) = \Vert \bm \zeta \Vert_1$, and the evaluation function is a linear optimization problem given by 
\begin{equation}\label{opt:nominal_twostage3.3}
\begin{aligned}
g(\bm x,\bm z)=\; &\text{minimize}&& \bm c^\top \bm x +\bm{d}^\top\bm{y}\\
 &   \text{subject to}&& \bm{B}\bm{y} \geq  \bm{f}(\bm{x})+ \bm{F}(\bm{x})\bm{z}\\
&&& \bm{y} \in \mathbb{R}^{n_{y}},
\end{aligned}
\end{equation} 
and we can express the feasible set from~\eqref{eq:corerobustcounterpart} as
\begin{equation}
\label{opt:nonneg_orthantcone_satisficing}
\mathcal Y = \left\{ (\bm x, \upsilon , k) \in \cx \times \bbr \times \bbr_+ \left|
\begin{array}{ll}
\exists  \bm{y}\in \mathcal{R}^{n_{z},n_{y}}:\\
  \bm{c}^\top \bm{x}+\bm{d}^\top\bm{y}(\bm{z}) -  k \Vert \bm z\Vert_1 \leq \upsilon   & \forall \bm{z} \in \mathcal{Z} \\
\bm{B}\bm{y}(\bm{z}) \geq \bm{f}(\bm{x}) + \bm{F}(\bm{x})\bm{z} & \forall \bm{z} \in \mathcal{Z} 
\end{array} \right. \right\}.
\end{equation}

Generally, two-stage robust linear optimization problems are not tractable and are typically solved approximately via affine primal recourse adaptation, but this would not necessarily lead to a feasible solution \citep{long2022robust}. As a result, we consider a more flexible, non-affine primal recourse adaptation extension with extra coefficients $\bm{q}^\dag \in \mathbb{R}^{n_y}$:
\begin{equation}
\label{eq:primal_nldr}
    \bm{y}(\bm{z}) = \bm{q} + \bm{Q}\bm{z} + \bm{q}^\dag \Vert \bm{z} \Vert_1,
\end{equation}
which results in the following tractable safe approximation
\begin{equation}
\label{opt:satisficing_nldr}
\bar{\mathcal Y}_P = \left\{ (\bm x, \upsilon , k) \in \cx \times \bbr \times \bbr_+ \left|
\begin{array}{ll}
\exists  \bm{q} \in \mathbb{R}^{n_y}, \; \bm{Q} \in \mathbb{R}^{n_y \times n_z}, \; \bm{q}^\dag \in \mathbb{R}^{n_y}:\\
 \bm{c}^\top \bm{x} +\bm{d}^\top (\bm{q} + \bm{Q}\bm{z} + \bm{q}^\dag \Vert \bm{z} \Vert_1) -k \Vert \bm{z} \Vert_1  \leq \upsilon   & \forall \bm{z} \in \mathcal{Z} \\
 \bm{B} (\bm{q} + \bm{Q}\bm{z} + \bm{q}^\dag \Vert \bm{z} \Vert_1) \geq \bm{f}(\bm{x}) + \bm{F}(\bm{x})\bm{z} & \forall \bm{z} \in \mathcal{Z} \\
 k - \bm{d}^\top \bm{q}^\dag \geq 0, \; \bm{B}\bm{q}^\dag \geq \bm{0} \\
 \end{array} \right. \right\},
\end{equation}
where the two linear constraints $k - \bm{d}^\top \bm{q}^\dag \geq 0$ and $\bm{B}\bm{q}^\dag \geq \bm{0}$ are added to ensure that the two robust constraints are concave in the uncertain $\bm{z}$ and consequently tractability of the safe approximation. 
\begin{proposition}
\label{prop:satisficing_ldr_feasibility}
Under Assumption~\ref{assm:evalutionfunction}, $\bar{\mathcal Y}_P$ is a tractable safe approximation of $\mathcal Y$, and has the following explicit representation 
\begin{equation}
\label{opt:satisficing_nldr_rc}
\bar{\mathcal Y}_P = \left\{ (\bm x, \upsilon , k) \in \cx \times \bbr \times \bbr_+ \left|
\begin{array}{ll}
\exists \bm{q} \in \mathbb{R}^{n_y}, \; \bm{Q} \in \mathbb{R}^{n_y \times n_z}, \; \bm{q}^\dag \in \mathbb{R}^{n_y}, \bm{w}^0, \hdots, \bm{w}^{n_f} \in \mathbb{R}^{n_h}_+:\\
 \bm c^\top \bm x+ \bm{d}^\top \bm{q} + \bm{h}^\top \bm{w}^0  \leq \upsilon    \\
  -(k -\bm{d}^\top \bm{q}^\dag) \bm{1} \leq   \bm{H}^\top \bm{w}^0 - \bm{Q}^\top \bm{d} \leq (k -\bm{d}^\top \bm{q}^\dag) \bm{1}\\
  \bm{f}_i(\bm{x}) + \bm{h}^\top \bm{w}^i  \leq \bm{B}_i \bm{q}^\dag & \hspace{-6mm}\forall i \in [n_f] \\
  -\bm{B}_i \bm{q}^\dag  \bm{1} \leq \bm{H}^\top \bm{w}^i - \bm{F}^\top_i (\bm{x}) + \bm{Q}^\top \bm{B}^\top_i \leq \bm{B}_i \bm{q}^\dag  \bm{1} & \hspace{-6mm}\forall i \in [n_f] %\\
  % (\bm x, \upsilon , k) \in \cx \times \bbr \times \bbr_+
 \end{array} \right. \right\}.
\end{equation}
%Additionally when $n_y=1$, Problem~\eqref{opt:satisficing_nldr} yields the optimal solution to the robust satisficing Problem \eqref{opt:nonneg_orthantcone_satisficing}. 
 \end{proposition}

\begin{proof}{Proof.}
Similar results have been shown in \citep{long2022robust}, and for completeness we provide an alternate proof in Appendix \ref{appendix:proof_of_results}.
\end{proof}

Next, we will similarly analyze the dual robust satisficing by considering the following equivalent dual tractable safe approximation from~\eqref{opt:cone_satisficing_dual_ldr} as follows 
\begin{equation}
\label{opt:satisficing_dual_ldr}
\bar{\mathcal Y}_D = \left\{ (\bm x, \upsilon , k) \in \cx \times \bbr \times \bbr_+ \left|
\begin{array}{ll}
\exists  \bm\pi \in \mathbb{R}^{n_h}, \; \bm{\Pi} \in \mathbb{R}^{n_h \times n_f}:\\
\bm{c}^\top \bm{x}+  \bm\rho^\top \bm{f}(\bm{x}) + (\bm\pi + \bm\Pi \bm\rho)^\top\bm{h} \leq \upsilon & \forall \bm{\rho} \in \mathcal{P}  \\
-k \bm{1} \leq \bm{H}^\top (\bm\pi + \bm\Pi \bm\rho) - \bm{F}(\bm{x})^\top \bm\rho  \leq  k \bm{1} & \forall \bm{\rho} \in \mathcal{P} \\
\bm\pi + \bm\Pi \bm\rho \geq 0& \forall \bm{\rho} \in \mathcal{P} \\  
 \end{array} \right. \right\},
\end{equation}
which we have earlier shown is feasible due to complete recourse, for dual uncertainty set $\mathcal{P}$ that is $\left\{ \bm\rho \in \mathbb{R}^{n_f}_+ \mid \bm{B}^\top \bm\rho = \bm{d} \right\}$ and for $\bm\pi$ and $\bm\Pi$ that are the intercept and the gradient of $\bm\beta(\bm\rho)$ with respect to $\bm\rho$. Our objective here is therefore to show that, despite the restricted linearity of $\bm\beta(\bm\rho)$, the dual tractable safe approximation is a tighter approximation than the primal safe approximation.
This observation provides a fresh perspective, which goes in a direction opposite to~\cite{Bertsimas2016}, that the primal and the dual affine adaptations are not necessarily equivalent approximations.

\begin{theorem}\label{thm:dual_satisficing_lowerbound}
Under Assumption~\ref{assm:evalutionfunction},
$$
\bar{\mathcal Y}_P \subseteq \bar{\mathcal Y}_D. 
$$
\end{theorem}
\begin{proof}{Proof.}
%\begin{proof}{Proof of Theorem~\ref{thm:dual_satisficing_lowerbound}.}
First of all, we compare the feasible sets~\eqref{opt:satisficing_nldr} and~\eqref{opt:satisficing_dual_ldr} when $\bm{x}$ is fixed to $\bm{x}' \in \mathcal{X}$. We can then abbreviate $\bm{f}(\bm{x}')$ and $\bm{F}(\bm{x}')$ to $\bm{f}'$ and $\bm{F}'$, respectively. Besides, we let $\upsilon'$ denote the value of $\upsilon  - \bm{c}^\top\bm{x}'$. Equivalently, it suffices to show that 
\iffalse 
\begin{equation}
\label{opt:satisficing_dual_ldr_fix_x}
\begin{aligned}
    &\text{minimize} && k \\
    &\text{subject to} && \bm\rho^\top \bm{f}' + (\bm\pi + \bm\Pi \bm\rho)^\top\bm{h} \leq \upsilon' && \forall \bm{\rho} \in \mathcal{P} \\
    &&&-k \bm{1} \leq \bm{H}^\top (\bm\pi + \bm\Pi \bm\rho) - (\bm{F}')^\top \bm\rho  \leq  k \bm{1}&& \forall \bm{\rho} \in \mathcal{P} \\
    & && \bm\pi + \bm\Pi \bm\rho \geq 0  && \forall \bm{\rho} \in \mathcal{P} \\
%     &&&\bm{c}^\top \bm{x} \leq \tau'\\
    & && \bm\pi \in \mathbb{R}^{n_h}, \; \bm{\Pi} \in \mathbb{R}^{n_h \times n_f}, \; k \in \mathbb{R}_+    
\end{aligned}
\end{equation}
\fi 
\begin{equation}
\label{opt:satisficing_dual_ldr_fix_x}
\bar{\mathcal Y}'_D = \left\{ (\upsilon , k) \in \bbr \times \bbr_+ \left|
\begin{array}{ll}
\exists  \bm\pi \in \mathbb{R}^{n_h}, \; \bm{\Pi} \in \mathbb{R}^{n_h \times n_f}:\\
\bm\rho^\top \bm{f}' + (\bm\pi + \bm\Pi \bm\rho)^\top\bm{h} \leq \upsilon' & \forall \bm{\rho} \in \mathcal{P}  \\
-k \bm{1} \leq \bm{H}^\top (\bm\pi + \bm\Pi \bm\rho) - (\bm{F}')^\top \bm\rho  \leq  k \bm{1} & \forall \bm{\rho} \in \mathcal{P} \\
\bm\pi + \bm\Pi \bm\rho \geq 0  & \forall \bm{\rho} \in \mathcal{P} \\  
 \end{array} \right. \right\},
\end{equation}
is a superset of 
\begin{equation}
\label{opt:satisficing_nldr_fix_x}
\bar{\mathcal Y}'_P = \left\{ (\upsilon , k) \in \bbr \times \bbr_+ \left|
\begin{array}{ll}
\exists  \bm{q} \in \mathbb{R}^{n_y}, \; \bm{Q} \in \mathbb{R}^{n_y \times n_z}, \; \bm{q}^\dag \in \mathbb{R}^{n_y}:\\
 \bm{d}^\top (\bm{q} + \bm{Q}\bm{z} + \bm{q}^\dag \Vert \bm{z} \Vert_1) -k \Vert \bm{z} \Vert_1  \leq \upsilon'    & \forall \bm{z} \in \mathcal{Z} \\
 \bm{B} (\bm{q} + \bm{Q}\bm{z} + \bm{q}^\dag \Vert \bm{z} \Vert_1) \geq \bm{f}' + \bm{F}'\bm{z} & \forall \bm{z} \in \mathcal{Z} \\
 k - \bm{d}^\top \bm{q}^\dag \geq 0, \; \bm{B}\bm{q}^\dag \geq \bm{0}
 \end{array} \right. \right\}.
\end{equation}
%If $\upsilon' < \min_{\bm{q}} \{ \bm{d}^\top \bm{q} \mid \bm{B}\bm{q} \geq \bm{f}' \} = \max_{\bm\rho \in \mathcal{P}} \{ \bm\rho^\top \bm{f}'\}$, then there would be no $k$ such that $(\upsilon,k)$ belongs to $\bar{\mathcal Y}'_D$ or $\bar{\mathcal Y}'_P$ because their first respective robust constraint cannot be satisfied (recall that $\bm{0} \in \mathcal{Z}$ and $\bm{h} \geq \bm{0}$). Otherwise if $\upsilon' \geq \min_{\bm{q}} \{ \bm{d}^\top \bm{q}: \bm{B}\bm{q} \geq \bm{f}' \}$, or equivalently if $\upsilon = \upsilon' + \bm{c}^\top\bm{x}' \geq g(\bm{x}',\bm{0})$, both sets must be non-empty thanks to Theorem~\ref{thm:satisficing_dual_ldr_feasiblity} and Proposition~\ref{prop:satisficing_ldr_feasibility}. Henceforth, we will assume that $\upsilon'$ is sufficiently large to avoid trivialities.

For any $(\upsilon,k) \in \bar{\mathcal Y}'_P$, there exists $(\bm{q},\bm{Q},\bm{q}^\dag)$ that satisfies all of the inequality requirements in~\eqref{opt:satisficing_nldr_fix_x}. It then suffices to show that for the same $(\upsilon,k)$, we can construct $\bm\pi$ and $\bm{\Pi}$ such that $(\bm\pi, \bm{\Pi})$ satisfies the inequality requirements in~\eqref{opt:satisficing_dual_ldr_fix_x}. To achieve this, we specifically consider $\bm\pi = \bm{w}^0 \in \mathbb{R}_+^{n_h}$ and $\bm\Pi = \left[ \bm{w}^1, \hdots, \bm{w}^{n_f} \right] \in \mathbb{R}_+^{n_h \times n_f}$, where $\{ \bm{w}^i \}_{i=0}^{n_f}$ satisfy the following conditions.
\begin{subequations}
\begin{eqnarray}
 \bm{d}^\top \bm{q} + \bm{h}^\top \bm{w}^0  \leq \upsilon' \label{eq:a} \\
  -(k - \bm{d}^\top \bm{q}^\dag) \bm{1} \leq \bm{H}^\top \bm{w}^0 - \bm{Q}^\top \bm{d} \leq (k - \bm{d}^\top \bm{q}^\dag) \bm{1} \label{eq:b} \\
 \bm{f}'_i + \bm{h}^\top \bm{w}^i  \leq \bm{B}_i \bm{q} \quad \forall i \in [n_f] \label{eq:c} \\
  -\bm{B}_i \bm{q}^\dag\bm{1} \leq \bm{H}^\top \bm{w}^i - (\bm{F}')^\top_i  + \bm{Q}^\top \bm{B}^\top_i  \leq \bm{B}_i \bm{q}^\dag\bm{1} \quad \forall i \in [n_f] \label{eq:d}
\end{eqnarray}
\end{subequations}
Note that the existence $\{ \bm{w}^i \}_{i=0}^{n_h}$ is guaranteed by Proposition~\ref{prop:satisficing_ldr_feasibility}, which holds when $\mathcal{X} = \{ \bm{x}' \}$. 

We will now show that this choice of $(\bm\pi, \bm{\Pi})$ satisfies all constraints in Problem~\eqref{opt:satisficing_dual_ldr_fix_x}. By a usual duality argument, the first inequality in~\eqref{opt:satisficing_dual_ldr_fix_x} is robustly satisfied for all $\bm\rho \in \mathcal{P}$ if and only if there exists a vector $\bm\nu \in \mathbb{R}^{n_y}$ such that
\begin{equation*}
    \bm\pi^\top \bm{h} + \bm{\nu}^\top \bm{d} \leq \upsilon' \quad \text{and} \quad \bm{B}\bm\nu \geq \bm{f}' + \bm{\Pi}^{\top}\bm{h}.
\end{equation*}
Thanks to the inequalities~\eqref{eq:a} and \eqref{eq:c}, $\bm\nu$ can be simply chosen as $\bm{q}$. 
\iffalse 
Next, the second constraint of~\eqref{opt:satisficing_dual_ldr_fix_x} can be decomposed into two element-wise inequalities
\begin{equation*}
    \bm{H}^\top \bm\pi + ( \bm{H}^\top \bm\Pi  - (\bm{F}')^\top ) \bm\rho \leq k \cdot \bm{1} \quad \text{and} \quad 
    - \bm{H}^\top \bm\pi + ( (\bm{F}')^\top - \bm{H}^\top \bm\Pi ) \bm\rho \leq k \cdot \bm{1}
\end{equation*}
\fi

Similarly, the two subsequent inequalities of~\eqref{opt:satisficing_dual_ldr_fix_x}, namely $-k \bm{1} \leq \bm{H}^\top (\bm\pi + \bm\Pi \bm\rho) - (\bm{F}')^\top \bm\rho  \leq  k \bm{1}$, are robustly satisfied for all $\bm\rho \in \mathcal{P}$ if and only if there exist matrices $\bm\Phi, \bm\Psi \in \mathbb{R}^{n_y \times n_z}$ such that
\begin{equation*}
\left. \begin{array}{c}
    \bm{H}^\top\bm\pi + \bm{\Phi}^\top \bm{d} \leq k \bm{1} \\
    \bm{B}\bm{\Phi} \geq \bm{\Pi}^\top \bm{H} - \bm{F}'
\end{array} \right\}
\quad \text{and} \quad
\left\{ \begin{array}{c}
    -\bm{H}^\top\bm\pi + \bm{\Psi}^\top \bm{d} \leq k \bm{1} \\
    \bm{B}\bm{\Psi} \geq \bm{F}' - \bm{\Pi}^\top \bm{H} .
\end{array} \right.
\end{equation*}
In this case, we can choose, $\bm\Phi = \bm{q}^\dag \bm{1}^\top - \bm{Q}$ and $\bm\Psi = \bm{q}^\dag \bm{1}^\top_{n_z} + \bm{Q}$. Observe that
\begin{equation*}
\begin{aligned}
    &\bm{H}^\top\bm\pi + \bm{\Phi}^\top \bm{d} = \bm{H}^\top\bm{w}^0 + \bm{1}(\bm{q}^\dag)^\top\bm{d} -  \bm{Q}^\top\bm{d} \leq k  \bm{1} \\
    &\bm{B}\bm{\Phi} - \bm{\Pi}^\top \bm{H} + \bm{F}' = \left[ \bm{B}_i \bm{q}^\dag \bm{1}^\top - \bm{B}_i\bm{Q} - (\bm{w}^i)^\top \bm{H} + \bm{F}'_i \right]_{i=1}^{n_f} \geq \bm{0},
\end{aligned}
\end{equation*}
where the inequalities holds due to the inequalities~\eqref{eq:b} and~\eqref{eq:d}, as desired. Similarly,
\begin{equation*}
\begin{aligned}
    &-\bm{H}^\top\bm\pi + \bm{\Psi}^\top \bm{d} = -\bm{H}^\top\bm{w}^0 + \bm{1}(\bm{q}^\dag)^\top\bm{d} +  \bm{Q}^\top\bm{d} \leq k  \bm{1} \\
    &\bm{B}\bm{\Psi} + \bm{\Pi}^\top \bm{H} - \bm{F}' = \left[ \bm{B}_i \bm{q}^\dag \bm{1}^\top + \bm{B}_i\bm{Q} + (\bm{w}^i)^\top \bm{H} - \bm{F}'_i \right]_{i=1}^{n_f} \geq \bm{0}.
\end{aligned}
\end{equation*}
Finally, the last requirement in~\eqref{opt:satisficing_dual_ldr_fix_x} is trivially satisfied because of the non-negativity of $\bm\pi$, $\bm\Pi$ and $\bm\rho$, by construction. Hence, $\bar{\mathcal{Y}}'_P \subseteq \bar{\mathcal{Y}}'_D$, and the theorem follows. 
\qed 
%\end{proof}
\end{proof}

%As the consequence of Theorem~\ref{thm:dual_satisficing_lowerbound}, when the affine dual recourse adaptation is used to solve tractable robust satisficing problems in which $g$ is a saddle function including biaffine functions, we will also recover the optimum solution.

Finally, to illustrate the distinction between the primal non-affine and the dual affine adaptions of a two-stage robust satisficing linear optimization problem, we consider the following specific instance
\begin{equation}
\label{opt:teststrict}
\begin{aligned}
    &\text{minimize} && k \\
    &\text{subject to} && \bm{d}^\top \bm{y}(\bm{z}) \leq \tau + k \Vert \bm{z} \Vert_1 && \forall \bm{z}: \bm{H}\bm{z} \leq \bm{h} \\ %, \bm{z} \geq -100\cdot \bm{1} \\
    & && \bm{B}\bm{y}(\bm{z}) \geq \bm{f} + \bm{F}\bm{z} && \forall \bm{z}: \bm{H}\bm{z} \leq \bm{h} \\ %, \bm{z} \geq -100\cdot\bm{1},
    & &&\bm{y} \in \mathcal{R}^{n_y, n_z}, \; k \in \mathbb{R}_+
\end{aligned}    
\end{equation}
for the parameters listed in Appendix~\ref{appendix:testrestrictparam}. 
The chosen recourse matrix $\bm{B}$ is non-negative and numerically satisfies $\min_{\bm{y}} \left\{ \bm{d}^\top\bm{y} \vert \ \bm{B}\bm{y} \geq \bm{0} \right\}=0$ and Assumption~\ref{assm:evalutionfunction} is thus satisfied. For different values of the target objective $\tau$ that are sufficiently large, Figure~\ref{fig:p_vs_d} suggests that the dual affine adaptation is an approximation that is less conservative compared to its primal non-affine counterpart. 
\begin{figure}[hbt!]
    \centering 
    \includegraphics[scale=0.65]{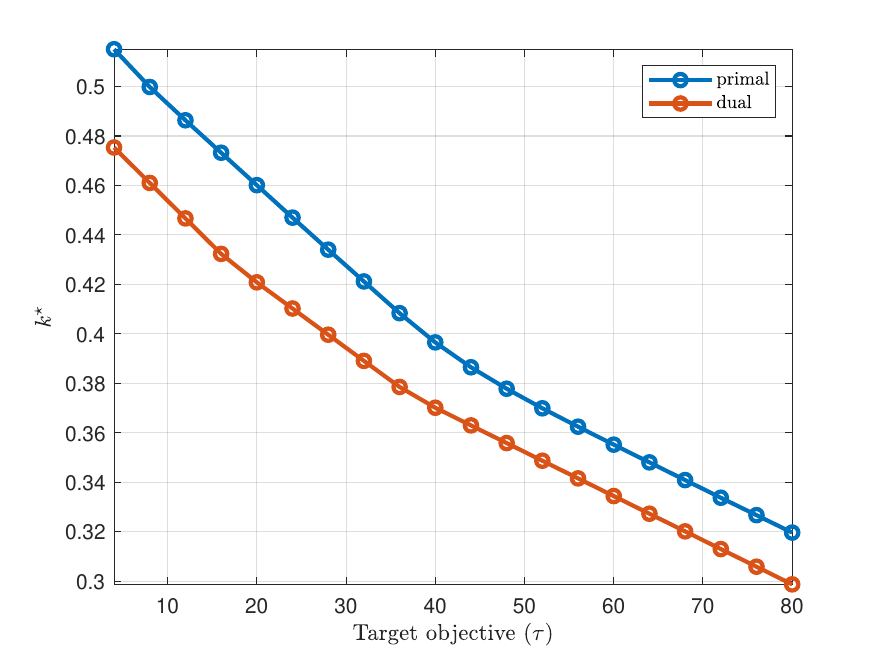}
    \caption{Robust satisficing objectives obtained from the primal non-affine and the dual affine adaptations.}
    \label{fig:p_vs_d}
\end{figure}

In Appendix~\ref{subsec:Lot-sizing}, we present an additional computational study of solving various robust network lot-sizing problems. Focusing on evaluating the computational speed, we show that the dual affine adaptation method again significantly outperforms the primal method in solving both the robust optimization and satisficing models.

%%%%%%%%%%%%%%%%%%%%%%%%%%%%%%%%%%%%%%%%%%%%%%%%%%%%%%%%%%%%%%%%%%%%%%%%%%%%%%%%%%%%%%%%%%%%%%%%%%%%%%%%%%%%%%%%%%%%%%%%%%%%
\section{Portfolio optimization via expected utility maximization}{\label{sec:tractable examples}}
In our numerical study, we consider a data-driven portfolio optimization problem where $\bmt z$ denotes a random vector of asset returns, with distribution $\bmt z\sim \bbp^\star$, $\bbp^\star \in \mathcal P_0(\cz)$. Ideally, the portfolio manager would like to maximize the certainty equivalent of the portfolio returns, \emph{i.e.},
\begin{equation*}
\begin{aligned}
    &\text{maximize} && u_a^{-1}\left(\epstar{ u_a(\bm{x}^\top \tilde{\bm{z}}) }\right) \\
    &\text{subject to} && \bm{x} \in \mathcal{X},
\end{aligned}
\end{equation*}
over the feasible set $\cx=\{\bm x \in \bbr^n | \bm{1}^\top \bm{x} = 1, \bm{x} \geq \bm{0} \}$, and based on an exponential utility with risk aversion $a>0$, {\em i.e.}, 
$u_a(r) =  (1-\exp(-ar))/a$. However, the actual distribution $\bbp^\star$ is unobservable. Instead, we are only given $\Omega$ historical returns given by $\bmh z^1,\dots,\bmh z^\Omega$. 

For the empirical optimization model, we minimize the expected disutility of the portfolio loss as follows:
\begin{equation}\label{eq:portemp}Z_0= \min_{\bm x \in \mathcal X} \ephat{\ell_a(-\bm{x}^\top \bmt{z})}  = \min_{\bm x \in \mathcal X} \frac{1}{\Omega} \sum_{ \omega \in [\Omega]} \ell_a(-\bm{x}^\top \bmh{z}^\omega),
\end{equation}
where $\ell_a(w) = -u_a(-w) = (\exp(aw)-1)/a$ is the disutility defined on loss $w$.

For the stochastic optimization model, we assume a multivariate normal distribution,  $\bbp^\dag = \mathcal N( \bmh \mu, \bmh \Sigma )$. The assumption of normal distribution implies that these parameters  can be estimated by maximizing the likelihood function so that $\bmh \mu = \ephat{\bmt r}$ and $\bmh \Sigma = \ephat{(\bmt r- \bmh \mu)(\bmt r- \bmh \mu)^\top}$, which are computed over the $\Omega$ empirical samples. However, asset returns frequently exhibit
non-normal behavior. This is particularly true of certain types of assets, such as hedge funds, emerging markets, or assets with credit risk, and is even more evident during financial crises. Under normality, maximizing the certainty equivalent of the portfolio returns under exponential utility is equivalent to solving the celebrated Markowitz mean-variance portfolio optimization problem as follows:    
\begin{equation}\label{eq:portstoc}
\begin{aligned}
    &\text{maximize} && \bm x^\top \bmh \mu -a \frac{\bm x^\top \bmh \Sigma \bm x}{2} \\
    &\text{subject to} && \bm{x} \in \mathcal{X}.
\end{aligned}
\end{equation}
Comparing to the empirical optimization problem, which is an exponential conic optimization problem that depends on the data size, the Markowitz portfolio optimization problem is a fixed and modest sized quadratic optimization problem.  For the robust satisficing model, we assume that the support set is a box $\cz = \{ \bm z | \underline{\bm z} \leq \bm z  \leq \bar{\bm z}\}$, where $\underline{\bm z}$ and $\bar{\bm z}$ are  determined from data.  
The corresponding robust satisficing investment problem can be formulated as
\begin{equation}
\label{eq:satisficingexputil}
\begin{aligned}
    &\text{minimize} && k \\
    &\text{subject to} && \frac{1}{\Omega} \sum_{\omega \in [\Omega]} \sup_{\bm z \in \cz^\omega} \left( g^{\omega}(\bm{x},\bm{z})-k p_\Gamma(\bm{z})\right) \leq \tau  &&\\
    & && \bm{x} \in \mathcal{X}, \; k \in \mathbb{R}_+.
\end{aligned}
\end{equation}
where $p_\Gamma(\bm{z})$ is the budgeted norm (see Example~\ref{example:budgeted_norm}), $\cz^\omega = [\underline{\bm z} - \bmh z^\omega, \bar{\bm z} - \bmh z^\omega]$, and $ g^\omega(\bm x,\bm z)$ is the conic representable evaluation function satisfying 
$g^\omega(\bm x,\bm z) =\ell_a(-\bm x^\top (\bmh z^\omega + \bm z))$ for all $\bm x \in \cx$, $\bm z \in \cz^\omega$. Specifically, from Theorem~\ref{thm:convextoconic}, we express the evaluation function as 

%\notenr{Do we have to explain the interval $[\text{vector},\text{vector}]$?} 
\begin{equation*}
\begin{aligned}
    g^{\omega}(\bm{x},\bm{z}) = \; &\text{minimize} && y_1+Py_2 \\
    &\text{subject to} && 
   \left[ \begin{array}{ccc} 
 0& 0 \\
1&0 \\ 
0&1\\
0&1\\
\end{array}\right]
 \left[ \begin{array}{c}   y_1\\y_2
\end{array}\right] \succeq_{\mathcal{K}}
    \left[ \begin{array}{c} \bm{x}^\top \bmh z^\omega\\  -P \\0
\\1 \end{array}\right]+ \left[ \begin{array}{c}  \bm{x}^\top\\  \bm{0}^{\top} 
\\\bm{0}^{\top} \\\bm{0}^{\top}\end{array}\right]\bm{z}\\
&&& \bm y \in \bbr^2,
\end{aligned}    
\end{equation*}
where $\mathcal{K} = \mathcal K_{\ell_a} \times  \bbr_{+}$. Note that since we are evaluating $\ell_a(-\bm x^\top( \bmh z^\omega + \bm z))$ for $\bm x \in \cx =\{\bm x {\color{black}\in \bbr^n} | \bm{1}^\top \bm{x} = 1, \bm{x} \geq \bm{0}\}$ and $\bm z \in \cz^\omega$, it suffices to choose the featured domain $\mathcal W = [\underline w,\bar w]$ such that
$$
\bar w  \geq \max_{\bm z \in \cz^\omega, \bm x \in \cx}\{ -\bm x^\top (\bmh z^\omega + \bm z)\}   =\max_{\bm z \in [\underline{\bm z}, \bar{\bm z} ] , \bm x \in \cx}\{ -\bm x^\top \bm z\} =-\underset{i \in [n]}{\min}\left\{\underline{z}_i\right\}
$$
and 
$$
\underline w  \leq \min_{\bm z \in \cz^\omega, \bm x \in \cx}\{ -\bm x^\top (\bmh z^\omega + \bm z)\}   =\min_{\bm z \in [\underline{\bm z}, \bar{\bm z} ] , \bm x \in \cx}\{ -\bm x^\top \bm z\} = -\underset{i \in [n]}{\max}\left\{\bar{z}_i\right\}.
$$
From Proposition~\ref{prop:BigM}, we choose
 \begin{equation*}
\begin{array}{rcl}
  P &\geq &  \underset{ v \in \{\underline w,\bar w\}}{\max} \left\{
 \ell_a'( v)v - \ell_a( v)\right\}
 = \underset{v \in \{\underline{w},\bar{w}\}}{\max}\left\{ \exp(a v)\left(v-1/a\right)+1/a \right\},
\end{array}
\end{equation*}
so that
$$
 g^{\omega}(\bm{x},\bm{z}) = \ell_a(-\bm x^\top( \bmh z^\omega + \bm z)) \qquad \forall \bm x \in \cx,\bm z \in \cz^\omega.
$$

\begin{proposition}
\label{prop:portfoliosafeapproximation}
The tractable safe approximation of Problem~\eqref{eq:satisficingexputil} is given by
\begin{equation}
\allowdisplaybreaks
\label{eq:satisficingexputil3}
\begin{array}{rclll}
    &{\rm minimize} & k  \\
    &{\rm~subject~to} &  \displaystyle  \frac{1}{\Omega} \sum_{\omega \in [\Omega]}\upsilon_\omega \leq \tau 
      \\
     &&\displaystyle   \bm{x}^\top \bmh z^\omega \rho_1-P\rho_2 +\rho_4 +  (\bar{\bm{z}}^{\omega})^\top \bar{\bm\beta}^{\omega,0}- (\underline{\bm{z}}^{\omega})^\top \underline{\bm\beta}^{\omega,0} \\
     && \displaystyle \qquad\qquad +\sum_{j \in [4]} \left((\bar{\bm{z}}^{\omega})^\top \bar{\bm\beta}^{\omega,j}- (\underline{\bm{z}}^{\omega})^\top \underline{\bm\beta}^{\omega,j} \right)\rho_j  \leq \upsilon_\omega  &  \forall \bm{\rho}  \in \mathcal{P}, \omega \in [\Omega]  \\
     && \displaystyle \bm{1}^{\top}\bm \mu^{\omega,0} + \sum_{j \in [4]} \bm{1}^{\top}\bm \mu^{\omega,j}\rho_j  \leq \Gamma k& \forall \bm{\rho}  \in \mathcal{P},  \omega \in [\Omega]\\
     && \displaystyle x_i \rho_1  - \bar{\beta}_i^{\omega,0} + \underline{\beta}_i^{\omega,0} - \mu_i^{\omega,0} + \sum_{j \in [4]} \left(- \bar{\beta}_i^{\omega,j} + \underline{\beta}_i^{\omega,j} - \mu_i^{\omega,j}\right)\rho_j \leq  0&  \forall \bm{\rho}  \in \mathcal{P}, \omega \in [\Omega], i \in [n]  \\
      && \displaystyle -x_i \rho_1  + \bar{\beta}_i^{\omega,0} - \underline{\beta}_i^{\omega,0} - \mu_i^{\omega,0} + \sum_{j \in [4]} \left(\bar{\beta}_i^{\omega,j} - \underline{\beta}_i^{\omega,j} - \mu_i^{\omega,j}\right)\rho_j \leq 0 ~~~&  \forall \bm{\rho}  \in \mathcal{P}, \omega \in [\Omega], i \in [n]  \\
     && \displaystyle \mu_i^{\omega,0} +\sum_{j \in [4]}\mu_i^{\omega,j}\rho_j  \leq k&  \forall \bm{\rho}  \in \mathcal{P}, \omega \in [\Omega], i \in[n]  \\
     && \displaystyle \bar{\beta}_i^{\omega,0} +\sum_{j \in [4]}\bar{\beta}_i^{\omega,j}\rho_j  \geq 0&  \forall \bm{\rho}  \in \mathcal{P}, \omega \in [\Omega], i \in[n]  \\
     && \displaystyle \underline{\beta}_i^{\omega,0} +\sum_{j \in [4]}\underline{\beta}_i^{\omega,j}\rho_j  \geq 0&  \forall \bm{\rho}  \in \mathcal{P}, \omega \in [\Omega], i \in[n]  \\
    & & \bm{x} \in \mathcal{X},\;  k \in \mathbb{R}_+,\; \bm \upsilon \in \bbr^{\Omega}, \; \bar{\bm{\beta}}^{\omega,j}, \; \underline{\bm{\beta}}^{\omega,j}, \; \bm\mu^{\omega,j} \in \bbr^{n} & \forall \omega \in [\Omega], j \in \{0,\dots, 4\},
\end{array}
\end{equation}
where $\bar{\bm z}^\omega = \bar{\bm z} - \bmh z^\omega$ and  $\underline{\bm z}^\omega = \underline{\bm z} - \bmh z^\omega$. Every constraint in Problem~\eqref{eq:satisficingexputil3} is linear in the dual uncertainty $\bm\rho$ and its robust counterpart could be derived from the following equivalence:
$$\begin{array}{c}
\forall \bm \rho \in \mathcal P: \bm v^\top \bm \rho \leq w \quad \Longleftrightarrow \quad 
\exists y \in \bbr: \left\{ \begin{array}{l}
(y-v_3)(\exp(-av_1/(y-v_3))-1)/a + v_2 + Py \leq w, \\
y\geq v_3, \\
y\geq v_4.
\end{array}\right.
\end{array}
$$
\end{proposition}
\begin{proof}{Proof.}
   The proof is relegated to Appendix \ref{appendix:proof_of_results}.
   \qed 
\end{proof}

%Suppose the $P$ is large enough and $[\underline w, \bar w]$ to contain 

To evaluate the performance of the portfolio, we consider the distribution $\bbp^\star$, $\bmt z \sim \bbp^\star$, such that each $\tilde{z}_i$ is an independently distributed two-point random variable,
$$
\bbp^\star\left[ \tilde{z}_i = z \right] 
= \left\{ \begin{array}{ll}
\beta_i & \mbox{if $z = \mu_i + \sigma_i\frac{\sqrt{\beta_i(1-\beta_i)}}{\beta_i}$},\\
1-\beta_i & \mbox{if $z = \mu_i - \sigma_i\frac{\sqrt{\beta_i(1-\beta_i)}}{1-\beta_i}$}. 
\end{array}\right. 
$$
Observe that $\epstar{\tilde{z}_i} = \mu_i$ and $\epstar{(\tilde{z}_i-\mu_i)^2}=\sigma_i^2$ represent the mean and the variance of the $i^\text{th}$ asset. The skewness of the $i^\text{th}$ asset would depend on $\beta_i$; higher values imply more negative skewness and result in large losses and small upside gains. We consider $n = 8$,
\begin{equation*}
\begin{aligned}
    \bm\mu &\;=\; \left[ 0.12,\; 0.16,\; 0.14,\; 0.13,\; 0.15, \; 0.12,\; 0.14,\; 0.15 \right]^\top, \\
    \bm\sigma &\;=\; \left[ 0.18,\; 0.22,\; 0.20,\; 0.16,\; 0.14,\; 0.10,\; 0.14,\; 0.19 \right]^\top,
\end{aligned}    
\end{equation*}
and $\beta_i$, $i \in [n]$, is chosen as  
%\begin{array}{rcl}
$\beta_i = \displaystyle\frac{1}{2}\left(1+\frac{i}{n+1}\right)$.
%\end{array}
%for all $i \in [n]$. 
Hence, the return distributions for stocks with
high index numbers in the portfolio are more negatively skewed than those with low index numbers.
For a given $\bm x \in \cx$, we can evaluate the actual certainty equivalent of the portfolio return as follows
\begin{equation}\label{eq:port_certeq}
\begin{aligned}
&u_a^{-1}\left(\epstar{ u_a(\bm{x}^\top \tilde{\bm{z}})) }\right) \\
&\quad= \displaystyle  -\frac{1}{a}\log\left(\epstar{(\exp(-a\bm{x}^\top \bmt z)}\right)\\
&\quad= \displaystyle \bm x^\top \bm \mu -\frac{1}{a} \sum_{i \in [n]} 
 \log\left( \beta_i \exp\left(\frac{-ax_i \sigma_i\sqrt{\beta_i(1-\beta_i)}}{\beta_i}\right) + (1-\beta_i) \exp\left(\frac{ax_i \sigma_i\sqrt{\beta_i(1-\beta_i)}}{1-\beta_i}\right) \right).
\end{aligned}
\end{equation}
 For benchmarking purposes, we first generate $100$ independent realizations of $\tilde{\bm{z}}$ from the two-point distribution as described above. For a given $a \geq 0$, we denote the optimal solutions (computed using the generated samples) of the empirical and stochastic optimization problems in~\eqref{eq:portemp} and ~\eqref{eq:portstoc} by $\bm{x}^{\rm emp}_a$ and $\bm{x}^{\rm stoc}_a$ (also referred to as the Markowitz solution), respectively. We then solve the robust satisficing problem using the safe approximation from~\eqref{eq:satisficingexputil3}, by conservatively choosing the parameters of the support set as
 $$\underline{z}_j = \underset{i \in [n], \omega \in [\Omega]}{\min}\;\left\{{\hat{z}_{i}^{\omega}}\right\}-0.1 \quad \text{and} \quad \bar{ z}_j=\underset{i \in [n], \omega \in [\Omega]}{\max}\;\left\{{\hat{z}_{i}^{\omega}}\right\}+0.1. $$ 
 for all $j \in [n]$, \emph{i.e.}, all $z_j$'s share the same support. Even though this support is unnecessarily large, the robust satisficing solutions, as we will see, are not overly conservative. Then, we assign 
 $\underline{w}= -\bar{z}_1, \bar{w}= -\underline{z}_1$
and choose the budgeted norm parameter $\Gamma \in \{1,3,6,8\}$. For each such $\Gamma$, we select four target risk values above the optimal empirical objective $Z_0$, \emph{i.e.}, $\tau =Z_0+\lambda$ for $\lambda= \{0.005,0.01,0.04,0.07\}$. Denote the respective optimal satisficing solutions by $\bm{x}_{a ,\Gamma,\tau}^{\rm st}$.  Note that while we provide a range of values of $\Gamma$ and $\lambda$ in our numerical studies, in a data-driven environment, these  hyper-parameters should be determined using cross-validation techniques  \citep[see, {\em e.g.},][among others]{esfahani2018data,shafieezadeh2019regularization}. 

All optimization problems are solved by using Mosek 9.2.38 together with YALMIP modeling language \citep{yalmip_2004}. After obtaining $\bm{x}_{a ,\Gamma,\tau}^{\rm st}, \bm{x}^{\rm emp}$ and $\bm{x}^{\rm stoc}$ for varying $a, \Gamma, \tau$, the certainty equivalent of each investment strategy is then computed using~\eqref{eq:port_certeq}. The obtained results are shown in Figure~\ref{fig:portexp}. For $a=1$, it can be observed that the certainty equivalents for the empirical and stochastic models coincide. Indeed, when the investor is mildly  risk-averse, both resulting portfolios invest all wealth in the asset with the highest empirical mean return. The satisficing model, on the other hand, opts for a more diversified investment and yields a superior certainty equivalent for all values of $\Gamma$ and $\tau$. The performance gaps between the satisficing model and the empirical and the stochastic models appear to increase with $a$. Amongst these three different approaches, the stochastic model yields the most disappointing certainty equivalents because it suffers from trusting the samples too much and from imposing an incorrect shape for the asset return distribution.

\iffalse 
\begin{figure}[H]
    \includegraphics[scale=0.5,valign=b]{figures/bnorm=1_fourtauvalues.pdf}
    \includegraphics[scale=0.5,valign=b]{figures/bnorm=3_fourtauvalues.pdf} \\
    \includegraphics[scale=0.5,valign=t,]{figures/bnorm=6_fourtauvalues.pdf}
   \includegraphics[scale=0.5,valign=t]{figures/bnorm=8_fourtauvalues.pdf}
    \caption{Certainty equivalent comparison of three models for $\Gamma=1,3,6,8$}
    \label{fig:portexp}
\end{figure}
\fi 

\begin{figure}[H]
    \includegraphics[scale=0.55]{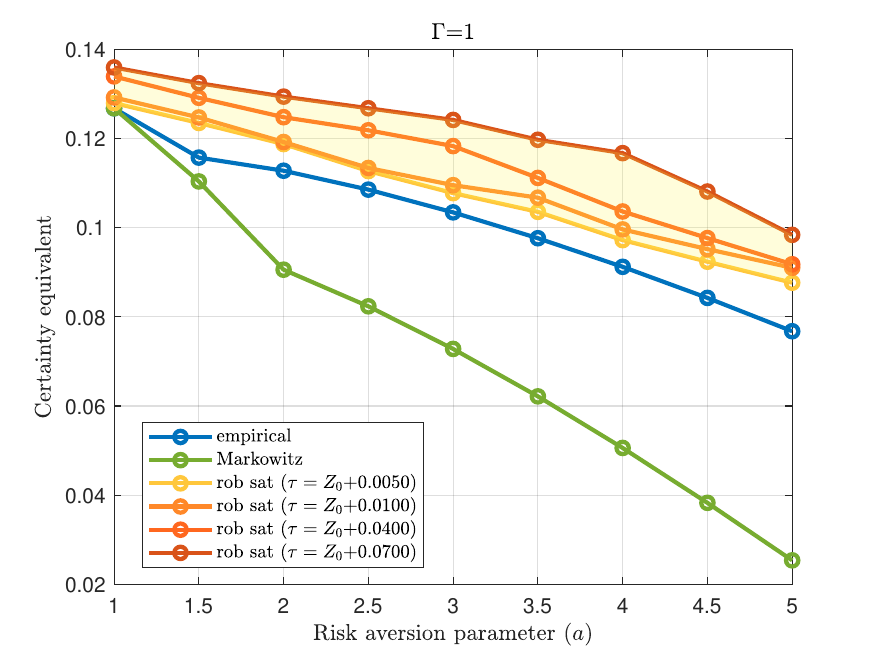}
    \includegraphics[scale=0.55]{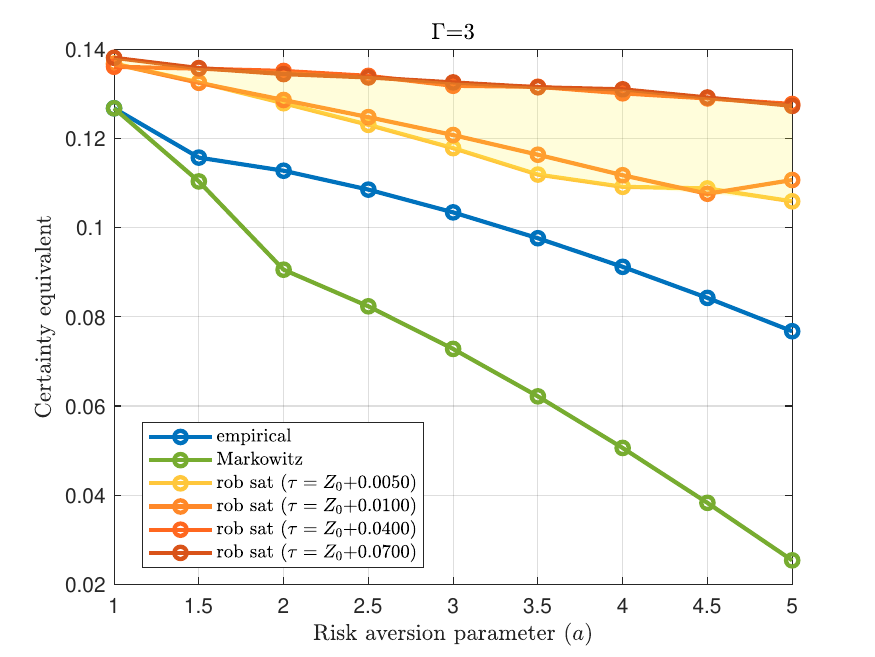} \\
    \includegraphics[scale=0.55]{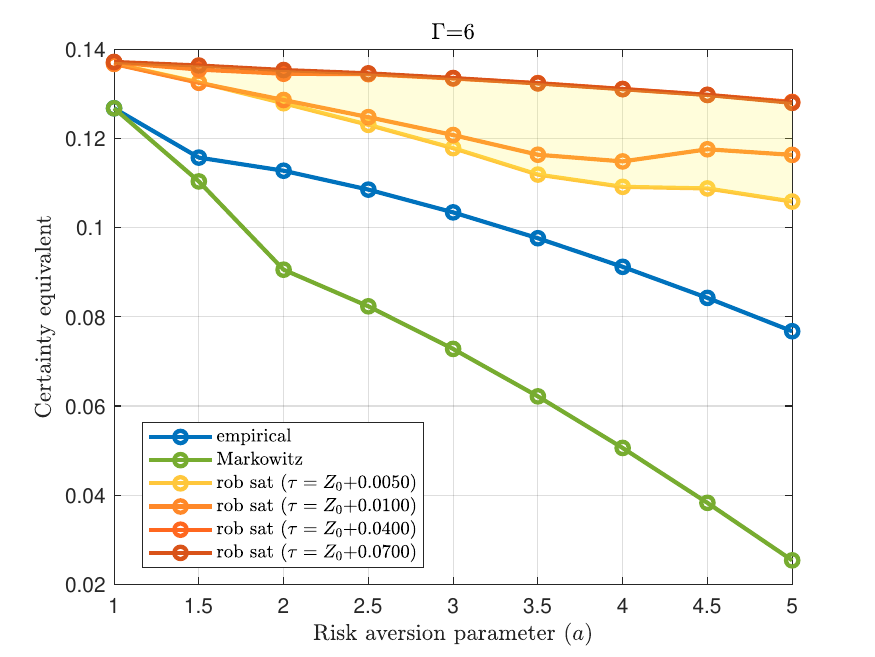}
    \includegraphics[scale=0.55]{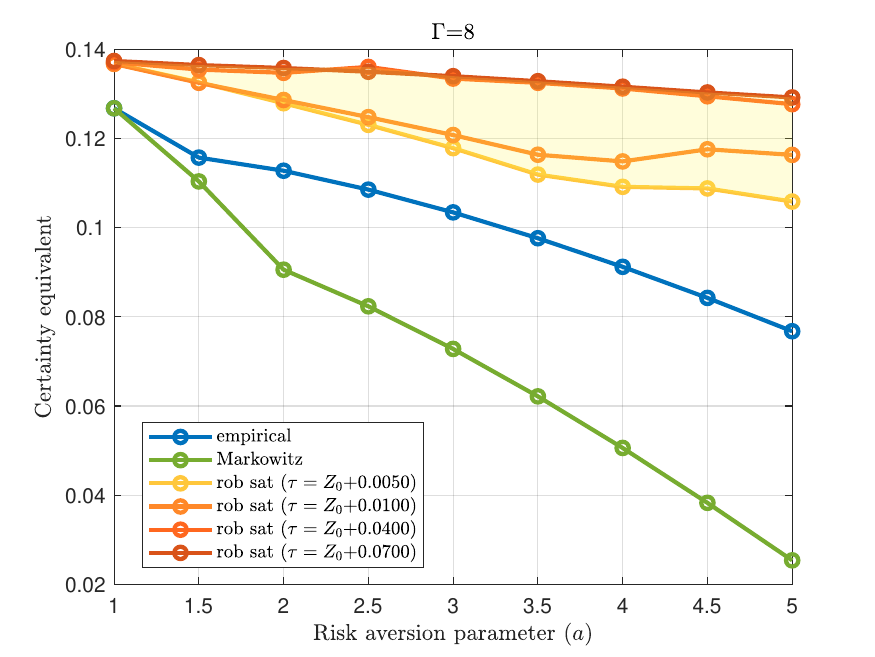}
    \caption{Certainty equivalent comparison between the empirical, stochastic (Markowitz) and robust satisficing investment models.}
    \label{fig:portexp}
\end{figure}

\paragraph{Acknowledgements:} We are grateful to Jianzhe Zhen for sharing the source code used in \cite{Roos2020} and Peng Xiong for the valuable discussions regarding RSOME. %This research is supported by the Ministry of Education, Singapore, under its 2019 Academic Research Fund Tier 3 grant call (Award ref: MOE-2019-T3-1-010). 

%%%%%%%%%%%%%%%%%%%%%%%%%%%%%%%%%%%%%%%%%%%%%%%%%%%%%%%%%%%%%%%%%%%%%%%%%%%%%%%%%%%%%%%%%%%%%%%%%%%%%%%%%%%%%%%%%%%%%%%%%%%%
\bibliography{bibliography}
\bibliographystyle{ormsv080}

%%%%%%%%%%%%%%%%%%%%%%%%%%%%%%%%%%%%%%%%%%%%%%%%%%%%%%%%%%%%%%%%%%%%%%%%%%%%%%%%%%%%%%%%%%%%%%%%%%%%%%%%%%%%%%%%%%%%%%%%%%%%

\newpage
\renewcommand{\theHsection}{A\arabic{section}}

%%%%%%%%%%%%%%%%%%%%%%%%%%%%%%%%%%%%%%%%%%%%%%%%%%%%%%%%%%%%%%%%%%%%%%%%%%%%%%%%%%%%%%%%%%%%%%%%%%%%%%%%%%%%%%%%%%%%%%%%%%%%
\begin{APPENDICES}
\section{Proofs of Results}\label{appendix:proof_of_results}

\begin{proof}
{Proof of Proposition~\ref{prop:RObound}.}

	According to the Chernoff bound, for all $t>0$, we have
	\begin{equation*}
		\displaystyle \bbp\left[\sum_{i\in[n_z]}|\tilde{z}_i| \leq (1-\epsilon)\mu  \right] \leq \exp(t(1-\epsilon)\mu ) \prod_{i\in[n_z]}\ep{\exp(-t|\tilde{z}_i|)},
	\end{equation*}
	where 
	$$
	\mu = \sum_{i \in [n_z]} \ep{|\tilde z_i|} \geq \sum_{i \in [n_z]} \ep{\tilde z^2_i}= n_z\theta. 
	$$
	Notice that $|\tilde{z}_i|$ has a support set $[0,1]$. Let $\tilde{v}_i \sim \mbox{Bernoulli}(p_i)$ be a Bernoulli random variable with the same mean as $|\tilde{z}_i|$. Then,
	\begin{equation}
	\label{eqn:mgf_bound}
		\ep{\exp(-t|\tilde{z}_i|)} \leq \ep{\exp(-t\tilde{v}_i)}.
	\end{equation}
	We first prove the above inequality. For any random variable $\tilde{v}$ with mean $\hat{v}$ and a support set $[0,1]$, we consider the following distributionally robust optimization problem: 
	\begin{equation*}
		\begin{array}{rcll}
			Z(t) = & \displaystyle \sup_{\bbp\in\cq} & \displaystyle \ep{\exp(-t\tilde{v})},
		\end{array}
	\end{equation*}
	where the ambiguity set is $\cq :=\{\bbp\in\cp_0([0,1]) ~|~ \tilde{v}\sim \bbp, ~\ep{\tilde{v}}=\hat{v}\}$. By weak duality, we have $Z(t)\leq Z_d(t)$, where 
	\begin{equation*}
		\begin{array}{rcll}
			Z_d(t) = & \min & r_0 + r_1\hat{v} &
			\\
			& {\rm s.t.} & r_0 \geq \exp(-tv) - r_1v & ~~\forall v\in[0,1] \\
			& & r_0, r_1 \in \mathbb{R}.
		\end{array}
	\end{equation*}
	Because $\exp(-tv) - r_1v$ is convex in $v$, the maximum occurs at the boundary $v=0$ or $v=1$. Therefore, the above problem is equivalent to
	\begin{equation*}
		\begin{array}{rcll}
			Z_d(t) = & \min & r_0 + r_1\hat{v} &
			\\
			& {\rm s.t.} & r_0 \geq \exp(-t) - r_1 &
			\\
			& & r_0 \geq 1 & \\
			& & r_0, r_1 \in \mathbb{R}.
		\end{array}
	\end{equation*}	
	This is a linear optimization problem and the optimal value is $Z_d(t) = 1-\hat{v} + \hat{v}\exp(-t)$, which is the moment generating function of a Bernoulli random variable with mean $\hat{v}$. Therefore, $Z(t)\leq Z_d(t)$ leads to Inequality~\eqref{eqn:mgf_bound}.
	
	By Inequality~\eqref{eqn:mgf_bound}, we can write
	\begin{equation*}
	\begin{array}{rcl}
		\displaystyle \bbp\left[\sum_{i\in[n_z]}|\tilde{z}_i| \leq (1-\epsilon)\mu \right] &\leq & \displaystyle \inf_{t>0} \exp(t(1-\epsilon)\mu ) \prod_{i\in[n_z]}\ep{\exp(-t|\tilde{z}_i|)} \\
		&\leq &  \displaystyle \inf_{t>0} \exp(t(1-\epsilon)\mu ) \prod_{i\in[n_z]}\ep{\exp(-t\tilde{v}_i)}
		\\
		&\leq &  \displaystyle \exp\left( -\frac{\mu \epsilon^2}{2} \right)
		%\\
		%&\leq & \displaystyle \exp\left( - \frac{N\sigma^2\epsilon^2}{2} \right)
	\end{array}
	\end{equation*}
	where the last inequality follows from the multiplicative Chernoff bound for the sum of independent Bernoulli random variables, which is a known result.% and the last inequality follows from $\mu \geq N\sigma^2$. 
	
	Finally, the proposition follows because
	\begin{equation*}
		\begin{array}{rcl}
		\displaystyle  \bbp\left[\bmt{z}\in \cal{U}_{c\sqrt{n_z}} \right]  &=& \displaystyle  \bbp\left[\sum_{i\in[n_z]}|\tilde{z}_i| \leq c\sqrt{n_z} \right] 
		\\
		&=& \displaystyle  \bbp\left[\sum_{i\in[n_z]}|\tilde{z}_i| \leq \left(1-\left(1-\frac{c\sqrt{n_z}}{ \mu}\right)\right)\mu \right] 
		\\
		
			&\leq & \displaystyle \exp\left( -\frac{\mu \left(1-\frac{c\sqrt{n_z}}{\mu}\right)^2}{2} \right)
		\\
		&\leq & \displaystyle \exp\left( -\frac{n_z\theta \left(1-\frac{c\sqrt{n_z}}{n_z\theta }\right)^2}{2} \right)
		\\
		%& = & \displaystyle \exp\left( -\frac{(\sqrt{\mu }-\frac{r}{\sqrt{\mu}})^2}{2} \right)
		%\\
		&= & \displaystyle \exp\left( -\frac{\left(\theta\sqrt{n_z}-c\right)^2}{2\theta} \right),
		\end{array}
	\end{equation*}
	where the last inequality follows from $\mu \geq n_z \theta$ and  $c\in (0, \theta \sqrt{n_z})$.	\hfill \Halmos
\end{proof}

\begin{proof}{Proof of Proposition~\ref{prop:RSfeasibility}.} Since the empirical optimization problem is solvable, there exists $\bmh x \in \cx$ such that  
$$
Z_0 = \frac{1}{\Omega}\sum_{\omega \in [\Omega]}g(\bmh x,\bmh z^\omega) =  \frac{1}{\Omega}\sum_{\omega \in [\Omega]}g^\omega(\bmh x,\bm 0).
$$ 
Hence, for any $\tau > Z_0$, let $
\upsilon_\omega  = g^\omega(\bmh{x},\bm{0}) + (\tau-Z_0)$. Observe that
$\frac{1}{\Omega}\sum_{\omega \in [\Omega]} \upsilon_\omega  = \tau$ and $ \upsilon_\omega  > g^\omega(\bmh{x},\bm{0})$. From the forth property of Definition~\ref{def:safeapprox},  there exists $\bar{k}_\omega \in \bbr_+$, $\omega \in [\Omega]$, such that 
$$
(\bmh x,\upsilon_\omega,\bar k_\omega) \in \bar{\mathcal Y}^\omega \qquad \forall \omega \in [\Omega].
$$
Likewise, from  the third property of Definition~\ref{def:safeapprox}, a feasibility solution of  Problem~\eqref{eq:data-drivenRS4} is to choose $k = \max_{\omega \in [\Omega]}\{ \bar k_\omega\}$, together with $\bm x = \bmh x$ and the above $\upsilon_\omega$, $\omega \in [\Omega]$. \qed
\end{proof}

\begin{proof}{Proof of Proposition~\ref{prop:quadratic_satisficing}.}
To avoid clutter, we first drop the dependency of $\bm{A}, \bm{b}, \bm{c}$ on $\bm{x}$, and then we expand the squared Euclidean norm in the robust constraint 
$\|\bm A(\bm{x})\bm z + \bm a(\bm x)\|_2^2  + \bm b(\bm x)^\top \bm z + c(\bm x) \leq k \bm z ^\top \bm  z+\upsilon$,$ \,\, \forall \bm{z}\in \mathcal{E}_{r}$ to equivalently express it as
\begin{equation*}
\left[\begin{array}{c} {1}\\ \bm{z} \end{array}\right]^{\top}
\left[ \begin{array}{cc}
 {-c+\upsilon -\bm{a}^{\top}\bm{a}} & {-\left(\bm{A}^{\top}\bm{a}+\frac{1}{2}\bm{b}\right)^{\top}} \\ 
{-\left(\bm{A}^{\top}\bm{a}+\frac{1}{2}\bm{b}\right)} & {k\bm I_{n_z}-\bm{A}^{\top}\bm{A}} \\ 
\end{array} 
\right] \left[\begin{array}{c} {1}\\ \bm{z} \end{array}\right]\geq 0 \bm \quad  
\forall \bm z : 
\left[\begin{array}{c} {1}\\ \bm{z} \end{array}\right]^{\top}
\left[ \begin{array}{cc}
r^2 & {\bm{0}^{\top}} \\ 
{\bm{0}} & {-\bm{I}_{n_z}} \\ 
\end{array} 
\right] \left[\begin{array}{c} {1}\\ \bm{z}\\\end{array}\right]\geq \bm{0}.
\end{equation*}
Applying $\mathcal{S}$-lemma leads to another equivalent representation of our quadratic constraint which is
\begin{equation*}
\begin{aligned}
&~~\exists \lambda \geq 0: ~ 
\left[ \begin{array}{cc}
{-c + \upsilon +\bm a^{\top} \bm a} & {-\left(\bm A^{\top} \bm a+\frac{1}{2} \bm b\right)^{\top}} \\ 
{-\left(\bm A^{\top} \bm a+\frac{1}{2} \bm b\right)} & {k\bm I_{n_z} -\bm A^{\top} \bm A}
\end{array}\right]-\lambda \left[ 
\begin{array}{cc}
{r^2} & {\bm{0}^{\top}} \\ 
{\bm{0}} & {-\bm I_{n_z}}\end{array}\right] \succeq \bm{0} \\
\Longleftrightarrow 
&~~\exists \lambda \geq 0: ~
\left[ \begin{array}{cc}
{-c+\upsilon +\lambda  r^2} & {-\frac{1}{2} \bm b^{\top}} \\
{-\frac{1}{2} \bm{b}} & {(k+\lambda) \bm{I}_{n_z}}
\end{array}\right]-\left[ \begin{array}{c}{\bm a^{\top}} \\ {\bm A^{\top}}
\end{array}\right] \bm{I}_{n_a}\left[ \begin{array}{c}{\bm a^{\top}} \\ {\bm A^{\top}}
\end{array}\right]^{\top} \succeq \bm 0.
\end{aligned}
\end{equation*}
Finally, using Schur's complement and recovering the dependency of $\bm{A}, \bm{b}, c$ on $\bm{x}$ yields the desired result. 
For the second half of the proof, as there exists $\hat{k} > 0$ such that
\begin{equation*}
\begin{aligned}
    \hat{k} \left[ \begin{array}{cc}
    \bm{I}_{n_a} & \bm{a}(\bm{x}) \\
    \bm{a}(\bm{x})^\top & \upsilon-c(\bm{x}) 
    \end{array} \right] \succeq 
    \left[ \begin{array}{c}
         \bm{A}(\bm{x}) \\ -\frac{1}{2}\bm{b}(\bm{x})^\top 
    \end{array} \right]
    \left[ \begin{array}{c}
         \bm{A}(\bm{x}) \\ -\frac{1}{2}\bm{b}(\bm{x})^\top
    \end{array} \right]^\top.
\end{aligned}
\end{equation*}
By the virtue of Schur's complement, choosing a vanishing $\lambda$ and $k = \hat{k}$ satisfies the semidefinite constraint presented in the proposition. The proof is now complete. \qed 
\end{proof}

\begin{proof}{Proof of Proposition~\ref{prop:BigM}.}  We first consider $0\leq v \leq \bar v$.
Since $\ell$ is a convex function, we have  
$$
\begin{array}{rcll}
&& \ell'(\bar v)\bar v- \ell(\bar v)  - ( \ell'(v)v-\ell(v)) \\
&=& \ell'(\bar v)\bar v- \ell'(v)v - (\ell(\bar v) -\ell(v)) \\
&\geq &  \ell'(\bar v)\bar v- \ell'(\bar v)v - (\ell(\bar v) -\ell(v))& \mbox{: since $\ell'(v) \leq \ell'(\bar v)$ and $v\geq 0$}\\
&=&\ell'(\bar v)(\bar v-v) - (\ell(\bar v) -\ell(v)) \\
&\geq& (\ell(\bar v) - \ell(v)) - (\ell(\bar v) -\ell(v)) & \mbox{: since $\ell$ is convex} \\
&=& 0.
\end{array}
$$
Hence, if $0\leq \underline w\leq \bar w$, we have 
$$
 \max_{ v \in [\underline w,\bar w]} \left\{ \ell'( v)v - \ell( v) \right\}=  \ell'( \bar w)\bar w - \ell( \bar w)=\max_{ v \in \{\underline w,\bar w\}} \left\{ \ell'( v)v - \ell( v) \right\}.
$$
Likewise,  consider $0\geq v \geq \underline v$ as follows,
$$
\begin{array}{rcll}
&& \ell'(\underline{v})\underline{v}- \ell(\underline{v})  - ( \ell'(v)v-\ell(v)) \\
&=& \ell'(\underline{v})\underline{v}- \ell'(v)v - (\ell(\underline{v}) -\ell(v)) \\
&\geq &  \ell'(\underline{v})\underline{v}- \ell'(\underline{v})v - (\ell(\underline{v}) -\ell(v))& \mbox{: since $\ell'(v) \geq \ell'(\underline{v})$ and $v\leq 0$}\\
&=&\ell'(\underline{v})(\underline{v}-v) - (\ell(\underline{v}) -\ell(v)) \\
&\geq& (\ell(\underline{v}) - \ell(v)) - (\ell(\underline{v}) -\ell(v)) & \mbox{: since $\ell$ is convex} \\
&=& 0.
\end{array}
$$
Hence, if $\underline w\leq \bar w \leq 0$, we have 
$$
 \max_{ v \in [\underline w,\bar w]} \left\{ \ell'( v)v - \ell( v) \right\}=  \ell'( \underline w)\underline w - \ell( \underline w)=\max_{ v \in \{\underline w,\bar w\}} \left\{ \ell'( v)v - \ell( v) \right\}.
$$
Finally, if $\underline w\leq 0\leq \bar w$, we would also have
$$
 \max_{ v \in [\underline w,\bar w]} \left\{ \ell'( v)v - \ell( v) \right\}=\max_{ v \in \{\underline w,\bar w\}} \left\{ \ell'( v)v - \ell( v) \right\}.
$$
\qed
\end{proof}

\begin{proof}{Proof of Proposition~\ref{prop:polypenaltynorm}}
Suppose that $p$ is a norm function and define
\begin{equation*}
    \mathcal{J} = \bigcup_{\bm{\zeta} \in \mathbb{R}^{n_z}} \argmax_{(\bm\lambda,\eta) \in \mathcal{V}} \left\{ \bm\lambda^\top \bm{\zeta} - \eta \right\}.
\end{equation*}
We will first show that $(\bm\lambda^\star,\eta^\star) \in \mathcal{J}$ implies $\eta^\star = 0$, that is, $\eta$ is superfluous in the optimization problem underlying the definition of $p$. Suppose otherwise for the sake of a contradiction that there exists $(\bm{\zeta}^\star,\bm{\lambda}^\star,\eta^\star) \in \mathbb{R}^{n_z} \times \mathbb{R}^{n_z} \times \mathbb{R}_{++}$ such that $(\bm\lambda^\star,\eta^\star) \in \argmax_{(\bm{\lambda},\eta) \in \mathcal{V}} \left\{ \bm\lambda^\top \bm{\zeta}^\star - \eta \right\}$. It follows that $p(\bm{\zeta}^\star) = (\bm\lambda^\star)^\top\bm{\zeta}^\star - \eta^\star$ and, as $p$ is a norm, $p(2\bm{\zeta}^\star) = 2(\bm\lambda^\star)^\top\bm{\zeta}^\star - 2\eta^\star$. Note also that 
$$
p(2\bm{\zeta}^\star) = \max_{(\bm\lambda,\eta) \in \mathcal{V}} \left\{ \bm\lambda^\top (2\bm{\zeta}^\star) - \eta \right\} \geq (\bm\lambda^\star)^\top (2\bm{\zeta}^\star) - \eta^\star.
$$ 
By comparing $p(2\bm{\zeta}^\star)$ with its lower bound, we find $\eta^\star \leq 0$, reaching hence a contradiction. In conclusion, $\eta$ always vanish at optimality, and we can assume that $\bm{s} = \bm{0}$ without any loss of generality:
\begin{equation*}
    p(\bm{\zeta}) = \max_{\bm\lambda \in \bbr^{n_z}, \bm\mu \in \bbr^{n_\mu}} \left\{ \bm\lambda^\top\bm{\zeta} \mid \bm{M}\bm\lambda + \bm{N}\bm\mu \leq \bm{t}  \right\}.
\end{equation*}
Since $p(\bm \zeta) = p(-\bm \zeta)$ for all $\bm \zeta \in \bbr^{n_z}$, the polytope $\mathcal Q_1 = \left\{ \bm \lambda \mid \bm{M}\bm\lambda + \bm{N}\bm\mu \leq \bm{t}  \right\}$ must be identical to the polytope $\mathcal Q_2= \left\{ \bm \lambda \mid -\bm{M}\bm\lambda + \bm{N}\bm\mu \leq \bm{t}  \right\}$. Otherwise, suppose $\bm \lambda^* \in \mathcal Q_1$ but $\bm \lambda^* \notin \mathcal Q_2$, then by a separating hyperplane argument there would exist a vector $\bm \zeta^* \in \bbr^{n_z}$ such that 
$$
p(\bm \zeta^*)= \max_{\bm \lambda \in \mathcal Q_1}\{ \bm \lambda^\top \bm \zeta^*\}  \geq   {\bm \lambda^*}^\top \bm \zeta^* >  \max_{\bm \lambda \in \mathcal Q_2}\{ \bm \lambda^\top \bm \zeta^*\}=p(-\bm \zeta^*),
$$
which is a contradiction. Likewise, similar contradiction can be established if  $\bm \lambda^* \in \mathcal Q_2$ but $\bm \lambda^* \notin \mathcal Q_1$. Hence, $\mathcal Q_1 = \mathcal Q_2$. Next, we derive the dual norm $p^\star$:
\begin{equation*}
\begin{aligned}
    p^\star(\bm{\zeta}) &= \max_{\bm\omega \in \bbr^{n_z}} \left\{ \bm\omega^\top \bm{\zeta} \mid p(\bm{\omega}) \leq 1 \right\} \\
    &= \max_{\bm\omega \in \bbr^{n_z}} \left\{ \bm\omega^\top \bm{\zeta} \mid \bm\lambda^\top \bm\omega \leq 1,\; \forall (\bm\lambda, \bm\mu): \bm{M}\bm\lambda + \bm{N}\bm\mu \leq \bm{t} \right\} \\
    &= \max_{\bm\omega \in \bbr^{n_z}, \bm\alpha \in \bbr^{n_m}_+} \left\{ \bm\omega^\top \bm{\zeta} \mid \bm\alpha^\top\bm{t} \leq 1, \; \bm{M}^\top\bm\alpha = \bm\omega, \; \bm{N}^\top \bm\alpha = \bm{0} \right\} \\
    &= \max_{\bm\alpha \in \bbr^{n_m}_+} \left\{ \bm\alpha^\top\bm{M\zeta} \mid \bm\alpha^\top \bm{t}\leq 1, \; \bm{N}^\top \bm\alpha = \bm{0} \right\} \\ 
    &= \min_{\bm\mu \in \bbr^{n_\mu}, \delta \in \bbr_+} \left\{ \delta \mid \bm{M\zeta} + \bm{N\mu} \leq \delta \bm{t} \right\},
\end{aligned} 
\end{equation*}
where the third maximization problem constitutes a robust counterpart of the second and the fifth equation holds due to the standard linear optimization duality argument. \qed 
\end{proof}

\begin{proof}{Proof of Proposition~\ref{prop:innermax}.}
We first consider $k>0$. From the definition of $p$ in Assumption~\ref{assm:penaltysupport}, we find
\begin{equation}\label{opt:innermax_z}
    \max_{\bm{z} \in \mathcal{Z}} \left\{\bm{a}^\top\bm{z} - k p(\bm{z}) \right\} \;=\; \max_{\bm{z} \in \cz} \min_{(\bm \lambda,\eta) \in \mathcal V}  \left\{(\bm{a} - k\bm \lambda)^\top \bm z + k\eta    \right\}.
\end{equation}
Since this is linear optimization problem, by the standard linear optimization duality argument, this latter maximization problem can be expressed as
\begin{equation*}
\begin{aligned}
    &\max_{\bm{z} \in \cz} \min_{(\bm \lambda,\eta) \in \mathcal V}  \left\{(\bm{a} - k\bm \lambda)^\top \bm z + k\eta    \right\} \\
    =& \min_{(\bm \lambda,\eta) \in \mathcal V}  \max_{\bm{z} \in \cz} \left\{(\bm{a} - k\bm \lambda)^\top \bm z + k\eta    \right\}\\
    =&\min_{ (\bm \lambda,\eta) \in \mathcal V, \bm \beta}  \; \left\{ \bm{\beta}^\top\bm{h} + k\eta \mid  \bm\beta \geq \bm{0},\; \bm H^\top\bm \beta = \bm a - k \bm \lambda  \right\} \\
    =&\min_{{\bm\beta}\geq \bm 0,\eta \geq 0} \; \left\{ \bm{\beta}^\top\bm{h} + k\eta  \mid  (\bm a -\bm H^\top\bm \beta,k\eta,k) \in \bar{\mathcal V}  \right\} \\
    =&\min_{{\bm\beta}\geq \bm 0,\eta \geq 0} \; \left\{ \bm{\beta}^\top\bm{h} + \eta  \mid  (\bm a -\bm H^\top\bm \beta,\eta,k) \in \bar{\mathcal V}  \right\},
\end{aligned}
\end{equation*}
where the first interchange between minimization and maximization is justified because $\mathcal{V}$ is bounded, and this concludes the desired equivalence. Next, we consider the case when $k=0$. Since $\mathcal V$ is bounded, we must have 
$\{ (\bm \lambda,\eta) \mid  (\bm \lambda,\eta,0) \in \bar{\mathcal V} \} = \{ (\bm 0,0)\}$.
Hence, 
$(\bm a -\bm H^\top\bm \beta,\eta,0) \in \bar{\mathcal V} $
is equivalent to $\eta =0$ and $\bm H^\top\bm \beta = \bm a$. By noting that $\max_{\bm{z} \in \mathcal{Z}} \bm{a}^\top \bm{z} = \min_{\bm\beta \in \mathbb{R}^{n_h}_+} \{ \bm\beta^\top \bm{h} \mid \bm{H}^\top \bm\beta = \bm{a} \}$, the first half of the proposition follows.

If the penalty function $p$ is a norm, from the derivation of the dual norm in Proposition \ref{prop:polypenaltynorm}, we have
\begin{equation*}
\begin{aligned}
    &\min_{{\bm\beta} \in \mathbb{R}^{n_h}_+}  \left\{ \bm{\beta}^\top\bm{h} \mid  p^\star(\bm a -\bm H^\top\bm\beta) \leq k \right\} \\
    =&\min_{{\bm\beta} \in \mathbb{R}^{n_h}_+, \bm\mu \in \bbr^{n_\mu}, \delta \in \bbr_+}  \left\{ \bm{\beta}^\top\bm{h} \mid \delta \leq k, \;  \bm{M}(\bm a -\bm H^\top\bm\beta) + \bm{N\mu} \leq \delta \bm{t} \right\} \\
    =&\min_{{\bm\beta} \in \mathbb{R}^{n_h}_+, \bm\mu \in \bbr^{n_\mu}}  \left\{ \bm{\beta}^\top\bm{h} \mid  \bm{M}(\bm a -\bm H^\top\bm\beta) + \bm{N\mu} \leq k\bm{t} \right\},
\end{aligned}
\end{equation*}
where the second equality follows because, to ensure that $(\bm{0},0) \in \mathcal{V}$, $\bm{t}$ must be non-negative making the constraint $\delta \leq k$ binding at optimality. Noting from the previous half of this proposition and the proof of Proposition \ref{prop:polypenaltynorm} that, as $\bm{s}$ vanishes when the penalty function is a norm, the latter minimization problem is equivalent to $\max_{\bm{z} \in \mathcal{Z}} \left\{\bm{a}^\top\bm{z} - k p(\bm{z}) \right\}$ completes the proof.\qed 
\end{proof}

\begin{proof}{Proof of Proposition~\ref{prop:satisficing_ldr_feasibility}.}

Its suffices to show that there a feasible solution in $\bar{\mathcal Y}_P$, for any $\bm x \in \cx$ and $\upsilon > g(\bm x,\bm 0)$ Let $\bm q$ be the optimal $\bm y$ solution to Problem~\eqref{opt:nominal_twostage3.3} for $\bm z = \bm 0$ that yields the objective $g(\bm x,\bm 0)$, and $k=\bm{d}^\top\hat{\bm{q}}^\dag$ where $\hat{\bm{q}}^\dag \in \mathbb{R}^{n_y}$ satisfies $\bm{B}\hat{\bm{q}}^\dag \geq \bm{0}$ and $\bm{B}\hat{\bm{q}}^\dag \geq \max_{i \in [n_f], j \in [n_z]} \vert \bm{F}_{ij}(\hat{\bm{x}}) \vert \cdot \bm{1}$. 
Note that $\hat{\bm{q}}^\dag$ always exists because of complete recourse, which we assume. The suggested solution robustly satisfies first inequality requirement in~\eqref{opt:satisficing_nldr} because
\begin{equation*}
     \bm{c}^\top \bm{x} + \bm{d}^\top \left( \bm{q} + \bm{q}^\dag \Vert \bm{z} \Vert_1 \right) =  \bm{c}^\top \bm{x} + \bm{d}^\top \bm{q} + k \Vert \bm{z} \Vert_1 = g(\bm x,\bm 0) + k \Vert \bm{z} \Vert_1 \leq \upsilon  + k \Vert \bm{z} \Vert_1.     
\end{equation*}
Moreover, the second inequality requirement is also robustly satisfied as
\begin{equation*}
\begin{aligned}
\bm{B}(\bm{q} + \bm{q}^\dag  \Vert \bm{z} \Vert_1) &\geq \bm{f}(\bm{x}) + \bm{B}\bm{q}^\dag  \Vert \bm{z} \Vert_1 \\
    &\geq \bm{f}(\bm{x}) + \Vert \bm{z} \Vert_1 \left\{ \max_{i \in [n_f], j \in [n_z]} \vert \bm{F}_{ij}(\bm{x}) \vert  \right\} \cdot \bm{1} \geq \bm{f}(\bm{x}) + \bm{F}(\bm{x})\bm{z},\quad \forall \bm{z} \in \mathcal{Z}
\end{aligned}
\end{equation*}
where the first two inequalities follow from the feasibility of $\bm {y} = \bm q$ in Problem~\eqref{opt:nominal_twostage3.3} for $\bm z = \bm 0$, and the construction of $\bm{q}^\dag$, respectively. Consequently, $(\bm x, \upsilon, k) \in \bar{\mathcal Y}_P$. Since $\bar{\mathcal Y}_P$ is increasing in $\upsilon$ and $k$, if it is a tractable set, which we show next, then it must be a tractable safe approximation of $\mathcal{Y}$.

Regarding the explicit tractable formulation, we note that the first robust constraint of the feasible set \eqref{opt:satisficing_nldr} can be expressed as
    \begin{equation*}
        \max_{\bm{z}} \left\{ \bm{d}^\top \bm{Q} \bm{z} - (k - \bm{d}^\top \bm{q}^\dag) \Vert \bm z \Vert_1 \mid \; \bm{H}\bm{z} \leq \bm{h} \right\} \leq \tau - \bm{d}^\top \bm{q} - \bm{c}^\top \bm{x}.
    \end{equation*}
    By Proposition~\ref{prop:innermax}, we can replace the maximization problem on the left-hand side of by a minimization problem:
    \begin{equation*}
        \min_{\bm{w}^0} \left\{\bm{h}^\top \bm{w}^0 \mid \; \Vert \bm{H}^\top \bm{w}^0 - \bm{Q}^\top \bm{d} \Vert_\infty \leq k - \bm{d}^\top \bm{q}^\dag ,\; \bm{w}^0 \geq \bm{0} \right\}.
    \end{equation*}
    %thereby giving rise to the first two constraints in Problem~\eqref{opt:satisficing_nldr_rc}. 
    Similarly, the second robust constraint of the feasible set \eqref{opt:satisficing_nldr} can be written down as
    \begin{equation*}
        \max_{\bm{z}} \left\{ (\bm{F}_i(\bm{x}) - \bm{B}_i\bm{Q}) \bm{z} - \bm{B}_i \bm{q}^\dag \Vert \bm z \Vert_1 : \; \bm{H}\bm{z} \leq \bm{h} \right\} \leq \bm{B}_i \bm{q} - \bm{f}_i(\bm{x}) \qquad \forall i \in [n_f],
    \end{equation*}
    whose left-hand side maximization problem can be replaced by
    \begin{equation*}
        \min_{\bm{w}^i} \left\{ \bm{h}^\top \bm{w}^i: \; \Vert \bm{H}^\top \bm{w}^i - \bm{F}_i^\top (\bm{x}) + \bm{Q}^\top \bm{B}^\top_i  \Vert_\infty \leq \bm{B}_i \bm{q}^\dag, \; \bm{w}^i \geq \bm{0} \right\} \qquad \forall i \in [n_f],
    \end{equation*}
    %thereby giving rise to the next two constraints in Problem~\eqref{opt:satisficing_nldr_rc}. 
    Finally, the two deterministic linear constraints (namely, $k - \bm{d}^\top \bm{q}^\dag \geq 0$ and $\bm{B}\bm{q}^\dag \geq \bm{0}$) of the feasible set~\eqref{opt:satisficing_nldr} are redundant in view of the feasible set~\eqref{opt:satisficing_nldr_rc} and can therefore be safely omitted. The proof is thus completed. \qed
\end{proof}

\begin{proof}{Proof of Proposition~\ref{prop:portfoliosafeapproximation}.}
We first express the budgeted norm in Example~\ref{example:budgeted_norm} as well as the support sets $\mathcal{Z}^\omega = [\underline{\bm{z}} - \hat{\bm{z}}^\omega, \bar{\bm{z}} - \hat{\bm{z}}^\omega]$, $\omega \in [\Omega]$, as
\begin{equation*}
   p_\Gamma(\bm \zeta) = \max_{\bm\lambda, \bm\mu \in \bbr^{n}} \left\{ \bm\lambda^\top\bm{\zeta} \left|  \underbrace{\left[ \begin{array}{c} \bm{0}^{\top}\\ \bm{I}\\  -\bm{I}\\ \bm{0} \end{array}\right]}_{\bm{M}} \bm\lambda +\underbrace{\left[ \begin{array}{c} \bm{1}^{\top}\\ -\bm{I}\\  -\bm{I}\\ \bm{I} \end{array}\right]}_{\bm{N}}\bm\mu \leq   \underbrace{\left[ \begin{array}{c} \Gamma\\ \bm{0}\\  \bm{0}\\ \bm{1} \end{array}\right]}_{\bm{t}}  \right.\right\}
   \ \ \text{and} \ \ 
   \cz^\omega =\left\{ \bm z \in \mathbb{R}^{n_x} \left|\underbrace{ \left[ \begin{array}{c} \bm{I}\\  -\bm{I}\ \end{array}\right]}_{\bm H} \bm z  \leq \underbrace{\left[ \begin{array}{c} \bar{\bm{z}}^\omega\\  -\underline{\bm{z}}^\omega \end{array}\right]}_{\bm h^\omega}\right.\right\},
\end{equation*}
where $\bar{\bm z}^\omega = \bar{\bm z} - \bmh z^\omega$ and  $\underline{\bm z}^\omega = \underline{\bm z} - \bmh z^\omega$. Then, we express the evaluation function $g^\omega(\bm{x},\bm{z})$, $\omega \in [\Omega]$, as 
\begin{equation*}
\begin{aligned}
 g^{\omega}(\bm{x},\bm{z}) = \; &\text{minimize} && y_1+Py_2 \\
    &\text{subject to} && 
    \underbrace{\left[ \begin{array}{ccc} 
 0& 0 \\
1&0 \\ 
0&1\\
0&1\\
\end{array}\right]}_{\bm{B}}
 \left[ \begin{array}{c}   y_1\\y_2
\end{array}\right] \succeq_{\mathcal{K}}
    \underbrace{ \left[ \begin{array}{c} \bm{x}^\top \bmh z^\omega\\  -P \\0
\\1 \end{array}\right]}_{\bm{f}^\omega(\bm{x})}+ \underbrace{\left[ \begin{array}{c}  \bm{x}^\top\\  \bm{0}^{\top} 
\\\bm{0}^{\top} \\\bm{0}^{\top}\end{array}\right]}_{\bm{F}(\bm{x})}\bm{z}\\
&&& \bm y \in \bbr^2,
\end{aligned}    
\end{equation*}
where 
$\mathcal{K} = \mathcal K_{\ell_a} \times  \bbr_{+}.$
The safe approximation of Problem~\eqref{eq:satisficingexputil}, which is due to~\eqref{opt:satisficing_dd_dual}, reads
\begin{equation}
\label{eq:satisficingexputil2}
\begin{array}{rclll}
    &{\rm minimize} & k  \\
    &{\rm~subject~to} &  \displaystyle  \frac{1}{\Omega} \sum_{\omega \in [\Omega]}\upsilon_\omega \leq \tau 
      \\
     &&\displaystyle    \bm{x}^\top \bmh z^\omega \rho_1 -P\rho_2+ \rho_4+ (\bar{\bm{z}}^{\omega})^\top \bar{\bm\beta}^\omega (\bm\rho)- (\underline{\bm{z}}^{\omega})^\top \underline{\bm\beta}^\omega (\bm\rho) \leq \upsilon_\omega  &  \hspace{-28mm}\forall \bm{\rho}  \in \mathcal{P}, \omega \in [\Omega]  \\
     && \displaystyle  \bm{1}^{\top}\bm \mu^\omega(\bm\rho) \leq \Gamma k& \hspace{-28mm}\forall \bm{\rho}  \in \mathcal{P},  \omega \in [\Omega]\\
     && \displaystyle \bm{x} \rho_1- \bar{\bm\beta}^\omega (\bm\rho) + \underline{\bm\beta}^\omega (\bm\rho) - \bm \mu^\omega(\bm\rho) \leq \bm 0&\hspace{-28mm}  \forall \bm{\rho}  \in \mathcal{P}, \omega \in [\Omega]  \\
     && \displaystyle -\bm{x} \rho_1+ \bar{\bm\beta}^\omega (\bm\rho) - \underline{\bm\beta}^\omega (\bm\rho) - \bm \mu^\omega(\bm\rho) \leq \bm 0&\hspace{-28mm}  \forall \bm{\rho}  \in \mathcal{P}, \omega \in [\Omega]  \\
     &&\bm \mu^\omega(\bm\rho) \leq \bm 1k &\hspace{-28mm}  \forall \bm{\rho}  \in \mathcal{P}, \omega \in [\Omega]  \\
     & &\bar{\bm\beta}^\omega (\bm\rho) \geq \bm{0},\underline{\bm\beta}^\omega (\bm\rho) \geq \bm{0}\; &\hspace{-28mm} \forall \bm\rho \in \mathcal{P}, \omega \in [\Omega] \\
    & & \bm{x} \in \mathcal{X},\;  k \in \mathbb{R}_+,\; \bm \upsilon \in \bbr^{\Omega}, \;\bar{\bm{\beta}}^1, \hdots, \bar{\bm\beta}^\Omega \in \mathcal{L}^{4, n},\;\underline{\bm{\beta}}^1, \hdots, \underline{\bm\beta}^\Omega \in \mathcal{L}^{4, n}, \bm\mu^1, \hdots, \bm\mu^\Omega \in \mathcal{L}^{4, n}.
\end{array}
\end{equation}
where the dual uncertainty set $\mathcal{P}$ is 
\iffalse
$$
\mathcal{P}=\left\{ \bm\rho \in \mathcal{K}^\star \left| \left[ \begin{array}{ccc}
 0& 0 \\
1&0 \\ 
0&1\\
0&1\\
\end{array}\right]^\top \bm\rho = \left[ \begin{array}{c} 
1\\
P\\
\end{array}\right] \right.\right\}.
$$ 
\fi
\begin{equation*}
\mathcal{P} = \left\{ \bm\rho \in \mathcal{K}^\star \left| \bm{B}^\top\bm\rho = \left[ \begin{array}{c} 1 \\ P \end{array} \right]\right.\right\} .
\end{equation*}
Problem~\eqref{eq:satisficingexputil3} is then obtained by characterizing each $\bar{\bm\beta}^\omega \in \mathcal{L}^{4,n}$ ($\underline{\bm\beta}^\omega \in \mathcal{L}^{4,n}$) as $\bar{\bm\beta}^{\omega,0} + \sum_{j\in [4]} \bar{\bm{\beta}}^{\omega,j}\rho_j$ for some $\bar{\bm\beta}^{\omega,0}, \hdots, \bar{\bm\beta}^{\omega,4} \in \mathbb{R}^n$ ($\underline{\bm\beta}^{\omega,0} + \sum_{j\in [4]} \underline{\bm{\beta}}^{\omega,j}\rho_j$ for some $\underline{\bm\beta}^{\omega,0}, \hdots, \underline{\bm\beta}^{\omega,4} \in \mathbb{R}^n$). Analogous treatment is applied for each $\bm\mu^\omega \in \mathcal{L}^{4,n}$.

\iffalse 
Or equivalently, we have
Let 
$$
\begin{array}{rcl}
\bar{\bm{\beta}}^{\omega}(\bm \rho)&=& \displaystyle \bar{\bm{\beta}}^{\omega,0}+\sum_{j\in [4]} \bar{\bm{\beta}}^{\omega,j}\rho_j\\
\underline{\bm{\beta}}^{\omega}(\bm \rho)&=& \displaystyle \underline{\bm{\beta}}^{\omega,0}+\sum_{j\in[4]} \underline{\bm{\beta}}^{\omega,j}\rho_j\\
\bm{\mu}^{\omega}(\bm \rho)&=& \displaystyle \bm{\mu}^{\omega,0}+\sum_{j\in [4]} \bm{\mu}^{\omega,j}\rho_j,
\end{array}
$$ 
then \eqref{eq:satisficingexputil2} can be re-written as Problem~\eqref{eq:satisficingexputil3}.
\fi

To complete the proof, we finally note that
\begin{equation*}
\begin{aligned}
    \forall \bm \rho \in \mathcal P: \bm v^\top \bm \rho \leq w 
    \quad &\Longleftrightarrow \quad 
    \max_{\bm\rho \in \mathcal{P}} \bm{v}^\top \bm\rho \leq w \\
    \quad &\Longleftrightarrow \quad 
    \min_{\bm{y} \in \mathcal{K} } \left\{ y_1 + Py_2: \bm{By} \succeq_{\mathcal{K}} \bm{v} \right\} \leq w \\
    \quad &\Longleftrightarrow \quad
    \exists \bm{y} \in \mathbb{R}^2: \bm{By} \succeq_{\mathcal{K}} \bm{v}, \ y_1 + Py_2 \leq w \\
    \quad &\Longleftrightarrow \quad \exists \bm y \in \bbr^2 :
    \left\{ \begin{array}{l}
        y_1+Py_2 \leq w \\
        (y_2-v_3)(\exp(-av_1/(y_2-v_3))-1)/a\leq (y_1 -v_2)\\
        y_2-v_3 \geq 0\\
        y_2-v_4 \geq 0
    \end{array} \right. \\
    \quad &\Longleftrightarrow \quad \exists y \in \bbr :
    \left\{ \begin{array}{l}
        (y-v_3)(\exp(-av_1/(y-v_3))-1)/a + v_2 + Py \leq w \\
        y\geq v_3 \\
        y\geq v_4,
    \end{array} \right.
\end{aligned}
\end{equation*}
where the last equivalence is due to the elimination of $y_1$ and renaming $y_2$ as $y$.
\iffalse 
$$
\begin{array}{c}
\bm v^\top \bm \rho \leq w \qquad  \forall \bm \rho \in \mathcal P \\
\Updownarrow \\
\exists \bm y \in \bbr^3 : \left\{ \begin{array}{l}
y_1+Py_2 \leq w \\
(y_2-v_3)(\exp(-av_1/(y_2-v_3))-1)/a\leq (y_1 -v_2)\\
y_2-v_3 \geq 0\\
y_2-v_4 \geq 0
\end{array}\right.
\end{array}
$$
or equivalently as
$$
\exists y \in \bbr: \left\{ \begin{array}{l}
(y-v_3)(\exp(-av_1/(y-v_3))-1)/a + v_2 + Py \leq w \\
y\geq v_3 \\
y\geq v_4.
\end{array}\right.
$$
\fi
\qed
\end{proof}
\section{Parameters used in Problem~\eqref{opt:teststrict}}
\label{appendix:testrestrictparam}
\begin{align*}
    & \bm{H} = \left[ \begin{array}{cccccccccc} 
        0.2220,  &  0.6117,  &  0.0807,  &  0.2741,  &  0.5999,  &  0.1442,  &  0.0243,  &  0.5777,  &  0.2591 \\
        0.8707,  &  0.7659,  &  0.7384,  &  0.4142,  &  0.2658,  &  0.1656,  &  0.2046,  &  0.0016,  &  0.8025 \\
        0.2067,  &  0.5184,  &  0.4413,  &  0.2961,  &  0.2847,  &  0.9639,  &  0.6998,  &  0.5155,  &   0.8705 \\
        0.9186,  &  0.2968,  &  0.1583,  &  0.6288,  &  0.2536,  &  0.9602,  &  0.7795,  &  0.6398,  &  0.9227 \\
        0.4884,  &  0.1877,  &  0.8799,  &  0.5798,  &  0.3276,  &  0.1884,  &  0.0229,  &  0.9856,  &  0.0022 \\
        -1 & 0 & 0 & 0 & 0 & 0 & 0 & 0 & 0 \\
        0 & -1 & 0 & 0 & 0 & 0 & 0 & 0 & 0 \\
        0 & 0 & -1 & 0 & 0 & 0 & 0 & 0 & 0 \\
        0 & 0 & 0 & -1 & 0 & 0 & 0 & 0 & 0 \\
        0 & 0 & 0 & 0 & -1 & 0 & 0 & 0 & 0 \\
        0 & 0 & 0 & 0 & 0 & -1 & 0 & 0 & 0 \\
        0 & 0 & 0 & 0 & 0 & 0 & -1 & 0 & 0 \\
        0 & 0 & 0 & 0 & 0 & 0 & 0 & -1 & 0 \\
        0 & 0 & 0 & 0 & 0 & 0 & 0 & 0 & -1 \\
    \end{array} \right], 
    \bm{h} = \left[ \begin{array}{c} 0.4695 \\ 0.9815 \\ 0.3989 \\ 0.8137 \\ 0.5465 \\ 100 \\ 100 \\ 100 \\ 100 \\ 100 \\ 100 \\ 100 \\ 100 \\ 100 \end{array} \right], \\
    &
    \bm{F} = \left[ \begin{array}{cccccccccc} 
        0.8248,  &  0.5464,  &  0.3655,  &  0.6389,  &  0.9435,  &  0.1008,  &  0.3715,  &  0.4783,  &  0.8005 \\
        0.0942,  &  0.7961,  &  0.2443,  &  0.4934,  &  0.1117,  &  0.3834,  &  0.0124,  &  0.8500,  &  0.0204 \\ 
        0.3610,  &  0.0511,  &  0.7951,  &  0.5835,  &  0.8436,  &  0.5104,  &  0.8597,  &  0.5147,  &  0.5726 \\ 
        0.0355,  &  0.1887,  &  0.3521,  &  0.9393,  &  0.3460,  &  0.9611,  &  0.1111,  &  0.4466,  &  0.4114 \\
    \end{array} \right], \\
    &\bm{d} = \left[ \begin{array}{c} 0.8167 \\ 0.5661 \end{array} \right]^\top, \quad
    \bm{f} = \left[ \begin{array}{c} 0.6354 \\ 0.8119 \\ 0.9267 \\ 0.9126 \end{array} \right], \quad
    %\bm{h} = \left[ \begin{array}{c} 0.4695 \\ 0.9815 \\ 0.3989 \\ 0.8137 \\ 0.5465 \\ 100 \\ 100 \\ 100 \\ 100 \\ 100 \\ 100 \\ 100 \\ 100 \\ 100 \end{array} \right], \quad
    \bm{B} = \left[ \begin{array}{cc} 
        0.7709,  &  0.1115 \\
        0.4849,  &  0.2512 \\ 
        0.0291,  &  0.9649 \\
        0.0865,  &  0.6318
    \end{array}\right].
\end{align*}

%%%%%%%%%%%%%%%%%%%%%%%%%%%%%%%%%%%%%%%%%%%%%%%%%%%%%%%%%%%%%%%%%%%%%%%%%%%%%%%%%%%%%%%%%%%%%%%%%%%%%%%%%%%%%%%%%%%%%%%%%%%%
\section{Additional study on the adaptive network lot-sizing problem}\label{subsec:Lot-sizing}
Next to demonstrate the handling of two-stage linear optimization under uncertainty using a satisficing approach, we consider a network lot-sizing example similar to \cite{Bertsimas2016} and \cite{trevor2018dual_twostage}. Suppose that there are $n$ nodes, each of which faces a random demand $z_i, \; i \in [n]$. Throughout, the support set of $\bm{z}$ is assumed to be a hyperrectangle, \emph{i.e.}, $\mathcal{Z} = \left\{ \bm{z} \in \bbr^{n_z}_+ \mid \bm{z} \leq \bar{\bm{z}} \right\}$. The initial stock $x_i \geq 0$ at each node is to be determined prior to the realization of the random demands. Similarly, we impose that $\mathcal{X} = \left\{ \bm{x} \in \bbr^{n_x}_+ \mid \bm{x} \leq \bar{\bm{x}} \right\}$. After observing the demand, we can transport stock $y_{ij} \geq 0$ from node $i$ to node $j$. To ensure that the demands can always be fulfilled, we allow for an emergency order $w_i \geq 0$ to be made at each node. Initial and emergency orders are purchased at the unit costs $c_i \geq 0$ and $\ell_i \geq c_i$, respectively, while the unit transportation costs are denoted by $t_{ij} \geq 0$. It is assumed that $t_{ij}$ is equal to the distance between the two nodes, and hence $t_{ij} = t_{ji}$. 

\iffalse
\textcolor{black}{
First, we present the data-driven robust variant of the lot-sizing problem with $\Omega =1$
\begin{equation}
\label{opt:lotsizing_datadriven}
\begin{aligned}
	Z_r =\; &\text{minimize} && \bm{c}^\top \bm{x} + x_0 \\
	&\text{subject to} && \bm{x} + \bm{Y}(\bm{z})^\top \bm{1} - \bm{Y}(\bm{z}) \bm{1} + \bm{w}(\bm{z}) \geq \bm{z} && \forall \bm{z} \in \mathcal{Z} \\
	& && \left\langle \bm{T}, \bm{Y}(\bm{z}) \right\rangle + \bm{\ell}^\top \bm{w}(\bm{z})-k(\bm{1}^\top \bm{z} -r) \leq x_0 && \forall \bm{z} \in \mathcal{Z} \\
	& && \bm{Y}(\bm{z}) \geq \bm{0}, \; \bm{w}(\bm{z}) \geq \bm{0} && \forall \bm{z} \in \mathcal{Z} \\ 
	& && \bm{x} \in \mathcal{X}, \; x_0 \in \mathbb{R}, \; \bm{Y} \in \mathcal{R}^{n, n\times n}, \; \bm{w} \in \mathcal{R}^{n,n},
\end{aligned}
\end{equation}}
\fi

First, we present the robust variant of the lot-sizing problem 
\begin{equation}
\label{opt:lotsizing_robust_primal}
\begin{aligned}
	Z_r =\; &\text{minimize} && \bm{c}^\top \bm{x} + x_0 \\
	&\text{subject to} && \bm{x} + \bm{Y}(\bm{z})^\top \bm{1} - \bm{Y}(\bm{z}) \bm{1} + \bm{w}(\bm{z}) \geq \bm{z} && \forall \bm{z} \in \mathcal{U}_r \\
	& && \left\langle \bm{T}, \bm{Y}(\bm{z}) \right\rangle + \bm{\ell}^\top \bm{w}(\bm{z}) \leq x_0 && \forall \bm{z} \in \mathcal{U}_r \\
	& && \bm{Y}(\bm{z}) \geq \bm{0}, \; \bm{w}(\bm{z}) \geq \bm{0} && \forall \bm{z} \in \mathcal{U}_r \\ 
	& && \bm{x} \in \mathcal{X}, \; x_0 \in \mathbb{R}, \; \bm{Y} \in \mathcal{R}^{n, n\times n}, \; \bm{w} \in \mathcal{R}^{n,n},
\end{aligned}
\end{equation}
where $\mathcal{U}_r = \{ \bm{z} \in \mathbb{R}^n_+ \mid \bm{z} \leq \bar{\bm{z}}, \; \bm{1}^\top \bm{z} \leq r \}$ denotes the uncertainty set and the epigraph variable $x_0$ captures the worst-case transportation and emergency purchase costs. %Note that since the nominal demand $\bm{\hat{z}}=\bm{0}$, the nominal lot-sizing objective $Z_0=0$. 
Similar to the proof of Theorem~\ref{thm:dual_satisficing}, we can dualize Problem~\eqref{opt:lotsizing_robust_primal} twice (the first time over the second-stage decisions $\bm{Y},\bm{w}$ and the second time over the uncertain demands $\bm{z}$) to obtain an alternative formulation.
\begin{proposition}
\label{prop:lotsizing_robust_dual}
Problem~\eqref{opt:lotsizing_robust_primal} is equivalent to
\begin{equation}
\label{opt:lotsizing_robust_dual}
\begin{aligned}
	Z_r =\; &\rm{minimize} && \bm{c}^\top \bm{x} + x_0 \\
	&\rm{subject to} 	&& \bm\beta(\bm\rho)^\top \bar{\bm{z}} + \hat\beta (\bm\rho) r \leq \bm\rho^\top \bm{x} + x_0 &&\hspace{-5mm} \forall \bm{\rho} \in \mathcal{P} \\
	&&& \bm\beta(\bm\rho) + \hat\beta(\bm\rho) \bm{1} \geq \bm\rho && \hspace{-5mm}\forall \bm{\rho} \in \mathcal{P} \\
	& && \bm\beta(\bm\rho) \geq \bm{0}, \; \hat\beta(\bm\rho) \geq 0 &&\hspace{-5mm} \forall \bm\rho \in \mathcal{P} \\ 
	& && \bm{x} \in \mathcal{X}, \; x_0 \in \mathbb{R}, \; \bm\beta \in \mathcal{R}^{n,n}, \; \hat\beta \in \mathcal{R}^{n,1},
\end{aligned}
\end{equation}
where the dual uncertainty set $\mathcal{P}$ is given by $\{ \bm\rho \in \mathbb{R}^n_+: \bm\rho \leq \bm\ell, \; \bm\rho\bm{1}^\top - \bm{1}\bm\rho^\top \leq \bm{T} \}$. 
\end{proposition}
\begin{proof}{Proof.}
We could express the robust constraints of Problem~\eqref{opt:lotsizing_robust_primal} as  
\begin{equation*}
\begin{aligned}
    &\max_{\bm{z} \in \mathcal{U}_r} \; \min_{\bm{Y}\geq\bm{0},\bm{w}\geq\bm{0}} \left\{ \left\langle \bm{T}, \bm{Y} \right\rangle + \bm\ell^\top \bm{w} \mid 
    \bm{x} + \bm{w} + \bm{Y}^\top\bm{1} - \bm{Y}\bm{1} \geq \bm{z} \right\} \leq x_0 \\
    \Longleftrightarrow\;\; & \max_{\bm{z} \in \mathcal{U}_r} \; \min_{\bm{Y}\geq\bm{0},\bm{w}\geq\bm{0}} \; \max_{\bm\rho \geq \bm{0}} \left\{ \left\langle \bm{T}, \bm{Y} \right\rangle + \bm\ell^\top \bm{w} + \bm\rho^\top \left( 
    \bm{z} - \bm{x} - \bm{w} + \bm{Y}\bm{1} - \bm{Y}^\top\bm{1} \right) \right\} \leq x_0 \\
    \Longleftrightarrow\;\; & \max_{\bm{z} \in \mathcal{U}_r} \;
    \max_{\bm\rho \geq \bm{0}} \; \left\{ \bm\rho^\top (\bm{z}-\bm{x}) +  \;\min_{\bm{Y}\geq\bm{0},\bm{w}\geq\bm{0}} \;  \left\{ \left\langle \bm{Y}, \bm{T} + \bm\rho\bm{1}^\top - \bm{1}\bm\rho^\top  \right\rangle + (\bm\ell - \bm\rho)^\top \bm{w} \right\} \right\} \leq x_0 \\
    \Longleftrightarrow\;\; & \max_{\bm{z} \in \mathcal{U}_r} \;
    \max_{\bm\rho \in \mathcal{P}} \; \left\{ \bm\rho^\top (\bm{z}-\bm{x}) \right\} \leq x_0 \\
    \Longleftrightarrow\;\; & 
    \max_{\bm\rho \in \mathcal{P}} \; \min_{\bm\beta \geq 0, \hat\beta \geq 0} \left\{ \bm\beta^\top\bar{\bm{z}} + \hat\beta r - \bm\rho^\top\bm{x} \mid \bm\beta + \hat\beta \bm{1} \geq \bm\rho \right\} \leq x_0.
\end{aligned}    
\end{equation*}
Treating $\bm\rho$ as uncertainty as well as $\bm\beta$ and $\hat\beta$ as dual recourse variables finally completes the proof.~\qed
\end{proof}

For the robust satisficing variant of the problem, we consider
\begin{equation}
\label{opt:lotsizing_satisficing_primal}
\begin{aligned}
	&\text{minimize} && k \\
	&\text{subject to} && \bm{x} + \bm{Y}(\bm{z})^\top \bm{1} - \bm{Y}(\bm{z}) \bm{1} + \bm{w}(\bm{z})\geq \bm{z} && \forall \bm{z} \in \mathcal{Z} \\
	& && \bm{c}^\top\bm{x} + \left\langle \bm{T}, \bm{Y}(\bm{z}) \right\rangle + \bm{\ell}^\top \bm{w}(\bm{z}) \leq \tau + k\Vert \bm{z} \Vert_1 && \forall \bm{z} \in \mathcal{Z} \\
	& && \bm{Y}(\bm{z}) \geq \bm{0}, \; \bm{w}(\bm{z}) \geq \bm{0} && \forall \bm{z} \in \mathcal{Z} \\ 
	& && \bm{x} \in \mathcal{X}, \; k \in \mathbb{R}_+, \; \bm{Y} \in \mathcal{R}^{n, n\times n}, \; \bm{w} \in \mathcal{R}^{n,n},
\end{aligned}
\end{equation}
where $\tau \geq Z_0=0$ is the prescribed target objective. The dualized formulation~\eqref{opt:nonneg_orthantcone_satisficing} corresponding to this problem is explicitly given below.
\begin{proposition}
\label{prop:lotsizing_satisficing_dual}
Problem~\eqref{opt:lotsizing_satisficing_primal} is equivalent to
\begin{equation}
\label{opt:lotsizing_satisficing_dual}
\begin{aligned}
    &\rm{minimize} && k \\
    &\rm{subject to} && (\bm{c}-\bm\rho)^\top\bm{x} + \bm\beta(\bm\rho)^\top \bar{\bm{z}}  \leq \tau && \forall \bm\rho \in \mathcal{P} \\
    & && \bm\beta(\bm\rho) + k\bm{1}\geq \bm\rho  && \forall \bm\rho \in \mathcal{P} \\
    & && \bm\beta(\bm\rho) \geq \bm{0} && \forall \bm\rho \in \mathcal{P} \\
    & && \bm{x} \in \mathcal{X}, \; k \in \mathbb{R}_+, \; \bm\beta \in \mathcal{R}^{n,n},
\end{aligned}
\end{equation}
%where the dual uncertainty set $\mathcal{P}$ is defined as in Proposition~\ref{prop:lotsizing_robust_dual}.
where the dual uncertainty set $\mathcal{P}$ is given by $\{ \bm\rho \in \mathbb{R}^n_+: \bm\rho \leq \bm\ell, \; \bm\rho\bm{1}^\top - \bm{1}\bm\rho^\top \leq \bm{T} \}$.
\end{proposition}
\begin{proof}{Proof.}
The proof widely parallels to that of~Proposition~\ref{prop:lotsizing_robust_dual} and is thus omitted. \qed
\end{proof}

Due to the linearity of the lot-sizing objective and constraint functions, we can solve Problem~\eqref{opt:lotsizing_robust_primal} approximately using an affine adaptation on either the primal or the dual recourse. These approximations turn out to be equivalent; see Proposition~\ref{prop:lotsizing_robust} below and also \cite{Bertsimas2016} who first establish a similar equivalence for a variant of the lot-sizing problem without emergency orders. For the robust satisficing problem~\eqref{opt:lotsizing_satisficing_primal}, \cite{long2022robust} explore how it could be solved using a primal affine adaptation, whereas here we demonstrate how the same problem could be solved using a dual affine adaptation. Perhaps unsurprisingly, we leverage Theorems~1 and~2 of \cite{Bertsimas2016} to establish in Proposition~\ref{prop:lotsizing_satisficing} that these two approximations are also equivalent. Though as earlier shown, we emphasize that for a broader class of two-stage robust satisficing linear optimization problems with complete recourse, the dual affine adaptation can provide a \emph{strictly} less conservative approximation compared to the primal affine adaptation.

\begin{proposition}
\label{prop:lotsizing_robust}
Problems~\eqref{opt:lotsizing_robust_primal} and~\eqref{opt:lotsizing_robust_dual} attain the same objective value when solved approximately using affine recourse adaptation. 
\end{proposition}
\begin{proof}{Proof.}
In their Theorem~1, \cite{Bertsimas2016} show an equivalent reformulation of Problem~\eqref{opt:lotsizing_robust_primal} which is
\begin{equation}
\label{opt:lotsizing_robust_Frans}
\begin{aligned}
	&\text{minimize} && \bm{c}^\top \bm{x} + x_0 \\
	&\text{subject to}  && \bm\beta(\bm\rho,\hat\rho)^\top \bar{\bm{z}} + \hat\beta (\bm\rho,\hat\rho) r \leq \bm\rho^\top \bm{x} + \hat\rho x_0 && \forall (\bm{\rho},\hat\rho) \in \hat{\mathcal{P}} \\
	&&& \bm\beta(\bm\rho,\hat\rho) + \hat\beta(\bm\rho,\hat\rho) \bm{1} \geq \bm\rho && \forall (\bm{\rho},\hat\rho) \in \hat{\mathcal{P}} \\
	& && \bm\beta(\bm\rho,\hat\rho) \geq \bm{0}, \; \hat\beta(\bm\rho,\hat\rho) \geq 0 && \forall (\bm{\rho},\hat\rho) \in \hat{\mathcal{P}} \\
	& && \bm{x} \in \mathcal{X}, \; x_0 \in \mathbb{R}, \; \bm\beta \in \mathcal{R}^{n+1,n}, \; \hat\beta \in \mathcal{R}^{n+1,1},
\end{aligned}
\end{equation}
where $\hat{\mathcal{P}} = \{ (\bm\rho,\hat{\rho}) \in \mathbb{R}^n_+ \times \mathbb{R}_+ \mid \bm\rho \leq \hat\rho\bm{\ell}, \; \bm\rho\bm{1}^\top - \bm{1}\bm\rho^\top \leq \hat\rho \bm{T}, \; \bm{1}^\top\bm\rho + \hat\rho = 1 \}$. They also argue in their Theorem~2 that the respective affine recourse approximations of Problems~\eqref{opt:lotsizing_robust_primal} and~\eqref{opt:lotsizing_robust_Frans} are equivalent. As a result, it suffices to show that the affine recourse approximations of Problems~\eqref{opt:lotsizing_robust_dual} and~\eqref{opt:lotsizing_robust_Frans} are equivalent. 

First, we write down the affine recourse approximation of Problem~\eqref{opt:lotsizing_robust_dual} by restricting 
$\bm\beta(\rho)$ and $\hat\beta(\rho)$ to $\bm\beta^{\rm{i}} + \bm\beta^{\rm{s}}\bm\rho$ and $\hat\beta^{\rm{i}} + \hat{\bm{\beta}}^{\rm{s}}\bm\rho$ (where `i' and `s' indicate the intercept and slope of the affine decision rules), respectively, and obtain
\begin{equation}
\label{opt:lotsizing_robust_dual_ldr}
\begin{aligned}
	&\text{minimize} && \bm{c}^\top \bm{x} + x_0 \\
	&\text{subject to}  && \bar{\bm{z}}^\top \bm\beta^{\rm{i}} + r \hat\beta^{\rm{i}} + \big( \bar{\bm{z}}^\top \bm\beta^{\rm{s}} + r\hat{\bm\beta}^{\rm{s}} \big)\bm\rho \leq \bm\rho^\top \bm{x} + x_0 && \forall \bm{\rho} \in \mathcal{P} \\
	&&& \bm\beta^{\rm{i}} + \bm{1}\hat{\beta}^{\rm{i}}  + \big( \bm\beta^{\rm{s}} + \bm{1} \hat{\bm\beta}^{\rm{s}}  \big)\bm\rho \geq \bm\rho && \forall \bm{\rho} \in \mathcal{P} \\
	& && \bm\beta^{\rm{i}} + \bm\beta^{\rm{s}}\bm\rho \geq \bm{0}, \quad \hat\beta^{\rm{i}} + \hat{\bm{\beta}}^{\rm{s}}\bm\rho \geq 0 && \forall \bm\rho \in \mathcal{P} \\ 
	& && \bm{x} \in \mathcal{X}, \; x_0 \in \mathbb{R}, \; \bm\beta^{\rm{i}} \in \mathbb{R}^n, \; \bm\beta^{\rm{s}} \in \mathbb{R}^{n \times n}, \; \hat\beta^{\rm{i}} \in \mathbb{R}, \; \hat{\bm\beta}^{\rm{s}} \in \mathbb{R}^{1 \times n}.
\end{aligned}
\end{equation}
Next, for Problem~\eqref{opt:lotsizing_robust_Frans}, we first observe that the uncertainty set $\hat{\mathcal{P}}$ requires $\hat{\rho}$ to be linearly dependent on $\bm\rho$. Hence, we can simply ignore the additional uncertain parameter $\hat{\rho}$ and work with the projection of $\hat{\mathcal{P}}$ on $\bm\rho$. By a slight abuse of notation, we will denote this projection by $\hat{\mathcal{P}}$ and note that
\begin{equation*}
	\hat{\mathcal{P}} = \left\{ \bm\rho \in \mathbb{R}^n_+ \mid \bm\rho \leq \big( 1 - \bm{1}^\top \bm\rho \big)\bm{\ell}, \; \bm\rho\bm{1}^\top - \bm{1}\bm\rho^\top \leq \big( 1 - \bm{1}^\top \bm\rho \big) \bm{T} \right\}. 
\end{equation*}
Note that as $\hat{\mathcal{P}} \subset \mathbb{R}^n_+$ and as $\bm\ell \geq \bm{0}$, it is a necessity that $\bm{1}^\top\bm\rho < 1$ for all $\bm\rho \in \hat{\mathcal{P}}$. We are now ready to present the explicit affine recourse approximation of Problem~\eqref{opt:lotsizing_robust_Frans}, which is
\begin{equation}
\label{opt:lotsizing_robust_Frans_ldr}
\begin{aligned}
	&\text{minimize} && \bm{c}^\top \bm{x} + x_0 \\
	&\text{subject to}  && \bar{\bm{z}}^\top \bm\beta^{\rm{i}} + r \hat\beta^{\rm{i}} + \big( \bar{\bm{z}}^\top \bm\beta^{\rm{s}} + r\hat{\bm\beta}^{\rm{s}} \big)\bm\rho \leq \bm\rho^\top \bm{x} + \big( 1 - \bm{1}^\top\bm\rho \big)x_0 && \forall \bm{\rho} \in \hat{\mathcal{P}} \\
	&&& \bm\beta^{\rm{i}} + \bm{1}\hat{\beta}^{\rm{i}}  + \big( \bm\beta^{\rm{s}} + \bm{1} \hat{\bm\beta}^{\rm{s}}  \big)\bm\rho \geq \bm\rho && \forall \bm{\rho} \in \hat{\mathcal{P}} \\
	& && \bm\beta^{\rm{i}} + \bm\beta^{\rm{s}}\bm\rho \geq \bm{0}, \quad \hat\beta^{\rm{i}} + \hat{\bm{\beta}}^{\rm{s}}\bm\rho \geq 0 && \forall \bm\rho \in \hat{\mathcal{P}} \\
	& && \bm{x} \in \mathcal{X}, \; x_0 \in \mathbb{R}, \; \bm\beta^{\rm{i}} \in \mathbb{R}^n, \; \bm\beta^{\rm{s}} \in \mathbb{R}^{n \times n}, \; \hat\beta^{\rm{i}} \in \mathbb{R}, \; \hat{\bm\beta}^{\rm{s}} \in \mathbb{R}^{1 \times n}.
\end{aligned}
\end{equation}

It remains to show that Problems~\eqref{opt:lotsizing_robust_dual_ldr} and~\eqref{opt:lotsizing_robust_Frans_ldr} are equivalent. First, we will show that Problem~\eqref{opt:lotsizing_robust_Frans_ldr} is a relaxation of Problem~\eqref{opt:lotsizing_robust_dual_ldr}. To this end, for any feasible solution $\bm{X} = (\bm{x},x_0,\bm\beta^{\rm{i}},\bm\beta^{\rm{s}},\hat\beta^{\rm{i}}, \hat{\bm\beta}^{\rm{s}})$ of Problem~\eqref{opt:lotsizing_robust_dual_ldr}, we will show that $\bm{X}' = (\bm{x},x_0,\bm\beta^{\rm{i}},\bm\beta^{\rm{s}} - \bm\beta^{\rm{i}}\bm{1}^\top,\hat\beta^{\rm{i}}, \hat{\bm\beta}^{\rm{s}} - \hat\beta^{\rm{i}}\bm{1}^\top)$ is feasible in Problem~\eqref{opt:lotsizing_robust_Frans_ldr}. Note that, as both solutions share the same~$\bm{x}$ and the same $x_0$, they attain the same objective value in their respective problem.

For any $\bm\rho \in \hat{\mathcal{P}}$, it is readily seen that $\bm\rho/(1 - \bm{1}^\top \bm\rho) \in \mathcal{P}$. As a result, the feasibility of~$\bm{X}$ in view of Problem~\eqref{opt:lotsizing_robust_dual_ldr} implies, for all $\bm\rho \in \hat{\mathcal{P}}$, that 
\begin{equation*}
\begin{aligned}
	&\big( 1 - \bm{1}^\top \bm\rho \big) \big( \bar{\bm{z}}^\top \bm\beta^{\rm{i}}  + r \hat\beta^{\rm{i}} \big) + \big( \bar{\bm{z}}^\top \bm\beta^{\rm{s}} + \hat{z}\hat{\bm\beta}^{\rm{s}} \big)\bm\rho \leq \bm\rho^\top \bm{x} + \big( 1 - \bm{1}^\top\bm\rho \big)x_0, \\
	&\big( 1 - \bm{1}^\top \bm\rho \big) \big( \bm\beta^{\rm{i}} + \bm{1}\hat{\beta}^{\rm{i}} \big)  + \big( \bm\beta^{\rm{s}} + \bm{1} \hat{\bm\beta}^{\rm{s}}  \big)\bm\rho \geq \bm\rho, \\
	&\big( 1 - \bm{1}^\top \bm\rho \big) \bm\beta^{\rm{i}} + \bm\beta^{\rm{s}}\bm\rho \geq \bm{0}, \\ 
	&\big( 1 - \bm{1}^\top \bm\rho \big) \hat\beta^{\rm{i}} + \hat{\bm{\beta}}^{\rm{s}}\bm\rho \geq 0 .
\end{aligned} 
\end{equation*}
Rearranging terms in the above four inequalities yields
\begin{equation*}
\begin{aligned}
	&\bar{\bm{z}}^\top \bm\beta^{\rm{i}} + r \hat\beta^{\rm{i}} + \left( \bar{\bm{z}}^\top \big( \bm\beta^{\rm{s}} - \bm\beta^{\rm{i}}\bm{1}^\top \big) + r \big( \hat{\bm\beta}^{\rm{s}} -  \hat{\beta}^{\rm{i}}\bm{1}^\top \big) \right)\bm\rho \leq \bm\rho^\top \bm{x} + \big( 1 - \bm{1}^\top\bm\rho \big)x_0, \\
	&\bm\beta^{\rm{i}} + \bm{1}\hat{\beta}^{\rm{i}}  + \left( \big( \bm\beta^{\rm{s}} - \bm\beta^{\rm{i}}\bm{1}^\top \big)  + \bm{1} \big( \hat{\bm\beta}^{\rm{s}}  - \hat{\beta}^{\rm{i}}\bm{1}^\top \big) \right)\bm\rho \geq \bm\rho, \\
	&\bm\beta^{\rm{i}} + \big( \bm\beta^{\rm{s}} - \bm\beta^{\rm{i}}\bm{1}^\top \big) \bm\rho \geq \bm{0}, \\
	&\hat\beta^{\rm{i}} + \big( \hat{\bm{\beta}}^{\rm{s}} - \hat{\beta}^{\rm{i}}\bm{1}^\top \big)\bm\rho \geq 0, 
\end{aligned} 
\end{equation*}
for all $\bm\rho \in \hat{\mathcal{P}}$, which in turn implies that $\bm{X}'$ is indeed feasible in Problem~\eqref{opt:lotsizing_robust_Frans_ldr}.

Conversely, for any feasible solution $\bm{X} = (\bm{x},x_0,\bm\beta^{\rm{i}},\bm\beta^{\rm{s}},\hat\beta^{\rm{i}}, \hat{\bm\beta}^{\rm{s}})$ of Problem~\eqref{opt:lotsizing_robust_Frans_ldr}, one can similarly show that $\bm{X}' = (\bm{x},x_0,\bm\beta^{\rm{i}},\bm\beta^{\rm{s}} + \bm\beta^{\rm{i}}\bm{1}^\top,\hat\beta^{\rm{i}}, \hat{\bm\beta}^{\rm{s}} + \hat\beta^{\rm{i}}\bm{1}^\top)$ is feasible in Problem~\eqref{opt:lotsizing_robust_dual_ldr} to conclude that Problem~\eqref{opt:lotsizing_robust_dual_ldr} is a relaxation of Problem~\eqref{opt:lotsizing_robust_Frans_ldr}. To see this, observe that $\bm\rho / (1 + \bm{1}^\top \bm\rho) \in \hat{\mathcal{P}}$ for any $\bm\rho \in \mathcal{P}$. As a result, the feasibility of~$\bm{X}$ in view of Problem~\eqref{opt:lotsizing_robust_Frans_ldr} implies, for all $\bm\rho \in \mathcal{P}$, that 
\begin{equation*}
\begin{aligned}
	&\big( 1 + \bm{1}^\top \bm\rho \big) \big( \bar{\bm{z}}^\top \bm\beta^{\rm{i}}  + r \hat\beta^{\rm{i}} \big) + \big( \bar{\bm{z}}^\top \bm\beta^{\rm{s}} + r\hat{\bm\beta}^{\rm{s}} \big)\bm\rho \leq \bm\rho^\top \bm{x} + x_0, \\
	&\big( 1 + \bm{1}^\top \bm\rho \big) \big( \bm\beta^{\rm{i}} + \bm{1}\hat{\beta}^{\rm{i}} \big)  + \big( \bm\beta^{\rm{s}} + \bm{1} \hat{\bm\beta}^{\rm{s}}  \big)\bm\rho \geq \bm\rho, \\
	&\big( 1 + \bm{1}^\top \bm\rho \big) \bm\beta^{\rm{i}} + \bm\beta^{\rm{s}}\bm\rho \geq \bm{0}, \\ 
	&\big( 1 + \bm{1}^\top \bm\rho \big) \hat\beta^{\rm{i}} + \hat{\bm{\beta}}^{\rm{s}}\bm\rho \geq 0 .
\end{aligned} 
\end{equation*}
Rearranging the terms in the above inequalities shows that $\bm{X}'$ is indeed feasible in Problem~\eqref{opt:lotsizing_robust_dual_ldr} as desired. Therefore, the optimal objective value of Problem~\eqref{opt:lotsizing_robust_dual_ldr} constitutes both a lower and an upper bound of that of Problem~\eqref{opt:lotsizing_robust_Frans_ldr}. The proof is hence completed. \qed 
\end{proof}

\begin{proposition}
\label{prop:lotsizing_satisficing}
Problems~\eqref{opt:lotsizing_satisficing_primal} and~\eqref{opt:lotsizing_satisficing_dual} attain the same objective value when solved approximately using affine recourse adaptation.
\end{proposition}
\begin{proof}{Proof.}
First, we can use \citet[Theorem~1]{Bertsimas2016} to argue~that Problem \eqref{opt:lotsizing_satisficing_primal} is equivalent to
\begin{equation*}
\begin{aligned}
	&\text{minimize} && k \\
	&\text{subject to} &&  \left(\hat\rho\bm{c}-\bm\rho\right)^\top \bm{x} + \bar{\bm{z}}^\top\bm\beta(\bm\rho,\hat\rho)   \leq  \hat\rho \tau && \forall (\bm{\rho},\hat\rho) \in \hat{\mathcal{P}} \\
	&&&\bm\beta(\bm\rho,\hat\rho) + k\hat\rho\bm{1} \geq \bm\rho && \forall (\bm{\rho},\hat\rho) \in \hat{\mathcal{P}} \\
	& && \bm\beta(\bm\rho,\hat\rho) \geq \bm{0} && \forall (\bm{\rho},\hat\rho) \in \hat{\mathcal{P}} \\ 
	& && \bm{x} \in \mathcal{X}, \; k \in \mathbb{R}_+, \; \bm\beta \in \mathcal{R}^{n+1,n}, 
\end{aligned}
\end{equation*}
where $\hat{\mathcal{P}}$ is the same as that in the proof of Proposition~\ref{prop:lotsizing_robust}. %\textcolor{blue}{Observe that the above formulation does not involve the dual recourse function $\hat{\beta}$ unlike \eqref{opt:lotsizing_robust_Frans} due to the absence of the extra constraint $\bm{1}^\top \bm{z} \leq r$ in the robustness model}. 
The proof largely follows that of  Proposition~\ref{prop:lotsizing_robust}, \emph{i.e.}, one can show that the affine adaptation approximation of the above problem is equivalent to that of~\eqref{opt:lotsizing_satisficing_dual}, which in turn implies that the affine adaptation approximations of~\eqref{opt:lotsizing_satisficing_primal} and~\eqref{opt:lotsizing_satisficing_dual} are equivalent \cite[Theorem~2]{Bertsimas2016}. Details are omitted for brevity. \qed
\end{proof}

Our goal here is to show that the robust satisficing approach (especially when it is approximately solved using the proposed affine dual adaptation from Section~\ref{sec:tractable_conic_formulations}) can be a relatively computationally friendly alternative to the classical robust optimization approach. We consider a network of a varying number of nodes $n \in \{ 10, 15, \hdots, 100 \}$ with $\mathcal{X} = \mathcal{Z} = [0,20]^n$. To ensure suitable variability among the nodes, the components of the initial ordering cost $\bm{c}$ and emergency cost $\bm{\ell}$ are generated uniformly from $[8,10]$ and $[18,20]$ respectively. We then select the node locations randomly from $[0,10]^2$ and accordingly compute the (Euclidean) distance matrix $\bm{T}$. Figure~\ref{fig:lotsizing} compares the  computational times required by Gurobi \& RSOME in logarithmic scale of the primal robust problem~\eqref{opt:lotsizing_robust_primal}, dual robust problem~\eqref{opt:lotsizing_robust_dual}, primal robust satisficing problem~\eqref{opt:lotsizing_satisficing_primal} and dual robust satisficing problem~\eqref{opt:lotsizing_satisficing_dual}. Using Gurobi 9.1.1 with RSOME (Robust Stochastic Optimization Made Easy) modeling language \citep{chen2020rsome} and Python $3.7.7$, the primal models~\eqref{opt:lotsizing_robust_primal} and~\eqref{opt:lotsizing_satisficing_primal} cannot be solved within a time limit of one hour (on an Intel Core i7 2.7GHz MacBook with 16GB of RAM) when the number of nodes exceed 80. When $n=100$, the dual robust model~\eqref{opt:lotsizing_robust_dual} takes about 15 minutes, whereas the dual robust satisficing model~\eqref{opt:lotsizing_satisficing_dual} takes only about 30 seconds. The better efficiency of the dual approaches was first hypothesized and observed in~\cite{Bertsimas2016}, whereas the better efficiency of the robust satisficing approach is, to our knowledge, established here for the first time.
\begin{figure}[h]
    \centering 
    \includegraphics[scale=0.65]{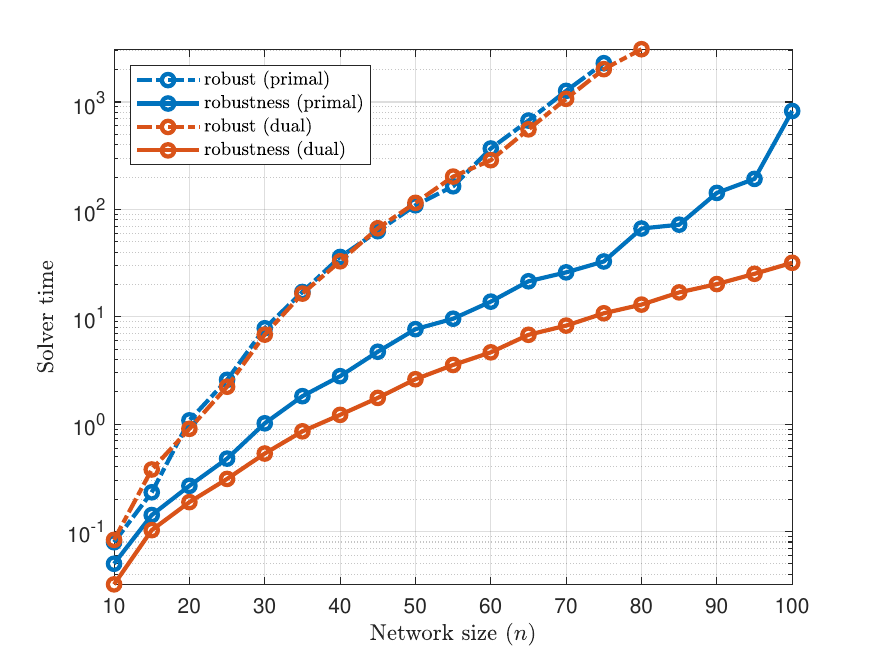}
    \caption{Computational times (in seconds) of the robust optimization (\texttt{rob opt}) and the robust satisficing (\texttt{rob sat}) lot-sizing solutions.}
    \label{fig:lotsizing}
\end{figure}

\end{APPENDICES}

\end{document}